\documentclass[a4paper]{article}
\usepackage{amssymb}
\usepackage{amsmath}
\usepackage{amsfonts}
\usepackage{mitpress}
\usepackage{geometry}
\usepackage{graphicx}
\usepackage{epstopdf}

\newtheorem{theorem}{Theorem}

\newtheorem{corollary}[theorem]{Corollary}

\newtheorem{definition}[theorem]{Definition}

\newtheorem{lemma}[theorem]{Lemma}

\newtheorem{proposition}[theorem]{Proposition}

\newtheorem{remark}[theorem]{Remark}

\newenvironment{proof}[1][Proof]{\noindent\textbf{#1.} }{\ \rule{0.5em}{0.5em}}

\newdimen\dummy
\dummy=\oddsidemargin
\input{tcilatex}
\begin{document}

\title{Marginal density expansions for diffusions and stochastic volatility,
part I: Theoretical Foundations}
\author{J.D. Deuschel, P.K. Friz, A. Jacquier, S. Violante \\
TU Berlin, TU and WIAS\ Berlin, TU\ Berlin, Imperial College}
\maketitle

\begin{abstract}
Density expansions for hypoelliptic diffusions $\left( X^{1},\dots
,X^{d}\right) $ are revisited. In particular, we are interested in density
expansions of the projection $\left( X_{T}^{1},\dots ,X_{T}^{l}\right) $, at
time $T>0$, with $l\leq d$. Global conditions are found which replace the
well-known "not-in-cutlocus" condition known from heat-kernel asymptotics.
Our small noise expansion allows for a "second order" exponential factor.
As application, new light is shed on the Takanobu--Watanabe expansion
of Brownian motion and L\'evy's stochastic area. Further applications 
include tail and implied volatility asymptotics in some stochastic volatility 
models, discussed in the compagnion paper \cite{DFJVpartII}.

\end{abstract}

\section{Introduction}

Given a multi-dimensional hypoelliptic diffusion process $\mathrm{X}%
_{t}=\left( X_{t}^{1},\dots ,X_{t}^{d}:t\geq 0\right) $, started at $\mathrm{%
X}_{0}=\mathrm{x}_{0}$, we are interested in the behaviour of the
probability density function $f=f\left( \mathrm{y}\,,t\right) $ of the
projected (in general non-Markovian) process%
\begin{equation*}
\mathrm{Y}_{t}:=\Pi _{l}\circ \mathrm{X}_{t}:=\left( X_{t}^{1},\dots
,X_{t}^{l}\right) ,\,\,\,1\leq l\leq d\text{.}
\end{equation*}%
Both short time asymptotics and tail asymptotics, in presence of some
scaling, can be derived from the small noise problem
\begin{equation*}
d\mathrm{X}_{t}^{\varepsilon }=b\left( \varepsilon ,\mathrm{X}%
_{t}^{\varepsilon }\right) dt+\varepsilon \sigma \left( \mathrm{X}%
_{t}^{\varepsilon }\right) dW_{t},\quad \text{with }\mathrm{X}%
_{0}^{\varepsilon }=\mathrm{x}_{0}^{\varepsilon }\in \mathbb{R}^{d}.
\end{equation*}%
Our main technical result, based on the Laplace method on Wiener space
following Azencott, Bismut and, in particular, Ben Arous \cite{Benarous,
Benarous2} is a density expansion for $\mathrm{Y}_{t}^{\varepsilon }:=\Pi
_{l}\circ \mathrm{X}_{t}^{\varepsilon }$ of the form, for $\mathrm{x}_{0},%
\mathrm{y},T$ fixed,
\begin{equation}
\text{ }f^{\varepsilon }\left( \mathrm{y},T\right) =e^{-c_{1}/\varepsilon
^{2}}e^{c_{2}/\varepsilon }\varepsilon ^{-l}\left( c_{0}+O\left( \varepsilon
\right) \right) \text{ as }\varepsilon \downarrow 0.  \label{feps_intronew}
\end{equation}%
Leaving definitions and precise statements to the main text below (cf
theorem \ref{thm:MainThm}) let us briefly mention our key assumptions\newline
(i) a strong H\"{o}rmander condition at all points (or in fact, a weak H\"{o}%
rmander condition at $\mathrm{x}_{0}$ and an explicit controllability
condition);\newline
(ii) existence of at most finitely many minimizers in the control problem
which govern the leading order behaviour;\newline
(iii)\ invertibility of the deterministic Malliavin covariance matrix at the
minimizers;\newline
(iv) a global condition on $\mathrm{x}_{0}\in \mathbb{R}^{d},\,\mathrm{y}\in
\mathbb{R}^{l}$ which we call \textit{non-focality,} motivated from
terminology in Riemannian geometry.

Conditions (i)-(iii) will not surprise the reader familiar with the works
\cite{Benarous, Benarous2, BAL91, TW}. However, condition (iv)\footnote{%
More precisely, $\mathrm{x}_{0}\in \mathbb{R}^{d}$ must not be focal for the
submanifold $N_{y}:=$ $\left( \mathrm{y,\cdot }\right) \subset \mathbb{R}%
^{d} $. The classical example here is of course $\left( 0,0\right) \in
\mathbb{R}^{2}$ which is focal for the unit circle $S^{1}\subset \mathbb{R}%
^{2}$.} which guarantees \textit{non-degeneracy} of the minimizers (cf.
proposition \ref{BA1_18}), appears to be new in the context of density
expansions, to the best of our knowledge, even in the elliptic case. It
forms the essence of what is needed to extend the well-known \textit{%
point-point} concept of \textit{non-conjugacy }(crucial part of the "$\notin
$\textit{\ cut-locus condition" }familiar from heat kernel expansions) to a
(sub-Riemannian, with drift) \textit{point-subspace} setting. A simple
(elliptic) example where (iv) \textit{and} (\ref{feps_intronew}) fails, is given in
section \ref{EEwD}. A similar situation arises in the (hypoelliptic) example
of Brownian motion and L\'evy area, see section \ref{LevyAreaSection},
where we recover (and then extend) some expansions previously 
derived by Takanobu--Watanabe \cite{TW}.  We emphasize that our applications,
notably those discussed in \cite{DFJVpartII},
require us to introduce and characterize non-focality in a control-theoretic
generality; cf. section \ref{SectionNonFocality}. (The reader may still be
interested to consult geometry text books such as \cite{BiGr, Sa} or \cite[%
Section 4, p. 227-229]{doC} for more information on focality in the
Riemannian setting.)

As far as the expansion (\ref{feps_intronew}) is concerned, we draw
attention to the (in the context of density expansions) somewhat unusual
\textit{second order exponential factor} present when $c_{2}\neq 0$. As was
understood in the context of the general Laplace method on Wiener space,
\cite{Azencott, Az2, Benarous2}, this has to do with allowing the drift
vector field $b$ (and in the present paper also: the starting point) to
depend on $\varepsilon $ at first order; the special case that arises from
considering short time asymptotics - the small noise parameter $\varepsilon $
is then introduced by Brownian scaling - always leads to $c_{2}=0$. It is
interesting to note that the work of Kusuoka--Stroock \cite{KS}, concerning
precise asymptotics for Wiener functionals (in the small noise limit), see
also \cite{O, KO} for recent applications to projected diffusions, was set
up as an expansion in $\varepsilon ^{2}$. This is enough to cover the model
case of short time expansions, but cannot yield an expansion of the type (%
\ref{feps_intronew}) with $c_{2}\neq 0$. A similar remark applies to the
small noise expansions for projected diffusions due to Takanobu--Watanabe
\cite{TW}.

Density expansions of diffusions in the small noise regime seem to go back
(at least) to \cite{Ki}; density expansions for projected diffusions in the
small noise regime (which include the short time regime), with applications
to implied volatility expansions, were recently considered by Y. Osajima
\cite{O}, based on work with S. Kusuoka \cite{KO}. We partially improve on
these results. First, as was already mentioned, $c_{2}=0$ in these works
whereas expansions with $c_{2}\neq 0$ are crucial in understanding the tail
behaviour of certain stochastic volatility models, \cite{DFJVpartII} (see
also \cite{GuSt, FGGS}). Additionally, in comparison with \cite{O} we do not
assume $x_{0}$ near $\left( y,\cdot \right) $, nor ellipticity of the
problem.\ In further contrast to (the general results in) \cite{KO, KS} we
provide a checkable, finite-dimensional criterion that guarantees that the
crucial infinite-dimensional non-degeneracy assumption, left as such in \cite%
{KO, KS}, is actually satisfied. On the other hand, these authors give
somewhat explicit formulae for $c_{0}$ which we (presently) do not.

Finally, our expansion (\ref{feps_intronew}) leads to short time expansion
for projected diffusion densities, under global conditions on $\left(
\mathrm{x}_{0},\mathrm{y}\right) $, of the form
\begin{equation}
f\left( \mathrm{y},t\right) \sim c_{0}\left( \mathrm{x}_{0},\mathrm{y}%
\right) \frac{1}{t^{l/2}}\exp \left( -\frac{d^{2}\left( \mathrm{x}_{0},%
\mathrm{y}\right) }{2t}\right) \text{ as }t\downarrow 0.
\label{fshort_intronew}
\end{equation}%
When $l=d$, such expansions go back to classical works ranging from
Molchanov \cite{Mo75} to Ben Arous \cite{Benarous}. The leading order
behaviour $2t\log f\left( \mathrm{y},t\right) \sim -d^{2}\left( \mathrm{x}%
_{0},\mathrm{y}\right) $ is due to Varadhan \cite{Va67}. The case $l<d$, in
particular our global condition on $\left( \mathrm{x}_{0},\mathrm{y}\right) $%
, appears to be new. That said, expansions of this form have appeared in
\cite{TW, HLW, O}; the last two references aimed at implied volatility
expansions. In the context of time homogenous local volatility models ($%
l=d=1 $), the expansion (\ref{fshort_intronew}) holds trivially without any
conditions on $\left( \mathrm{x}_{0},\mathrm{y}\right) $; the resulting
expansion was derived (with explicit constant $c_{0}$) in \cite{GaEtAl}.
Subject to mild technical conditions on the diffusion coefficient, they show
how to deduce first a call price and then an implied volatility expansion in
the short time (to maturity) regime%
\begin{equation*}
\sigma _{BS}\left( k,t\right) =|k|/d\left( \mathrm{x}_{0},k\right) +c\left(
\mathrm{x}_{0},k\right) t+O\left( t^{2}\right) \text{ as }t\downarrow 0;
\end{equation*}%
where $d\left( \mathrm{x}_{0},k\right) $ is a point-point distance and $%
c\left( \mathrm{x}_{0},k\right) $ is explicitly given; $k$ is $\log $-strike
(better: log-forward in moneyness). The celebrated
Berestycki--Busca--Florent (BBF) formula \cite{BBF} asserts that $\sigma
_{BS}\left( k,t\right) \sim |k|/d\left( \mathrm{x}_{0},k\right) $ as $%
t\downarrow 0$, and is in fact valid in generic stochastic volatility
models, $d\left( \mathrm{x}_{0},k\right) $ is then understood as a
point-hyperplane distance. In fact, $|k|/d\left( \mathrm{x}_{0},k\right) $
arose as the initial condition of a non-linear evolution equation for the
entire implied volatility surface. As briefly indicated in \cite[Sec 6.3]%
{BBF} this can be used for a Taylor expansion of $\sigma _{BS}\left(
k,t\right) $ in $t$. Such expansions have also been discussed, based on heat
kernel expansions on Riemannian manifolds by \cite{BC, HL, Pau}, not always
in full mathematical rigor. Some mathematical results are given in \cite{O},
assuming ellipticity and \textit{close-to-the-moneyness} $\left\vert
k\right\vert <<1$; see also forthcoming work by Ben Arous--Laurence \cite%
{BenArousLaurence}. We suspect that our formula (\ref{fshort_intronew}),
potentially applicable far-from-the-money, will prove useful in this context
and shall return to this in future work.

It should be noted that the BBF formula alone can be obtained from soft
large deviation arguments, cf. \cite[Sec. 3.2.1]{Ph} and the references
therein. In a similar spirit, cf. \cite[Sec 5, Rmk 2.9]{VLec}, the Varadhan-type formula $2t\log f\left(
\mathrm{y},t\right) \sim -d^{2}\left( \mathrm{x}_{0},\mathrm{y}\right) $,
when $l<d$, can be shown, without any conditions on $(\mathrm{x}_{0},\mathrm{%
y})$ by large deviation methods, only relying on the existence of a
reasonable density.  \footnote{Care is necessary, however, if one returns to the hypoellipitic 
small noise with drift setting; even when $l=d$, the Varadhan type formula, 
$ \varepsilon^2 \log f^{\varepsilon }\left( \mathrm{y},T\right) \sim -c_1 \text{ as } \varepsilon \to 0$,  
in the notation of (\ref{feps_intronew}), may fail to hold true, see \cite{BALII}.}

As a final note, we recall that the (in general, non-Markovian) $\mathbb{R}%
^{l}$-valued It\^{o}-process $\left( \mathrm{Y}_{t}:t\geq 0\right) $ admits
- subject to some technical assumptions \cite{Gy, Pi} - a \textit{Markovian
(or Gy\"{o}ngy) projection}. That is, a time-inhomogeneous Markov diffusion $%
(\mathrm{\tilde{Y}}_{t}:t\geq 0)$ with matching time-marginals i.e $\mathrm{Y%
}_{t}=\mathrm{\tilde{Y}}_{t}$ \ (in law) for every fixed $t\geq 0$. In a
financial context, when $l=1$, this process is known as the (Dupire) local
volatility model and various authors \cite{BBF, BC, HL, BenArousLaurence}
have used this as an important intermediate step in computing implied
volatility in stochastic volatility models. Since all our expansions (small
noise, tail, short time ) are relative to such time-marginals they may also
be viewed as expansions for the corresponding Markovian projections.

\bigskip

\ \textbf{Acknowledgement:} JDD and AJ acknowledge (partial resp. full)
financial support from MATHEON. PKF\ acknowledges partial support from
MATHEON\ and the European Research Council under the European Union's
Seventh Framework Programme (FP7/2007-2013) / ERC grant agreement nr.
258237. PKF would like to thank G. Ben Arous for pointing out conceptual
similarities in \cite{FGGS, Benarous} and several discussions thereafter. It
is also a pleasure to thank F. Baudoin, J.P. Gauthier, A. Gulisashvili and
P. Laurence for their interest and feedback.

\section{The main result and its corollaries}

\bigskip Consider a $d$-dimensional diffusion $\left( \mathrm{X}%
_{t}^{\varepsilon }\right) _{t\geq 0}$ given by the stochastic differential
equation
\begin{equation}
d\mathrm{X}_{t}^{\varepsilon }=b\left( \varepsilon ,\mathrm{X}%
_{t}^{\varepsilon }\right) dt+\varepsilon \sigma \left( \mathrm{X}%
_{t}^{\varepsilon }\right) dW_{t},\quad \text{with }\mathrm{X}%
_{0}^{\varepsilon }=\mathrm{x}_{0}^{\varepsilon }\in \mathbb{R}^{d}
\label{SDEXeps}
\end{equation}%
and where $W=(W^{1},\dots ,W^{m}$) is an $m$-dimensional Brownian motion.
Unless otherwise stated, we assume $b:[0,1)\times \mathbb{R}^{d}\rightarrow
\mathbb{R}^{d},$ $\sigma =\left( \sigma _{1},\dots ,\sigma _{m}\right) :%
\mathbb{R}^{d}\rightarrow \mathrm{Lin}\left( \mathbb{R}^{m}\rightarrow
\mathbb{R}^{d}\right) $ and $\mathrm{x}_{0}^{\cdot }:[0,1)\rightarrow
\mathbb{R}^{d}$ to be smooth, bounded with bounded derivatives of all
orders. Set $\sigma _{0}=b\left( 0,\cdot \right) $ and assume that, for
every multiindex $\alpha $, the drift vector fields $\left\{ b\left(
\varepsilon ,\cdot \right) :\varepsilon >0\right\} $ converge to $\sigma
_{0} $ in the sense\footnote{%
If (\ref{SDEXeps}) is understood in Stratonovich sense, so that $dW$ is
replaced by $\circ dW$, the drift vector field $b\left( \varepsilon ,\cdot
\right) $ is changed to $\tilde{b}\left( \varepsilon ,\cdot \right) =b\left(
\varepsilon ,\cdot \right) -\left( \varepsilon ^{2}/2\right)
\sum_{i=1}^{m}\sigma _{i}\cdot \partial \sigma _{i}$. In particular, $\sigma
_{0}$ is also the limit of $\tilde{b}\left( \varepsilon ,\cdot \right) $ in
the sense of (\ref{bepsTob}) .}%
\begin{equation}
\partial _{x}^{\alpha }b\left( \varepsilon ,\cdot \right) \rightarrow
\partial _{x}^{\alpha }b\left( 0,\cdot \right) =\partial _{x}^{\alpha
}\sigma _{0}\left( \cdot \right) \text{ uniformly on compacts as }%
\varepsilon \downarrow 0\text{.}  \label{bepsTob}
\end{equation}%
We shall also assume that%
\begin{equation}
\partial _{\varepsilon }b\left( \varepsilon ,\cdot \right) \rightarrow
\partial _{\varepsilon }b\left( 0,\cdot \right) \text{ \ uniformly on
compacts as }\varepsilon \downarrow 0  \label{bepsC1}
\end{equation}%
and (one-sided) differentiability of the starting point in $\varepsilon $,%
\begin{equation}
\mathrm{x}_{0}^{\varepsilon }=\mathrm{x}_{0}+\varepsilon \mathrm{\hat{x}}%
_{0}+o\left( \varepsilon \right) \text{ as }\varepsilon \downarrow 0\text{. }
\label{x0eps_ass}
\end{equation}

Leaving applications to stochastic volatility models to \cite{DFJVpartII},
the main result of this paper is a density expansion in $\varepsilon $ of
the $\mathbb{R}^{l}$-valued projection\footnote{%
While $\left( \mathrm{X}_{t}^{\varepsilon }:t\geq 0\right) $ is Markovian,
this will not be true, in general, for the projected process $\left( \mathrm{%
Y}_{t}^{\varepsilon }:t\geq 0\right) $; as a consequence the probability
density function of $\mathrm{Y^\varepsilon_T}$ cannot be analyzed directly via Kolmogorov's
forward PDE.}
\begin{equation*}
\mathrm{Y}_{T}^{\varepsilon }:=\Pi _{l}\circ \mathrm{X}_{T}^{\varepsilon
}:=\left( X_{T}^{\varepsilon ,1},\dots ,X_{T}^{\varepsilon ,l}\right) \in
\mathbb{R}^{l};
\end{equation*}%
\bigskip where $\Pi _{l}$ denotes the projection $\left( x^{1},\dots
,x^{d}\right) \mapsto \left( x^{1},\dots ,x^{l}\right) $, for fixed $l\in
\left\{ 1,\dots ,d\right\} $ and $T>0$. Of course, we need to guarantee that
$\mathrm{Y}_{T}^{\varepsilon }$ indeed admits a density. We make the \textit{%
standing assumption} that the \textit{weak H\"{o}rmander condition} holds at
$\mathrm{x}_{0}$,%
\begin{equation}
\text{\textrm{span}}\left[ \sigma _{i}:1\leq i\leq m;\text{ }\left[ \sigma
_{j},\sigma _{k}\right] :0\leq j,k\leq m;...\right] _{\mathrm{x}_{0}}=%
\mathcal{T}_{\mathrm{x}_{0}}\mathbb{R}^{d};  \tag{H}  \label{H}
\end{equation}%
that is, the linear span of $\sigma _{1},\dots ,\sigma _{m}$ and all Lie
brackets of $\sigma _{0},\sigma _{1},\dots ,\sigma _{m}$ at the starting
point is full. Since this condition is "open" it also holds, thanks to (\ref%
{bepsTob}), for $\varepsilon >0$ small enough, with $\sigma _{0}$ and $%
\mathrm{x}_{0}$ replaced by ${b}\left( \varepsilon ,\cdot \right) $ (or $%
\tilde{b}\left( \varepsilon ,\cdot \right) $, cf. previous footnote) and $%
\mathrm{x}_{0}^{\varepsilon }$, respectively. It then is a classical result
(due to H\"{o}rmander, Malliavin; see e.g. \cite{Nu}) that the $\mathbb{R}%
^{d}$-valued r.v. $\mathrm{X}_{T}^{\varepsilon }$ admits a (smooth) density
for all times $T>0$ and so does its $\mathbb{R}^{l}$-valued projection $%
\mathrm{Y}_{T}^{\varepsilon }$. We denote the probability density of $%
\mathrm{Y}_{T}^{\varepsilon }$ by %
$
f^{\varepsilon }\left( \cdot ,T\right) \equiv f^{\varepsilon }\left( \mathrm{%
y},T\right)$ with $\mathrm{y}\in \mathbb{R}^{l}$. In theorem \ref{thm:MainThm} below, where we expand
$$ f^{\varepsilon }\left( \mathrm{%
y},T\right)|_{\mathrm{y}=\mathrm{a}} = f^{\varepsilon }\left( \mathrm{%
a},T\right)
$$ in $\varepsilon$ for some fixed $\mathrm{a}\in \mathbb{R}^{l}$,
a crucial assumption is that $\mathcal{K}_{%
\mathrm{a}}$, defined by

\begin{equation}
\mathcal{K}_{\mathrm{x}_{0},T;\mathrm{a}}=\mathcal{K}_{\mathrm{a}}:=\left\{
\mathrm{h}\in H:\Pi _{l}\circ \phi _{T}\left( h\right) =\mathrm{a}\right\},
\end{equation}%
is non-empty. Here, $H$ denotes the Cameron-Martin space, i.e. absolutely continuous paths
with derivative in $L^{2}\left( \left[ 0,T\right] ,\mathbb{R}^{m}\right) $,
with norm given by
\begin{equation*}
\left\Vert \mathrm{h}\right\Vert _{H}^{2}=\sum_{i=1}^{m}\int_{0}^{T}|\dot{h}%
_{t}^{i}|^{2}dt,\,\,\,\,\,\,\,\forall \mathrm{h}=\left( h^{1},\dots
,h^{m}\right) \in H.
\end{equation*}%
We write $\phi _{T}\left( \mathrm{h}\right) $ for the time-$T$ solution to
the controlled ordinary differential equation (e.g. \cite[Sec.3]{FV})%
\begin{equation}
d\phi _{t}^{\mathrm{h}}=\sigma _{0}\left( \phi _{t}^{\mathrm{h}}\right)
dt+\sum_{i=1}^{m}\sigma _{i}\left( \phi _{t}^{\mathrm{h}}\right)
dh_{t}^{i},\,\,\phi _{0}^{\mathrm{h}}=\mathrm{x}_{0}\in \mathbb{R}^{d}.
\label{dphih}
\end{equation}%
At occasions, we emphasize the starting point $\mathrm{x}_{0}$ by writing $%
\phi _{T}^{\mathrm{h}}\left( \mathrm{x}_{0}\right) $, and note that $\phi
_{T}\left( \mathrm{h}\right) =\phi _{T}^{\mathrm{h}}\left( \mathrm{\cdot }%
\right) $, i.e. the mappig $\mathrm{x}_{0}\in \mathbb{R}^{d}\mapsto $ $\phi
_{T}^{\mathrm{h}}\left( \mathrm{x}_{0}\right) \in \mathbb{R}^{d}$, is a
diffeomorphism. Similarly, we write $\phi _{T\leftarrow t}\left( \mathrm{h}%
\right) $ for the time-$T$ solution, but started at time $t$. Each such $%
\phi _{T\leftarrow t}\left( \mathrm{h}\right) $ is again a diffeomorphism,
and we denote its differential by $\Phi _{T\leftarrow t}\left( \mathrm{h}%
\right) $. A well-known sufficient condition for $\mathcal{K}_{\mathrm{a}%
}\neq \varnothing $ \ is the \textit{strong H\"{o}rmander condition} (at all
points)
\begin{equation}
\forall \mathrm{x}\in \mathbb{R}^{d}:\text{\textrm{Lie}}\left[ \sigma
_{1},\dots ,\sigma _{m}\right] |_{\mathrm{x}}=\mathcal{T}_{\mathrm{x}}%
\mathbb{R}^{d}\cong \mathbb{R}^{d};  \tag{H1}  \label{H1}
\end{equation}%
see \cite[p.106]{Ju2} or \cite[p.441]{IW},\cite{Ku} for instance.\footnote{%
A weak H\"{o}rmander type condition which ensures $\mathcal{K}_{\mathrm{a}%
}\neq \varnothing $ is found in \cite{Ju}.}

\bigskip Whenever $\mathcal{K}_{\mathrm{a}}\neq \varnothing $, it makes
sense to define the \textit{energy} and the set of \textit{minimizers}
\begin{eqnarray}
\Lambda _{\mathrm{x}_{0},T}\left( \mathrm{a}\right) &:&=\Lambda \left(
\mathrm{a}\right) :=\inf \left\{ \frac{1}{2}\Vert \mathrm{h}\Vert _{H}^{2}:%
\mathrm{h}\in \mathcal{K}_{\mathrm{a}}\right\} ,  \label{Lambdax0y} \\
\mathcal{K}_{\mathrm{a}}^{\min } &:&=\left\{ \mathrm{h}_{0}\in \mathcal{K}_{%
\mathrm{a}}:\frac{1}{2}\Vert \mathrm{h}_{0}\Vert _{H}^{2}=\Lambda \left(
\mathrm{a}\right) \right\} .  \notag
\end{eqnarray}

In words, $\Lambda \left( \mathrm{a}\right) $ is the minimal energy required
to go in time $T$ from $\mathrm{x}_{0}\in \mathbb{R}^{d}$ to the "target"
submanifold%
\begin{equation*}
N:=N_{\mathrm{a}}:=\left\{ \mathrm{x}\in \mathbb{R}^{d}:\Pi _{l}\left(
\mathrm{x}\right) =\left( x^{1},\dots ,x^{l}\right) =\mathrm{a}\right\} .
\end{equation*}%
Elements of $\mathcal{K}_{\mathrm{a}}^{\min }$ will be called \textit{%
minimizers} or \textit{minimizing controls}. A standard weak-compactness
argument (e.g. \cite[Thm 1.14]{Bismut}) shows that $\mathcal{K}_{\mathrm{a}%
}\neq \varnothing $ already implies that $\mathcal{K}_{\mathrm{a}}^{\min }$
is non-empty. (Throughout the paper, we shall only be concerned with the
situation that $\mathcal{K}_{\mathrm{a}}^{\min }$ contains one or finitely
many minimizers.)

It will be crucial that $\mathcal{K}_{\mathrm{a}}$ enjoys a (Hilbert)
manifold structure, locally around (each) $\mathrm{h}_{0}\in \mathcal{K}_{%
\mathrm{a}}^{\min }$ . Following Bismut \cite[Thm 1.5]{Bismut} this can be
guaranteed by assuming invertibility of $C_{\mathrm{x}_{0},T;\mathrm{a}%
}\left( \mathrm{h}_{0}\right) $, the deterministic Malliavin matrix given by%
\begin{equation*}
C_{\mathrm{x}_{0},T;\mathrm{a}}\left( \mathrm{h}\right) :=C\left( \mathrm{h}%
\right) :=\left\langle D\phi _{T}\left( \mathrm{h}\right) ,D\phi _{T}\left(
\mathrm{h}\right) \right\rangle _{H}\in \mathrm{Lin}\left( \mathcal{T}_{%
\mathrm{x}_{T}}^{\ast }\mathbb{R}^{d}\rightarrow \mathcal{T}_{\mathrm{x}_{T}}%
\mathbb{R}^{d}\right) \cong \mathbb{R}^{d\times d}
\end{equation*}%
where $\mathrm{x}_{T}:=\phi _{T}\left( \mathrm{h}\right) $ and $D$ will
always denote the ($H$-valued) Fr\'{e}chet derivative of some function
depending on $\mathrm{h}\in H$. We can also view $C\left( \mathrm{h}\right) $
as (positive semi-definite) quadratic form on $\mathcal{T}_{\mathrm{x}%
_{T}}^{\ast }\mathbb{R}^{d}$, in coordinates\footnote{%
Einstein's summation convention is used whenever convenient.}, with $\mathrm{%
p}=p_{i}dx^{i}\in \mathcal{T}_{\mathrm{x}_{T}}^{\ast }\mathbb{R}^{d}$,%
\begin{equation*}
\left\langle C\left( \mathrm{h}\right) \mathrm{p},\mathrm{p}\right\rangle
=\left\Vert \left\langle \mathrm{p},D\phi _{T}\left( \mathrm{h}\right)
\right\rangle \right\Vert _{H}^{2}=\left\Vert \sum_{i=1}^{d}p_{i}D\phi
_{T}^{i}\left( \mathrm{h}\right) \right\Vert _{H}^{2}\text{ . }
\end{equation*}%
In fact, large parts of our analysis only rely on non-degeneracy of $C\left(
\mathrm{h}_{0}\right) $ restricted to $\mathbb{R}^{l\times l}$ but we find
it more convenient to deal with the "full" matrix $C\left( \mathrm{h}%
_{0}\right) $. A sufficient condition for "$C\left( \mathrm{h}\right) $ is invertible for
every $\mathrm{h}\neq 0$" in a strictly sub-elliptic setting is given as
condition \textrm{(H2)}\ by \cite{Bismut}; although much stronger than H\"{o}%
rmander's condition, it does apply to examples such as the $3$-dimensional
Heisenberg group and thus L\'evy's area, cf. section \ref{LevyAreaSection}.
Most financial applications, as discussed in \cite{DFJVpartII}, are actually
locally elliptic, or "almost elliptic", and are  covered by the following condition.

\begin{proposition}
Assume $\mathrm{h}\in H$ and the condition%
\begin{equation*}
\exists t\in \left[ 0,T\right] :\mathrm{span}\left[ \sigma _{1},\dots
,\sigma _{m}\right] |_{\mathrm{x}_{t}}=\mathcal{T}_{\mathrm{x}_{t}}\mathbb{R}%
^{d}
\end{equation*}%
where $\mathrm{x}_{t}:=\phi _{t}\left( \mathrm{h}\right) $. Then $C\left(
\mathrm{h}\right) $ is invertible.
\end{proposition}

\begin{proof}
We have the well-known formula (e.g. \cite[(4.1)]{TW}), for any $\,\mathrm{k}%
=\left( k^{1},\dots ,k^{m}\right) \in H,$
\begin{equation*}
\left\langle D\phi _{T}\left( \mathrm{h}\right) ,\mathrm{k}\right\rangle
_{H}=D\phi _{T}\left( \mathrm{h}\right) \left[ \mathrm{k}\right]
=\int_{0}^{T}\sum_{j=1}^{m}\Phi _{T\leftarrow t}\left( \mathrm{h}\right)
\sigma _{j}\left( \mathrm{x}_{t}\right) \dot{k}_{t}^{j}dt\in \mathcal{T}_{%
\mathrm{x}_{T}}\mathbb{R}^{d}\,\,
\end{equation*}%
When pairing this with $\mathrm{p}=p_{i}dx^{i}\in \mathcal{T}_{\mathrm{x}%
_{T}}^{\ast }\mathbb{R}^{d}$, we have%
\begin{equation*}
\left\langle \left\langle \mathrm{p},D\phi _{T}\left( \mathrm{h}\right)
\right\rangle ,\mathrm{k}\right\rangle
_{H}=\int_{0}^{T}\sum_{j=1}^{m}\left\langle \mathrm{p},\Phi _{T\leftarrow
t}\left( \mathrm{h}\right) \sigma _{j}\left( \mathrm{x}_{t}\right)
\right\rangle \dot{k}_{t}^{j}dt\in \mathcal{T}_{\mathrm{x}_{T}}\mathbb{R}^{d}
\end{equation*}%
and it easily follows that $\left\langle C\left( \mathrm{h}\right) \mathrm{p}%
,\mathrm{p}\right\rangle $ is given by
\begin{equation*}
\left\Vert \left\langle \mathrm{p},D\phi _{T}\left( \mathrm{h}\right)
\right\rangle \right\Vert _{H}^{2}=\int_{0}^{T}\sum_{j=1}^{m}\left\langle
\mathrm{p},\Phi _{T\leftarrow t}\left( \mathrm{h}\right) \sigma _{j}\left(
\mathrm{x}_{t}\right) \right\rangle
^{2}dt=\int_{0}^{T}\sum_{j=1}^{m}\left\langle \left( \Phi _{T\leftarrow
t}\left( \mathrm{h}\right) \right) ^{\ast }\mathrm{p},\sigma _{j}\left(
\mathrm{x}_{t}\right) \right\rangle ^{2}dt.
\end{equation*}%
Assume now $\left\langle C\left( \mathrm{h}\right) \mathrm{p},\mathrm{p}%
\right\rangle =0$. By assumption \textrm{span}$\left[ \sigma _{1},\dots
,\sigma _{m}\right] |_{\mathrm{x}_{t}}=\mathcal{T}_{\mathrm{x}_{t}}\mathbb{R}%
^{d}$ for some $t\in \left[ 0,T\right] $, and this clearly remains valid in
a small enough open interval containing $t$ which is enough to conclude $%
\left( \Phi _{T\leftarrow t}\left( \mathrm{h}\right) \right) ^{\ast }\mathrm{%
p}\equiv 0$. By non-degeneracy of the (co-)tangent flow, this implies $%
\mathrm{p}=0$ and so $C\left( \mathrm{h}\right) $ is non-degenerate, as
claimed.
\end{proof}

We now introduce the \textit{Hamiltonian}%
\begin{eqnarray*}
\mathcal{H}\left( \mathrm{x},\mathrm{p}\right) &:&=\left\langle \mathrm{p}%
,\sigma _{0}\left( \mathrm{x}\right) \right\rangle +\frac{1}{2}%
\sum_{i=1}^{m}\left\langle \mathrm{p},\sigma _{i}\left( \mathrm{x}\right)
\right\rangle ^{2} \\
&=&\left\langle \mathrm{p},\sigma _{0}\left( \mathrm{x}\right) \right\rangle
+\frac{1}{2}\left\langle \mathrm{p},\left( \sigma \sigma ^{T}\right) \left(
\mathrm{x}\right) \mathrm{p}\right\rangle
\end{eqnarray*}%
and \textrm{H}$_{t\leftarrow 0}=$\textrm{H}$_{t\leftarrow 0}\left( \mathrm{x}%
_{0},\mathrm{p}_{0}\right) $ as the flow associated to the vector field $%
\left( \partial _{\mathrm{p}}\mathcal{H},-\partial _{\mathrm{x}}\mathcal{H}%
\right) $ on $\mathcal{T}^{\ast }\mathbb{R}^{d}$. (As in \cite[p.37]{Bismut}%
, this vector\ field is complete, i.e. \textrm{H}$_{\cdot \leftarrow 0}$
does not explode.)

The following propositions generalize the respective results in Bismut's
book \cite{Bismut} (see also Ben Arous \cite[Theorems 1.15 and 1.1.8]%
{Benarous}) from a drift-free ($\sigma _{0}\equiv 0$), point-to-point
setting ($\mathrm{x}_{0}\in \mathbb{R}^{d}$ to $\mathrm{y}\in \mathbb{R}^{d}$%
) to a point-to-subspace setting ($\mathrm{x}_{0}\in \mathbb{R}^{d}$ to $%
\left( \mathrm{y},\cdot \right) \in \mathbb{R}^{l}\oplus \mathbb{R}^{d-l}$)
with non-zero drift vector field $\sigma _{0}$. Note that the Bismut setting \cite[%
Chapter I]{Bismut} is recovered by taking zero drift, $\sigma _{0}\equiv 0$,
and $l=d$.

\begin{proposition}
\label{BA1_15}\bigskip If (i) $\mathrm{h}_{0}\in \mathcal{K}_{\mathrm{a}%
}^{\min }$ is a minimizing control and (ii) the deterministic Malliavin
covariance matrix $C\left( \mathrm{h}_{0}\right) $ is invertible then there
exists a unique $\mathrm{p}_{0}=\mathrm{p}_{0}\left( \mathrm{h}_{0}\right)
\in \mathcal{T}_{\mathrm{x}_{0}}^{\ast }\mathbb{R}^{d}$, in fact\footnote{%
The (global) coordinate chart $\left( x^{1},\dots ,x^{d}\right) $ of $%
\mathbb{R}^{d}$ induces coordinates co-vectors fields (or one-forms) $\left(
dx^{1},\dots ,dx^{d}\right) $.}%
\begin{equation*}
\mathrm{p}_{0}\in \left( \Phi _{T\leftarrow 0}\left( \mathrm{h}_{0}\right)
\right) ^{\ast }\mathrm{span}\left\{ dx^{1},\dots ,dx^{l}\right\} |_{\phi
_{T}^{h}\left( \mathrm{x}_{0}\right) },
\end{equation*}%
such that%
\begin{equation}
\phi _{t}^{\mathrm{h}_{0}}\left( \mathrm{x}_{0}\right) =\pi \mathrm{H}%
_{t\leftarrow 0}\left( \mathrm{x}_{0}\mathrm{,p}_{0}\right) ,\text{ }0\leq
t\leq T  \label{projection_bi_char}
\end{equation}%
($\pi $ denotes the projection from $\mathcal{T}^{\ast }\mathbb{R}^{d}$ onto
$\mathbb{R}^{d}$; in coordinates $\pi \left( \mathrm{x},\mathrm{p}\right) =%
\mathrm{x}$)\newline
Moreover, $\left( \mathrm{x}\left( t\right) ,\mathrm{p}\left( t\right)
\right) :=\mathrm{H}_{t\leftarrow 0}\left( \mathrm{x}_{0},\mathrm{p}%
_{0}\right) $ solves the \underline{Hamiltonian ODEs} in $\mathcal{T}^{\ast }%
\mathbb{R}^{d}\cong \mathbb{R}^{d}\oplus \mathbb{R}^{d}$
\begin{equation}
\left(
\begin{array}{c}
\mathrm{\dot{x}} \\
\mathrm{\dot{p}}%
\end{array}%
\right) =\left(
\begin{array}{c}
\partial _{\mathrm{p}}\mathcal{H}\left( \mathrm{x}\left( t\right) ,\mathrm{p}%
\left( t\right) \right) \\
-\partial _{\mathrm{x}}\mathcal{H}\left( \mathrm{x}\left( t\right) ,\mathrm{p%
}\left( t\right) \right)%
\end{array}%
\right) ,  \label{HamiltonianODEs}
\end{equation}%
the minimizing control $\mathrm{h}_{0}=\mathrm{h}_{0}\left( \cdot \right) $
is recovered by%
\begin{equation}
\mathrm{\dot{h}}_{0}=\left(
\begin{array}{c}
\,\left\langle \sigma _{1}\left( \mathrm{x}\left( \cdot \right) \right) ,%
\mathrm{p}\left( \cdot \right) \right\rangle \\
\dots \\
\,\,\left\langle \sigma _{m}\left( \mathrm{x}\left( \cdot \right) \right) ,%
\mathrm{p}\left( \cdot \right) \right\rangle%
\end{array}%
\right) .  \label{Formula_h0}
\end{equation}%
At last, crucial for actual computations, $\left( \mathrm{x}\left( t\right) ,%
\mathrm{p}\left( t\right) \right) =\mathrm{H}_{t\leftarrow 0}\left( \mathrm{x%
}_{0},\mathrm{p}_{0}\right) $ satisfies the Hamiltonian ODEs (\ref%
{HamiltonianODEs}) as \underline{boundary value problem}, subject to the
following initial -, terminal - and transversality conditions,%
\begin{eqnarray}
\mathrm{x}\left( 0\right) &=&\mathrm{x}_{0}\in \mathbb{R}^{d},  \notag \\
\mathrm{x}\left( T\right) &=&\left( \mathrm{a},\cdot \right) \in \mathbb{R}%
^{l}\mathbb{\oplus R}^{d-l}\cong \mathbb{R}^{d},  \notag \\
\mathrm{p}\left( T\right) &=&\left( \cdot ,0\right) \in \mathbb{R}^{l}%
\mathbb{\oplus R}^{d-l}\cong \mathcal{T}_{\mathrm{x}\left( T\right) }^{\ast }%
\mathbb{R}^{d}.  \label{ITTcond}
\end{eqnarray}
\end{proposition}

\begin{remark}
With $C:=\mathcal{H}\left( \mathrm{x}\left( t\right) ,\mathrm{p}\left(
t\right) \right) $ independent of $t\in \left[ 0,T\right] $, one has%
\begin{equation}
\Lambda \left( \mathrm{a}\right) =\frac{1}{2}\Vert \mathrm{h}_{0}\Vert
_{H}^{2}=TC-\frac{1}{2}\int_{0}^{T}\left\langle \sigma _{0}\left( \mathrm{x}%
\left( t\right) \right) ,\mathrm{p}\left( t\right) \right\rangle dt\text{.}
\label{EnergyFormula}
\end{equation}
\end{remark}

\begin{proof}
The key remark, due to Bismut \cite[Chapter I]{Bismut}, is that under the
assumption "$\exists C\left( \mathrm{h}_{0}\right) ^{-1}$" the set $\mathcal{%
K}_{\mathrm{a}}^{\min }$ can be described by Hamilton--Jacobi theory. It
then suffices to adapt the arguments of Bismut, as done in the drift-free
case by Takanobu--Watanabe, \cite[Prop. 4.1]{TW}. Let us note that the
additional drift vector field $\sigma _{0}$ is trivially incorporated in
their setting, cf. the evolution given in (\ref{dphih}), by adding a $0$th
component to the controls, i.e. $h_{t}^{0}\equiv t$. The boundary conditions
- in particular, transversality, have not been pointed out explicitly in
\cite{TW} although are implicitly contained in their formulation. In fact,
formal application of Pontryagin's maximum principle leads precisely to the
above boundary value problem; care is necessary, however, since without
assuming invertibility of $C\left( \mathrm{h}_{0}\right) $, one can be in
the so-called "strictly abnormal" case; the above approach is then not
possible.
\end{proof}

\begin{remark}
\label{RmkPartialEnergyEqualPT}Assume that $\mathrm{y}\mapsto \mathrm{h}%
_{0}\left( \mathrm{y}\right) \in \mathcal{K}_{\mathrm{y}}^{\min }$ is a
smooth map (in a neighbourhood of the fixed point $\mathrm{a}$). Writing $'$ for the derivative with respect to ${\mathrm{y}}$,
we have%
\begin{equation*}
\Lambda \left( \mathrm{y}\right) =\frac{1}{2}\left\Vert \mathrm{h}_{0}\left(
\mathrm{y}\right) \right\Vert ^{2}\implies \Lambda'
\left( \mathrm{a}\right) =\left\langle \mathrm{h}_{0}\left( \mathrm{a}%
\right) ,\mathrm{h}'_{0}\left( \mathrm{a}\right)
\right\rangle _{H}.
\end{equation*}%
On the other hand, it follows from $\Pi _{l}\phi _{T}\left( \mathrm{h}%
_{0}\left( \mathrm{y}\right) \right) =\mathrm{y}$ that\footnote{$\left( \Pi
_{l}\right) _{\ast }:T_{\mathrm{x}}\mathbb{R}^{d}\rightarrow T_{\Pi _{l}%
\mathrm{x}}\mathbb{R}^{l}$ is the differential of the projection map $\Pi
_{l}:\left( x^{1},\dots ,x^{d}\right) \rightarrow \left( x^{1},\dots
,x^{l}\right) .$
}
\begin{equation}
\left( \Pi _{l}\right) _{\ast }D\phi _{T}\left( \mathrm{h}_{0}\left( \mathrm{%
a}\right) \right) \left[ \mathrm{h}'_{0}\left( \mathrm{a%
}\right) \right] =\mathrm{Id}\text{ on }\mathbb{R}^{l}.
\label{EquUsefulRmkDerLambda}
\end{equation}%
Write $\mathrm{p}\left( T\right) =\left( \mathrm{q}\left( \mathrm{a}\right)
,0\right) \in \left( \cdot ,0\right) $ as determined by the previous
proposition, equation (\ref{ITTcond}), where our notation $\mathrm{q}=\mathrm{q}\left( \mathrm{%
a}\right) $ emphasizes the dependency on $\mathrm{a}$, with $T$ fixed. One
has (cf. lemma \ref{LemmaFormOfh})%
\begin{equation*}
\mathrm{h}_{0}\left( \mathrm{a}\right) =D\phi _{T}\left( \mathrm{h}%
_{0}\left( \mathrm{a}\right) \right) ^{\ast }\mathrm{p}\left( T\right)
=\left( \left( \Pi _{l}\right) _{\ast }D\phi _{T}\left( \mathrm{h}_{0}\left(
\mathrm{a}\right) \right) \right) ^{\ast }\mathrm{q}\left( \mathrm{a}\right)
\end{equation*}%
and it follows that%
\begin{equation*}
\Lambda'
\left( \mathrm{a}\right) =\left\langle \mathrm{h}_{0}\left( \mathrm{a}%
\right) ,\mathrm{h}'_{0}\left( \mathrm{a}\right)
\right\rangle _{H}
=\left\langle
\mathrm{q} \left( \mathrm{a}\right) ,\left( \Pi _{l}\right) _{\ast }D\phi _{T}\left(
\mathrm{h}_{0}\left( \mathrm{y}\right) \right) \mathrm{%
h}'_{0}\left( \mathrm{a}\right) \right\rangle .
\end{equation*}%
Thanks to (\ref{EquUsefulRmkDerLambda}) we now see that%
\begin{equation}
 \Lambda'
\left( \mathrm{a}\right) = \mathrm{q}\left( \mathrm{a}\right). \label{dEnergyIsPT}
\end{equation}%
This can be a useful short-cut when computing the energy from the
Hamiltonian system. If $\#\mathcal{K}_{\mathrm{a}}^{\min }=1$ for some $%
\mathrm{a}$, and our non-degeneracy condition \textrm{(ND)} as introduced
below is met, the existence of such a map $\mathrm{h}_{0}\left( \mathrm{%
\cdot }\right) $ can be shown along the lines of \cite[Thm 1.26]{Bismut}. We
shall not rely on formula (\ref{dEnergyIsPT}) in the sequel.
\end{remark}

\begin{remark}[How to compute optimal controls $h_{0}$]
\label{HowToComputeH0}Proposition \ref{BA1_15} - as it stands - requires $%
\mathrm{h}_{0}$ to be a minimizer and then, subject to condition (ii),
provides us with some information about $\phi ^{\mathrm{h}_{0}}\left(
\mathrm{x}_{0}\right) $ and in particular allows us to reconstruct $\mathrm{h%
}_{0}$ from the Hamiltonian flow%
\begin{equation*}
\left( \mathrm{x,p}\right) \equiv H_{\cdot \leftarrow 0}\left( \mathrm{x}%
_{0},\mathrm{p}_{0}\right) =H_{\cdot \leftarrow T}\left( \mathrm{x}_{T},%
\mathrm{p}_{T}\right) ,
\end{equation*}%
cf. equation (\ref{Formula_h0}). That said, we can consider any solution to
the boundary valued problem (\ref{HamiltonianODEs}),(\ref{ITTcond}), say $%
\left( \mathrm{\hat{x}},\mathrm{\hat{p}}\right) $, and define a (possibly
non-minimizing) control path \textrm{\^{h}}$_{0}$ via (\ref{Formula_h0}) i.e.%
\begin{equation*}
\mathrm{\hat{h}}_{0}^{i}=\int_{0}^{\cdot }\,\left\langle \sigma _{i}\left(
\mathrm{\hat{x}}\left( t\right) \right) ,\mathrm{\hat{p}}\left( t\right)
\right\rangle dt,\,\,\,\,i=1,\dots ,m.
\end{equation*}%
From (\ref{HamiltonianODEs}),
\begin{equation*}
d\mathrm{\hat{x}}_{t}=\partial _{\mathrm{p}}\mathcal{H}\left( \mathrm{\hat{x}%
}_{t},\mathrm{\hat{p}}_{t}\right) dt=\sigma _{0}\left( \mathrm{\hat{x}}%
_{t}\right) dt+\sum_{i=1}^{m}\sigma _{i}\left( \mathrm{\hat{x}}_{t}\right)
\left\langle \sigma _{i}\left( \mathrm{\hat{x}}_{t}\right) ,\mathrm{\hat{p}}%
_{t}\right\rangle dt
\end{equation*}%
and so relation (\ref{projection_bi_char}) remains valid i.e. $\phi _{t}^{%
\mathrm{\hat{h}}_{0}}\left( \mathrm{x}_{0}\right) =\mathrm{\hat{x}}_{t}$. It
follows that the boundary conditions valid for \textrm{\^{x}} (namely,
\textrm{\^{x}}$\left( 0\right) =\mathrm{x}_{0},\Pi _{l}$\textrm{\^{x}}$%
\left( T\right) =\mathrm{a}$) are also valid for $\phi ^{\mathrm{\hat{h}}%
_{0}}\left( \mathrm{x}_{0}\right) $ and hence $\mathrm{\hat{h}}_{0}$ $\in
\mathcal{K}_{\mathrm{a}}$. While we do not know if $\mathrm{\hat{h}}_{0}\in
\mathcal{K}_{\mathrm{a}}^{\min }$, proposition \ref{BA1_15} guarantees that
every minimizer $\mathrm{h}_{0}$ $\in \mathcal{K}_{\mathrm{a}}^{\min }$ can
be found be the above procedure. We thus have the following \textbf{recipe}:%
\newline
\textbf{(i)} Argue a priori that $C\left( \mathrm{h}_{0}\right) $ is
invertible (or ignore and check in the end).\newline
\textbf{(ii)} Solve Hamiltonian ODEs as boundary value problem, cf. (\ref%
{HamiltonianODEs}),(\ref{ITTcond}). Characterize all solutions via the
(non-empty!) set%
\begin{equation*}
\left\{ \mathrm{\hat{p}}_{0}:\mathrm{H}_{t\leftarrow 0}\left( \mathrm{x}_{0};%
\mathrm{\hat{p}}_{0}\right) \equiv \left( \mathrm{\hat{x}}_{t},\mathrm{\hat{p%
}}_{t}\right) \text{ satisfies (\ref{HamiltonianODEs}),(\ref{ITTcond}) }%
\right\} ;
\end{equation*}%
equivalently, characterize all solutions by $\left( \text{\textrm{\^{z}}}%
_{T},\mathrm{\hat{q}}_{T}\right) $ where $\left( \mathrm{\hat{x}}_{t},%
\mathrm{\hat{p}}_{t}\right) =\mathrm{H}_{t\leftarrow T}\left( \left( \mathrm{%
a},\mathrm{\hat{z}}_{T}\right) ;\left( \mathrm{\hat{q}}_{T},0\right) \right)
$.\newline
\textbf{(iii)} For each such solution $\left( \mathrm{\hat{x}}_{t},\mathrm{%
\hat{p}}_{t}\right) $, compute $\Vert \mathrm{\hat{h}}_{0}\Vert _{H}^{2}$
where $\mathrm{\hat{h}}_{0}$ is given by
\begin{equation*}
\mathrm{\hat{h}}_{0}^{i}=\int_{0}^{\cdot }\,\left\langle \sigma _{i}\left(
\mathrm{\hat{x}}_{t}\right) ,\mathrm{\hat{p}}_{t}\right\rangle
dt,\,\,\,\,\,i=1,\dots ,m.
\end{equation*}%
\textbf{(iv)} The minimizing $\mathrm{h}_{0}$ are precisely those elements
in $\{\mathrm{\hat{h}}_{0}:$ as constructed in (ii),(iii)$\}$ which minimize
the energy $\Vert \mathrm{\hat{h}}_{0}\Vert _{H}^{2}$. In particular then,%
\begin{equation*}
\Lambda \left( \mathrm{a}\right) =\frac{1}{2}\Vert \mathrm{h}_{0}\Vert
_{H}^{2}.
\end{equation*}
\end{remark}

The following proposition is crucial.

\begin{proposition}
\label{BA1_18}Under the assumptions of the proposition \ref{BA1_15}, in
particular $\mathrm{h}_{0}\in \mathcal{K}_{\mathrm{a}}^{\min }$ with
associated $\mathrm{p}_{0}=\mathrm{p}_{0}\left( \mathrm{h}_{0}\right) \in
\mathcal{T}_{\mathrm{x}_{0}}^{\ast }\mathbb{R}^{d}$, the following are
equivalent:\newline
(iii) $\mathrm{h}_{0}\in \mathcal{K}_{\mathrm{a}}$ is a non-degenerate
minimum of the energy $I:=\frac{1}{2}\Vert \cdot \Vert _{H}^{2}|_{\mathcal{K}%
_{\mathrm{a}}}$ in the sense that%
\begin{equation*}
I^{\prime \prime }\left( \mathrm{h}_{0}\right) \left[ \mathrm{k},\mathrm{k}%
\right] >0\text{ \ \ \ }\forall \,0\neq \mathrm{k}\in \mathcal{T}_{\mathrm{h}%
_{0}}\mathcal{K}_{\mathrm{a}}\subset H;
\end{equation*}%
(iii') $\mathrm{x}_{0}$ is \textbf{non-focal} for $N=\left( \mathrm{a},\cdot
\right) $ along $\mathrm{h}_{0}$ in the sense that, with $\left( \mathrm{x}%
_{T}\mathrm{,p}_{T}\right) :=\mathrm{H}_{T\leftarrow 0}\left( \mathrm{x}_{0}%
\mathrm{,p}_{0}\left( \mathrm{h}_{0}\right) \right) \in \mathcal{T}^{\ast }%
\mathbb{R}^{d}$,
\begin{equation*}
\partial _{\left( \mathfrak{z},\mathfrak{q}\right) }|_{\left( \mathfrak{z},%
\mathfrak{q}\right) \mathfrak{=}\left( 0,0\right) }\pi \mathrm{H}%
_{0\leftarrow T}\left( \mathrm{x}_{T}+\left(
\begin{array}{c}
0 \\
\mathfrak{z}%
\end{array}%
\right) ,\mathrm{p}_{T}+\left( \mathfrak{q},0\right) \right)
\end{equation*}%
is non-degenerate (as $d\times d$ matrix; here we think of $\left( \mathfrak{%
z},\mathfrak{q}\right) \in \mathbb{R}^{d-l}\times \mathbb{R}^{l}\cong
\mathbb{R}^{d}$ and recall that $\pi $ denotes the projection from $\mathcal{%
T}^{\ast }\mathbb{R}^{d}$ onto $\mathbb{R}^{d}$; in coordinates $\pi \left(
\mathrm{x},\mathrm{p}\right) =\mathrm{x}$).\newline
\end{proposition}

\begin{proof}
Let us give a quick proof of (iii')$\implies $(iii) in the Riemannian
setting, the general (sub-Riemannian, with drift) case is new and full proof
is given in the next section. Since $\mathrm{h}_{0}\in \mathcal{K}_{\mathrm{a%
}}^{\min }$ we know that $I^{\prime \prime }\left( \mathrm{h}_{0}\right) $
must be positive semi-definite. In particular, the index of $\mathrm{h}_{0}$%
, relative to the point-submanifold problem $\mathrm{x}_{0}\times N$, is
zero. By the \textit{Morse index theorem} \cite{Sa, BiGr}, there cannot be
any focal point along the $\left( \mathrm{x}_{0}\times N\right) $-geodesic%
\begin{equation*}
\left\{ \pi \mathrm{H}_{t\leftarrow T}\left( \mathrm{x}_{T},\mathrm{p}%
_{T}\right) :t\in (0,T]\right\} .
\end{equation*}%
Condition (iii') guarantees that this extends to $t=0$, i.e. there is no
focal point along%
\begin{equation*}
\left\{ \pi \mathrm{H}_{t\leftarrow T}\left( \mathrm{x}_{T},\mathrm{p}%
_{T}\right) :t\in \left[ 0,T\right] \right\} .
\end{equation*}%
We can then use \cite[lemma 2.9 (b)]{Sa} to conclude that $I^{\prime \prime
}\left( \mathrm{h}_{0}\right) $ is positive definite.
\end{proof}

\begin{definition}[Condition (\textrm{ND); }generalized $\protect\notin $\
cut-locus condition]
\textit{\ \label{DefND}W}e say that $\left\{ \mathrm{x}_{0}\right\} \times
N_{\mathrm{a}}$ where $N_{\mathrm{a}}:=\left( \mathrm{a},\cdot \right)
:=\left\{ \mathrm{x}\in \mathbb{R}^{d}:\Pi _{l}\mathrm{x}=\mathrm{a}\in
\mathbb{R}^{l}\right\} $ satisfies condition \textrm{(ND)} if\newline
(i) $1\leq \#\mathcal{K}_{\mathrm{a}}^{\min }<\infty $,\newline
(ii) the deterministic Malliavin covariance matrix $C\left( \mathrm{h}%
\right) $ is invertible, $\forall \mathrm{h}\in \mathcal{K}_{\mathrm{a}%
}^{\min }$;\newline
(iii) $\mathrm{x}_{0}$ is not focal for $N_{\mathrm{a}}$ along $\mathrm{h}$,
for any $\mathrm{h}\in \mathcal{K}_{\mathrm{a}}^{\min }.$
\end{definition}

When $\sigma _{0}\equiv 0$ and $l=d$, i.e. $N=\left\{ \mathrm{a}\right\} $,
and $\#\mathcal{K}_{\mathrm{a}}^{\min }=1$, condition \textrm{(ND) }says
precisely that $\left( \mathrm{x}_{0},\mathrm{a}\right) $ is not contained
in the sub-Riemannian cut-locus in the sense of Ben Arous \cite{Benarous};
extending the usual Riemannian meaning. In this sense our (global) condition
\textrm{(ND) }is effectively a generalization of the well-known "$\notin $\
cut-locus" condition in the context of heat-kernel expansions.

\begin{theorem}
\label{thm:MainThm} \textbf{(Small noise)} Let $\left( \mathrm{X}%
^{\varepsilon }\right) $ be the solution process to%
\begin{equation*}
d\mathrm{X}_{t}^{\varepsilon }=b\left( \varepsilon ,\mathrm{X}%
_{t}^{\varepsilon }\right) dt+\varepsilon \sigma \left( \mathrm{X}%
_{t}^{\varepsilon }\right) dW_{t},\quad \text{with }\mathrm{X}%
_{0}^{\varepsilon }=\mathrm{x}_{0}^{\varepsilon }\in \mathbb{R}^{d}.
\end{equation*}%
Assume $b\left( \varepsilon ,\cdot \right) \rightarrow \sigma _{0}\left(
\cdot \right) $ in the sense of (\ref{bepsTob}), (\ref{bepsC1}), and $%
\mathrm{X}_{0}^{\varepsilon }\equiv \mathrm{x}_{0}^{\varepsilon }\rightarrow
\mathrm{x}_{0}$ as $\varepsilon \rightarrow 0$ in the sense of (\ref%
{x0eps_ass}). Assume the weak H\"{o}rmander condition (\ref{H}) at $\mathrm{x%
}_{0}\in \mathbb{R}^{d}$. Fix $T>0$ and also
\begin{equation*}
\mathrm{a}\in \mathbb{R}^{l}\text{ \ \ and \ \ \ }N_{\mathrm{a}}:=\left(
\mathrm{a},\cdot \right) \subset \mathbb{R}^{d}
\end{equation*}
and assume that $\left\{ \mathrm{x}_{0}\right\} \times N_{\mathrm{a}}$
satisfies \textrm{(ND), }i.e. the\textrm{\ }generalized $\notin $\ cut-locus
condition (in particular then, $\#\mathcal{K}_{\mathrm{a}}^{\min }\geq 1$).
Then the energy
\begin{equation*}
\Lambda \left( \mathrm{y}\right) =\inf \left\{ \frac{1}{2}\Vert \mathrm{h}%
\Vert _{H}^{2}:\mathrm{h}\in \mathcal{K}_{\mathrm{y}}\right\} =\frac{1}{2}%
\Vert \mathrm{h}_{0}\Vert _{H}^{2}.
\end{equation*}%
is smooth as a function of
$\mathrm{y}$ in a neighbourhood of $\mathrm{a}$ provided $\#\mathcal{K}_{%
\mathrm{a}}^{\min }=1$; otherwise i.e. when $\#\mathcal{K}_{\mathrm{a}%
}^{\min }>1$, we assume so.\footnote{%
It will not be true in general, when $\#\mathcal{K}_{\mathrm{a}}^{\min }>1$,
that $\Lambda \left( \cdot \right) $ is automatically smooth in a neighbourhood of $\mathrm{a}$.
 To wit consider, $\mathcal{K}_{\mathrm{y}}^{\min }=\{\mathrm{h}%
_{0}\left( \mathrm{y}\right) ,\mathrm{\tilde{h}}_{0}\left( \mathrm{y}\right)
\}$. Then $\Lambda \left( \mathrm{y}\right) =\min \left( \frac{1}{2}\Vert
\mathrm{h}_{0}\left( \mathrm{y}\right) \Vert _{H}^{2},\frac{1}{2}\Vert
\mathrm{\tilde{h}}_{0}\left( \mathrm{y}\right) \Vert _{H}^{2}\right) $ and
even if $\Vert \mathrm{h}_{0}\left( \mathrm{\cdot }\right) \Vert _{H}^{2}$
and $\Vert \mathrm{\tilde{h}}_{0}\left( \mathrm{\cdot }\right) \Vert
_{H}^{2} $ \ are smooth near $\mathrm{a}$, this need not be the case for the
minimum.}
Then there exists $%
c_{0}=c_{0}\left( \mathrm{x}_{0},\mathrm{a},T\right) >0$ such that%
\begin{equation*}
\mathrm{Y}_{T}^{\varepsilon }=\Pi _{l}\mathrm{X}_{T}^{\varepsilon }=\left(
X_{T}^{\varepsilon ,1},\dots ,X_{T}^{\varepsilon ,l}\right) ,\,\,\,\,1\leq
l\leq d
\end{equation*}%
admits a density with expansion (for fixed $\mathrm{x}_{0},\mathrm{a}$ and $T>0$)
\begin{equation*}
f^{\varepsilon }\left( \mathrm{a},T\right) =e^{-\frac{\Lambda \left( \mathrm{%
a}\right) }{\varepsilon ^{2}}}e^{\,\frac{\max \left\{ \Lambda ^{\prime
}\left( \mathrm{a}\right) \cdot \,\mathrm{\hat{Y}}_{T}\left( \mathrm{h}%
_{0}\right) :\mathrm{h}_{0}\in \mathcal{K}_{\mathrm{a}}^{\min }\right\} }{%
\varepsilon }}\varepsilon ^{-l}\left( c_{0}+O\left( \varepsilon \right)
\right) \text{ as }\varepsilon \downarrow 0.
\end{equation*}%
Here $\mathrm{\hat{Y}=\hat{Y}}\left( \mathrm{h}_{0}\right) =\left( \hat{Y}%
^{1},\dots ,\hat{Y}^{l}\right) $ is the projection, $\mathrm{\hat{Y}=}\Pi
_{l}\mathrm{\hat{X}}$, of the solution to the following (ordinary)
differential equation
\begin{eqnarray}
d\mathrm{\hat{X}}_{t} &=&\Big(\partial _{\mathrm{x}}b\left( 0,\phi _{t}^{%
\mathrm{h}_{0}}\left( \mathrm{x}_{0}\right) \right) +\partial _{\mathrm{x}%
}\sigma (\phi _{t}^{\mathrm{h}_{0}}\left( \mathrm{x}_{0}\right) )\mathrm{%
\dot{h}}_{0}\left( t\right) \Big)\mathrm{\hat{X}}_{t}dt+\partial
_{\varepsilon }b\left( 0,\phi _{t}^{\mathrm{h}_{0}}\left( \mathrm{x}%
_{0}\right) \right) dt,  \label{eq:SDEYHat} \\
\,\,\,\,\,\mathrm{\hat{X}}_{0} &=&\mathrm{\hat{x}}_{0}\text{ \ \ as given in
(\ref{x0eps_ass}).}  \notag
\end{eqnarray}
\end{theorem}

\begin{remark}
The assumption $\mathcal{K}_{\mathrm{a}}\neq \varnothing $, implicit through
$\#\mathcal{K}_{\mathrm{a}}^{\min }\geq 1$ in the statement of the above
theorem, is known to be necessary for the existence of a positive density;
in presence of (\ref{H}) and invertibility of $C\left( \mathrm{h}\right) $,
for some $\mathrm{h}\in \mathcal{K}_{\mathrm{a}}$, it is actually
sufficient; \cite{BAL91}. As noted earlier, the strong H\"{o}rmander
condition at all points (\ref{H1}) is sufficient for $\mathcal{K}_{\mathrm{a}%
}\neq \varnothing $; a less well-known condition of weak-H\"{o}rmander type
is given in \cite{Ju}.
\end{remark}

\begin{proof}
Assume $\#\mathcal{K}_{\mathrm{a}}^{\min }=1$ and see remark \ref%
{RmkMultipleMin} below for the reduction of $\#\mathcal{K}_{\mathrm{a}%
}^{\min }<\infty $ to this case. The basic observation is that $%
f^{\varepsilon }(\mathrm{y},T)$ is the Fourier inverse of its characteristic
function,
\begin{equation*}
\mathbb{E}\left[ \exp \left( i\mathrm{\xi }\cdot \mathrm{Y}_{T}^{\varepsilon
}\right) \right] =\mathbb{E}\left[ \exp \left( i(\mathrm{\xi },0)\cdot
\mathrm{X}_{T}^{\varepsilon }\right) \right]
\end{equation*}%
where we write $\left( \mathrm{\xi },0\right) =\left( \xi ^{1},\dots ,\xi
^{l},0,\dots ,0\right) \in \mathbb{R}^{d}$. In other words, it suffices to
\textit{restrict} the characteristic function of $\mathrm{X}%
_{T}^{\varepsilon }$, the full (Markovian) process evaluated at time $T,$ to
obtain the c.f. of $\mathrm{Y}_{T}^{\varepsilon }$. The density is then
obtained by Fourier-inversion. When $\mathrm{X}_{T}^{\varepsilon }$ is
affine the c.f. is analytically described by ODEs; (approximate) saddle
points are easy to compute and the Fourier inversion - after shifting the
contour through the saddle point - becomes a finite-dimensional Laplace
method which leads to the desired expansion of $f^{\varepsilon }(\mathrm{a}%
,T)$; in essence, this approach was carried out by Friz et al. in \cite{FGGS}%
. In our present situation, of course, $\mathrm{X}$ does not enjoy any
affine structure, but - following Ben Arous \cite{Benarous}, who considers
the "point-to-point" case $l=d$; a similar approach works and ultimately
boils down to applying the Laplace method on\ Wiener space \cite{Benarous2}.
The differences to the setting of \cite{Benarous}, aside from (i) allowing
for $l<d$, is that (ii) our drift-term does not vanish of order $\varepsilon
^{2}$ (which is typical when aiming for short time asymptotics; cf. also
proposition \ref{CorShortTime} below) and (iii) that the starting point is
allowed to depend on $\varepsilon $. In fact, (ii),(iii) are responsible for
the additional exponential $\exp \left\{ \left( ...\right) /\varepsilon
\right\} $ factor in our expansion (Such a factor was already seen in the
general context of the Laplace method on\ Wiener space \cite{Benarous2}.)
Also, (ii) implies that the limiting vector field $\sigma
_{0}=\lim_{\varepsilon \rightarrow 0}b\left( \varepsilon ,\cdot \right) $
affects the leading order behaviour in that the energy $\Lambda \left(
\mathrm{a}\right) $ has no geometric interpretation as square of some
(sub)Riemannian point-subspace distance. In particular, if we want to
implement the strategy of \cite{Benarous} we are forced to revisit the
meaning of all geometric concepts (cut-locus, geodesics, conjugate points
...) upon which the work \cite{Benarous} is based. The key observation now
is that essentially all geometric concepts channel through the
(non-geometric, but infinite-dimensional) condition (iii) of proposition \ref%
{BA1_18} into the application of Laplace's method. Now, the whole point of
proposition \ref{BA1_18} was to provide checkable conditions for $(\mathrm{x}%
_{0},\mathrm{a})$ to satisfy (iii). Having made these part of our assumptions
we are in fact ready to proceed along the lines of Ben Arous \cite{Benarous}.

Fix $\mathrm{y}=\mathrm{a}$ and note that for any $C^{\infty }$-bounded function $%
\mathrm{y}\mapsto F\left( \mathrm{y}\right) $ on $\mathbb{R}^{l}$, by
Fourier inversion,
\begin{align}
f^{\varepsilon }(\mathrm{a},T)e^{-F(\mathrm{a})/\varepsilon ^{2}}& =\frac{1}{%
\left( 2\pi \right) ^{l}}\int_{\mathbb{R}^{l}}\mathbb{E}\left[ \exp \left( i%
\mathrm{\xi }\cdot \left( \mathrm{Y}_{T}^{\varepsilon }-\mathrm{a}\right) -%
\frac{F\left( \mathrm{Y}_{T}^{\varepsilon }\right) }{\varepsilon ^{2}}%
\right) \right] d\mathrm{\xi }  \label{fespExpF} \\
& =\frac{1}{\left( 2\pi \varepsilon \right) ^{l}}\int_{\mathbb{R}^{l}}%
\mathbb{E}\left[ \exp \left( i\mathrm{\zeta }\cdot \left( \frac{\mathrm{Y}%
_{T}^{\varepsilon }-\mathrm{a}}{\varepsilon }\right) -\frac{F\left( \mathrm{Y%
}_{T}^{\varepsilon }\right) }{\varepsilon ^{2}}\right) \right] d\mathrm{%
\zeta }.  \notag \\
& =\frac{1}{\left( 2\pi \varepsilon \right) ^{l}}\int_{\mathbb{R}^{l}}%
\mathbb{E}\left[ \exp \left( i\left( \mathrm{\zeta },0\right) \cdot \left(
\frac{\mathrm{X}_{T}^{\varepsilon }-\left( \mathrm{a},0\right) }{\varepsilon
}\right) \right) e^{-\frac{F\left( \Pi _{l}\mathrm{X}_{T}^{\varepsilon
}\right) }{\varepsilon ^{2}}}\right] d\mathrm{\zeta }.
\label{ThisIntegrandForLaplaceMethod}
\end{align}%
In particular, the last integrand can be computed, as asymptotic expansion
in $\varepsilon $ for fixed $\mathrm{\zeta }$, by Laplace method in Wiener
space, cf. \cite{Benarous}, \cite{Benarous2}, based on the full (Markovian)
process $\mathrm{X}_{T}^{\varepsilon }$. We pick $F$ (for fixed $\mathrm{a}$%
) such that $F\left( \cdot \right) +\Lambda \left( \cdot \right) $ has
minimum at $\mathrm{a}$, i.e.
\begin{equation*}
F\left( \mathrm{a}\right) + \Lambda \left( \mathrm{a}\right) =\inf \left\{ F\left( \mathrm{y}\right)
+\Lambda \left( \mathrm{y}\right) :\mathrm{y}\in \mathbb{R}^{l}\right\}
\end{equation*}%
and such that this minimum is non-degenerate; a natural candidate for $%
F\left( \mathrm{y}\right) $ would then be given (at least for $\mathrm{y}$
near $\mathrm{a}$) by
\begin{eqnarray*}
\mathrm{y} &\mapsto &\lambda \left\vert \mathrm{y}-\mathrm{a}\right\vert
^{2}-\Lambda \left( \mathrm{a}\right) \text{, some }\lambda >0\text{;} \\
\text{or }\mathrm{y} &\mapsto &\lambda \left\vert \mathrm{y}-\mathrm{a}%
\right\vert ^{2}-[\Lambda \left( \mathrm{y}\right) -\Lambda \left( \mathrm{a}%
\right) ]\text{,}
\end{eqnarray*}%
since adding constants is irrelevant here (recall that $\mathrm{a}$ is kept
fixed). The trouble with the above candidate is their potential lack of
(global) smoothness; even in the classical Riemannian setting $\Lambda $ may
not be smooth at the cut-locus. On the other hand, $\Lambda \left( \cdot
\right) $ is smooth near $\mathrm{a}$ in case $\#\mathcal{K}_{\mathrm{a}%
}^{\min }=1$; this is a consequence of \cite[Thm 1.26]{Bismut}. Let us give
some detail. First note that our non-focality condition implies
non-conjugacy of $\left( \mathrm{x}_{0},\phi _{T}^{\mathrm{h}_{0}}\left(
\mathrm{x}_{0}\right) \right) $ along $\mathrm{h}_{0}$ and write $\phi _{T}^{%
\mathrm{h}_{0}}\left( \mathrm{x}_{0}\right) =\left( \mathrm{a,z}\right) $
where $\mathrm{z}=\mathrm{z}\left( \mathrm{a}\right) $, keeping $\mathrm{x}%
_{0}$ fixed. Using crucially $\#\mathcal{K}_{\mathrm{a}}^{\min }=1$, Bismut
shows that the point-point energy function, which he calls $\bar{E}$, is a
smooth function in a neighbourhood of $\left( \mathrm{a,z}\right) $. (His
proof extends without difficulty to non-zero drift, i.e. $\sigma _{0}\neq 0$
in the Hamiltonian; it suffices to use (\ref{EnergyFormula}) at the final
stage of his argument.) Noting $\Lambda \left( \mathrm{a}\right) =\bar{E}%
\left( \mathrm{a},\mathrm{z}\left( \mathrm{a}\right) \right) $ only
smoothness of $\mathrm{y}\mapsto \mathrm{z}\left( \mathrm{y}\right) $
near $\mathrm{a}$ remains to be seen. But since%
\begin{equation*}
\pi \mathrm{H}_{0\leftarrow T}\left( \left( \mathrm{y,z}\left( \mathrm{y}%
\right) \right) ;\left( \mathrm{q}\left( \mathrm{y}\right) ,0\right) \right)
=\mathrm{x}_{0}\text{ (fixed)}
\end{equation*}%
and the derivative with respect to $\mathrm{z},\mathrm{q}$ (non-focality!)
is non-degenerate, this is an immediate consequence from the implict
function theorem.
When $\#\mathcal{K}_{\mathrm{a}}^{\min }>1$, smoothness of
$\Lambda \left( \cdot \right) $ near $\mathrm{a}$ was in fact part of our
assumptions. It is thus natural to localize the above candidates around $%
\mathrm{a}$ which leads us to define $F$, at least in a neighbourhood of $%
\mathrm{a}$, by \footnote{As before, we write $'$ for the derivative with respect to $\mathrm{y}$.
The Hessian of the energy is then written as $\Lambda''$.}
\begin{equation*}
F\left( \mathrm{y}\right) =\lambda \left\vert \mathrm{y}-\mathrm{a}%
\right\vert ^{2}-\left[ \Lambda' \left(
\mathrm{a}\right) \left( \mathrm{y}-\mathrm{a}\right) +\frac{1}{2}
\Lambda'' \left( \mathrm{a}\right)
\left( \mathrm{y}-\mathrm{a},\mathrm{y}-\mathrm{a}\right) \right] ;
\end{equation*}%
a routine modification of $F$, away from $\mathrm{y}$, then guarantees $%
C^{\infty }$-boundedness of $F$. (Since $F\left( \mathrm{a}\right) =0$ with
this last choice of $F$, the l.h.s. of (\ref{fespExpF}) is actually
precisely $f^{\varepsilon }(\mathrm{a},T)$.) Non-degeneracy of the minimum $%
\mathrm{a}$ of $F$ entails that the functional $H\ni \mathrm{h}\mapsto
F\left( \Pi _{l}\circ \phi _{T}^{\mathrm{h}}\left( \mathrm{x}_{0}\right)
\right) +\frac{1}{2}\left\Vert \mathrm{h}\right\Vert _{H}^{2}$ has a
non-degenerate minimum at $\mathrm{h}_{0}\in H$. (The argument is identical
to \cite[Thm 2.6]{Benarous} and makes crucial use of proposition \ref{BA1_18}%
.) The Laplace method is then applicable: we replace $\varepsilon dW$ by $%
\varepsilon dW+d\mathrm{h}_{0}$ in (\ref{SDEXeps}) and call the resulting
diffusion process $Z^{\varepsilon }$. The integrand of (\ref%
{ThisIntegrandForLaplaceMethod}) can then be expressed in terms with $%
\mathrm{X}^{\varepsilon }$ replaced by \textrm{Z}$^{\varepsilon }$; of
course at the price of including the Girsanov factor%
\begin{equation*}
\mathcal{G}:=\exp \left( -\frac{1}{\varepsilon }\int_{0}^{T}\mathrm{\dot{h}}%
_{0}\left( t\right) dW_{t}-\frac{1}{2\varepsilon ^{2}}\int_{0}^{T}\left\vert
\mathrm{\dot{h}}_{0}\left( t\right) \right\vert ^{2}dt\right) =\exp \left( -%
\frac{1}{\varepsilon }\int_{0}^{T}\mathrm{\dot{h}}_{0}\left( t\right) dW_{t}-%
\frac{1}{\varepsilon ^{2}}\Lambda \left( \mathrm{a}\right) \right) .
\end{equation*}%
A stochastic Taylor expansion of \textrm{Z}$^{\varepsilon }$, noting right
away that%
\begin{equation*}
F\left( \Pi _{l}\mathrm{Z}_{T}^{\varepsilon }\right) |_{\varepsilon
=0}=F\left( \Pi _{l}\phi _{T}^{h}\left( \mathrm{x}_{0}\right) \right)
=F\left( \mathrm{a}\right) =0,
\end{equation*}%
then leads to (cf. \cite[Lemme 1.43]{Benarous2})
\begin{eqnarray}
&&\exp \left( -\frac{1}{\varepsilon ^{2}}F\left( \Pi _{l}\mathrm{Z}%
_{T}^{\varepsilon }\right) \right)  \notag \\
&=&\exp \left( -\frac{1}{\varepsilon ^{2}}\left[ F\left( \mathrm{a}\right)
-\varepsilon \int_{0}^{T}\dot{h}_{0}\left( t\right) dW_{t}-\varepsilon \Pi
_{l}\mathrm{\hat{X}}_{T}\cdot \Lambda' \left( \mathrm{a}%
\right) +O\left( \varepsilon ^{2}\right) \right] \right)  \notag \\
&=&\exp \left( \frac{1}{\varepsilon }\int_{0}^{T}\mathrm{\dot{h}}_{0}\left(
t\right) dW_{t}+\frac{1}{\varepsilon }\left( \mathrm{\hat{Y}}_{T}\right)
\cdot \Lambda' \left( \mathrm{a}\right) +O\left(
1\right) \right) .  \label{bigO1}
\end{eqnarray}%
Putting things together, we have, using $F\left( \mathrm{a}\right) =0$, and
noting cancellation of $\int_{0}^{T}\mathrm{\dot{h}}_{0}\left( t\right)
dW_{t}$ in (\ref{bigO1}) with the identical term in the Girsanov factor $%
\mathcal{G}$,
\begin{eqnarray}
f^{\varepsilon }(\mathrm{a},T) &=&\frac{1}{\left( 2\pi \varepsilon \right)
^{l}}\int_{\mathbb{R}^{l}}\mathbb{E}\left[ \mathcal{G\times }\exp \left(
i\left( \mathrm{\zeta },0\right) \cdot \left( \frac{\mathrm{Z}%
_{T}^{\varepsilon }-\left( \mathrm{a},0\right) }{\varepsilon }\right)
\right) e^{-\frac{F\left( \Pi _{l}\mathrm{Z}_{T}^{\varepsilon }\right) }{%
\varepsilon ^{2}}}\right] d\mathrm{\zeta }  \notag \\
&=&\frac{1}{\varepsilon ^{l}}\exp \left( -\frac{1}{\varepsilon ^{2}}\Lambda
\left( \mathrm{a}\right) \right) \exp \left( \frac{1}{\varepsilon }\left(
\mathrm{\hat{Y}}_{T}\right) \cdot \Lambda' \left(
\mathrm{a}\right) \right)  \notag \\
&&\times \underset{=:c_{0}}{\underbrace{\frac{1}{\left( 2\pi \right) ^{l}}%
\int_{\mathbb{R}^{l}}\mathbb{E}\left[ \exp \left( i\left( \mathrm{\zeta }%
,0\right) \cdot \left( \frac{\mathrm{Z}_{T}^{\varepsilon }-\left( \mathrm{a}%
,0\right) }{\varepsilon }\right) \right) \exp \left( O\left( 1\right)
\right) \right] d\mathrm{\zeta }}}  \label{finalFactorc0}
\end{eqnarray}%
where $O\left( 1\right) $ denotes the term, bounded as $\varepsilon
\downarrow 0$, from (\ref{bigO1}). What is left to show, of course, is that $%
c_{0}$, i.e. the final factor in the above expression, is indeed a strictly
positive and finite real number. But since our analysis is based on the full
Markovian process $\mathrm{X}^{\varepsilon }$ (resp. $\mathrm{Z}%
^{\varepsilon }$ after change of measure), the arguments of \cite[Lemme
(3.25)]{Benarous} apply with essentially no changes. In particular, one uses
large deviations as in \cite[Lemme (3.25)]{Benarous}) and, crucially,
non-degeneracy of the minimizer $\mathrm{h}_{0}\in H$, guaranteed by
proposition \ref{BA1_18}. Finally, integrating the asymptotic expansion with
respect to $\mathrm{\zeta }\in \mathbb{R}^{l}$ is justified using the
estimates of \cite[Lemme (3.48)]{Benarous}, obtained using Malliavin
calculus techniques. (There is a slip in \cite[Lemme (3.36)]{Benarous}; the
correction was given in \cite[p.23]{Me}). At last one sees $c_{0}>0$, as in
\cite[p. 330]{Benarous}.
\end{proof}

\begin{remark}[Finitely many multiple minimizers]
\label{RmkMultipleMin}The case
\begin{equation*}
\mathcal{K}_{\mathrm{a}}^{\min }=\left\{ \mathrm{h}_{0}^{\left( 1\right)
},\dots ,\mathrm{h}_{0}^{\left( n\right) }\right\} ,
\end{equation*}%
is handled as in \cite{Benarous2}. Assuming invertibility of the Malliavin
matrix as well as non-focality along each of these minimizers, the expansion
for $f^{\varepsilon }\left( \mathrm{y},T\right) $ as given in theorem \ref%
{thm:MainThm} remains valid. Indeed, after localization around each of these
$n$ minimizers,%
\begin{eqnarray*}
f^{\varepsilon }\left( \mathrm{y},T\right) &=&\left( \sum_{\mathrm{h}_{0}\in
\mathcal{K}_{\mathrm{a}}^{\min }}e^{-\frac{\Lambda \left( \mathrm{y}\right)
}{\varepsilon ^{2}}}e^{\frac{\Lambda ^{\prime }\left( \mathrm{y}\right)
\cdot \,\mathrm{\hat{Y}}_{T}\left( \mathrm{h}_{0}\right) }{\varepsilon }%
}\varepsilon ^{-l}c_{0}\left( \mathrm{h}_{0}\right) \right) (1+O\left(
\varepsilon \right) ) \\
&\sim &\mathrm{(const)}e^{-\frac{\Lambda \left( \mathrm{y}\right) }{%
\varepsilon ^{2}}}e^{\,\max \left\{ \frac{\Lambda ^{\prime }\left( \mathrm{y}%
\right) \cdot \,\mathrm{\hat{Y}}_{T}\left( \mathrm{h}_{0}\right) }{%
\varepsilon }:\mathrm{h}_{0}\in \mathcal{K}_{\mathrm{a}}^{\min }\right\}
}\varepsilon ^{-l}
\end{eqnarray*}%
where $\mathrm{\hat{Y}}_{T}\left( \mathrm{h}_{0}\right) $ denotes the
solution of (\ref{eq:SDEYHat}).
\end{remark}

\begin{remark}[Localization] \label{rem:loc}
The assumptions on the coefficients $b,\sigma $ in theorem \ref{thm:MainThm}
(smooth, bounded with bounded derivatives of all orders) are typical in this
context (cf. Ben\ Arous \cite{Benarous, Benarous2} for instance) but rarely
met in practical examples. This difficulty can be resolved by a
suitable localization which we now outline. Set $\tau _{R}:=\inf \left\{
t\in \left[ 0,T\right] :\sup_{s\in \left[ 0,t\right] }\left\vert \mathrm{X}%
_{s}^{\varepsilon }\right\vert \geq R\right\} $ and assume
\begin{equation*}
\mathbb{P}\left[ \tau _{R}\leq T\right] \lesssim e^{-J_{R}/\varepsilon ^{2}}%
\text{ as }\varepsilon \downarrow 0
\end{equation*}%
with $J_{R}\rightarrow \infty $ as $R\rightarrow \infty $ by this we mean,
more precisely,
\begin{equation}
\lim_{R\rightarrow \infty }\lim \sup_{\varepsilon \rightarrow 0}\varepsilon
^{2}\log \mathbb{P}\left[ \tau _{R}\leq T\right] =-\infty .  \label{Assloc}
\end{equation}%
In that case, we can pick $R$ large enough so that $c_1 = \Lambda \left( \mathrm{a}%
\right) <J_{R}$, uniformly for $\varepsilon $ near $0+$, and can expect that
the behaviour beyond some big ball of radius $R$ will not influence the
expansion. In particular, if the coefficients $b,\sigma $ are smooth, but
fail to be bounded resp. have bounded derivatives, we can modify them
outside a ball of radius $R$ such as to have this property; call $\tilde{b},%
\tilde{\sigma}$ these new coefficients and $\mathrm{\tilde{X}}^{\varepsilon
} $ the associated diffusion. To illustrate the localization, consider $l=1$%
, i.e. $\mathrm{Y}_{T}^{\varepsilon }\equiv \mathrm{X}_{T}^{\varepsilon ,1}$%
, and the distribution function for $\mathrm{Y}_{T}^{\varepsilon }$.
Clearly, one has the two-sided estimates
\begin{equation*}
\mathbb{P}\left[ \mathrm{Y}_{T}^{\varepsilon }\geq \mathrm{a};\tau _{R}>T%
\right] \leq \mathbb{P}\left[ \mathrm{Y}_{T}^{\varepsilon }\geq \mathrm{a}%
\right] \leq \mathbb{P}\left[ \mathrm{Y}_{T}^{\varepsilon }\geq \mathrm{a}%
;\tau _{R}>T\right] +\mathbb{P}\left[ \tau _{R}\leq T\right] ,
\end{equation*}%
and similar for $\mathrm{\tilde{Y}}_{T}^{\varepsilon }$ $\equiv \mathrm{%
\tilde{X}}_{T}^{\varepsilon ,1}$. Since $\mathbb{P}\left[ \mathrm{Y}%
_{T}^{\varepsilon }\geq \mathrm{a};\tau _{R}>T\right] =\mathbb{P}\left[
\mathrm{\tilde{Y}}_{T}^{\varepsilon }\geq \mathrm{a};\tau _{R}>T\right] $ it
then follows%
\begin{equation*}
\left\vert \mathbb{P}\left[ \mathrm{Y}_{T}^{\varepsilon }\geq \mathrm{a}%
\right] -\mathbb{P}\left[ \mathrm{\tilde{Y}}_{T}^{\varepsilon }\geq \mathrm{y%
}\right] \right\vert \leq \mathbb{P}\left[ \tau _{R}\leq T\right] \lesssim
e^{-J_{R}/\varepsilon ^{2}}.
\end{equation*}%
In particular, any expansion for $\mathrm{\tilde{Y}}_{T}^{\varepsilon }$ of
the form
\begin{equation*}
\mathbb{P}\left[ \mathrm{\tilde{Y}}_{T}^{\varepsilon }\geq \mathrm{a}\right]
=e^{-c_{1}/\varepsilon ^{2}}e^{\,c_{2}/\varepsilon ^{2}}\varepsilon
^{-l}c_{0}\left( 1+O\left( \varepsilon \right) \right)
\end{equation*}%
leads, upon taking $R$ large enough so that $J_{R}>c_{1}$, to the same
expansion for $\mathbb{P}\left[ \mathrm{Y}_{T}^{\varepsilon }\geq \mathrm{a}%
\right] $. With more work of routine type, this localization can also be
employed for the density expansion in theorem \ref{thm:MainThm}.
\end{remark}

\subsection{Corollary on short time expansions}

We have the following application to short time asymptotics.

\begin{corollary}
\textbf{(Short time)} \label{CorShortTime}Consider $d\mathrm{X}_{t}=b\left(
\mathrm{X}_{t}\right) dt+\sigma \left( \mathrm{X}_{t}\right) dW$, started at
$\mathrm{X}_{0}=\mathrm{x}_{0}\in \mathbb{R}^{d}$, with $C^{\infty }$%
-bounded vector fields such that the strong H\"{o}rmander condition holds,
\begin{equation}
\forall \mathrm{x}\in \mathbb{R}^{d}:\text{\textrm{Lie}}\left[ \sigma
_{1},\dots ,\sigma _{m}\right] |_{\mathrm{x}}=\mathcal{T}_{\mathrm{x}}%
\mathbb{R}^{d}.  \tag{H1}  \label{H1cond}
\end{equation}%
For fixed $l\in \left\{ 1,\dots ,d\right\} $ assume $\left\{ \mathrm{x}%
_{0}\right\} \times N_{\mathrm{a}}$, where $N_{\mathrm{a}}:=\left( \mathrm{a}%
,\cdot \right) $ for some $\mathrm{a}\in \mathbb{R}^{l}$, satisfies
condition \textrm{(ND)}. Let $f\left( t,\mathrm{y}%
\right) $ be the density of $\mathrm{Y}_{t}=\left( \mathrm{X}_{t}^{1},\dots ,%
\mathrm{X}_{t}^{l}\right)$. Then the following expansion holds at $\mathrm{y}=\mathrm{a}$,
\begin{equation*}
f\left( t,\mathrm{a}\right) \sim \text{(const)}\frac{1}{t^{l/2}}\exp \left( -%
\frac{d^{2}\left( \mathrm{x}_{0},\mathrm{a}\right) }{2t}\right) \text{ as }%
t\downarrow 0,
\end{equation*}%
where $d\left( \mathrm{x}_{0},\mathrm{a}\right) $ is the sub-Riemannian
distance, based on $\left( \sigma _{1},\dots ,\sigma _{m}\right) $, from the
point $\mathrm{x}_{0}$ to the affine subspace $N_{\mathrm{a}}$.
\end{corollary}

\begin{proof}
After Brownian scaling, we apply the theorem with $T=1,\varepsilon ^{2}=t$
so that%
\begin{equation*}
b\left( \varepsilon ,\cdot \right) =\varepsilon ^{2}b\left( \cdot \right)
\rightarrow \sigma _{0}\left( \cdot \right) \equiv 0;
\end{equation*}%
which explains why there is no drift vector field in the present H\"{o}%
rmander condition (\ref{H1cond}). Also $\mathrm{x}_{0}^{\varepsilon }=%
\mathrm{x}_{0}$ here. The identification of the energy with $1/2$ times the
square of the sub-Riemannian (or: control - , Carnot-Caratheodory - )
distance from $\mathrm{x}_{0}$ to $N_{\mathrm{a}}$ is classical. At last,
the unique ODE solution to (\ref{eq:SDEYHat}) is then given by $\mathrm{\hat{%
Y}}\equiv 0$ and there is no $\exp \left\{ \left( ...\right) /\varepsilon
\right\} $ factor.
\end{proof}

\section{Non-focality and infinite-dimensional non-degeneracy}

\label{SectionNonFocality}\textit{In this section we give the complete proof
of the crucial proposition \ref{BA1_18}}. \textit{To lighten notation, we
write }$\mathrm{h}$\textit{\ (rather than }$\mathrm{h}_{0}$\textit{) for an
arbitrary fixed element in }$K_{\mathrm{a}}^{\min }.$ \bigskip

By assumption the deterministic Malliavin matrix $C\left( \mathrm{h}\right) $
is invertible and so (cf. Bismut \cite[Thm 1.5]{Bismut}) the space $\mathcal{%
K}_{\mathrm{a}}$ enjoys a (Hilbert) manifold structure, locally around the
minimizer $\mathrm{h}$, and the tangent space at $\mathrm{h}$
\begin{equation*}
\mathcal{T}_{\mathrm{h}}\mathcal{K}_{\mathrm{a}}\cong \ker D\left( \Pi
_{l}\circ \phi _{T}\right) \left( \mathrm{h}\right) =:H_{0}.
\end{equation*}%
can be identified as%
\begin{equation*}
H_{0}=\left\{ \mathrm{k}\in H:\int_{0}^{T}\sum_{j=1}^{m}\left\langle \mathrm{%
p},\Phi _{T\leftarrow t}\left( \mathrm{h}\right) \sigma _{j}\left( \mathrm{x}%
_{t}\right) \right\rangle \text{\textrm{\.{k}}}_{t}^{j}dt=0\,\,\forall
\mathrm{p}\in \mathrm{span}\left[ dx^{1},\dots ,dx^{l}\right] |_{\mathrm{x}%
_{T}}\subset \mathcal{T}_{\mathrm{x}_{T}}^{\ast }\mathbb{R}^{d}\text{ }%
\right\} .
\end{equation*}%
Let us also write $\psi _{T}\equiv \Pi _{l}\phi _{T}$. Since $\mathrm{h}$ is
a minimizer of the energy, we have the first order optimality condition,
\begin{equation*}
I^{\prime }\left( \mathrm{h}\right) =DI\left( \mathrm{h}\right) =0\text{ on }%
\mathcal{T}_{\mathrm{h}}\mathcal{K}_{\mathrm{a}}=H_{0}.\text{ }
\end{equation*}%
We write $\mathrm{x}_{T}=\phi _{T\leftarrow 0}^{\mathrm{h}}\left( \mathrm{x}%
_{0}\right) $ or $\phi _{T}\left( \mathrm{h}\right) $ with $\mathrm{x}_{0}$
fixed. Given%
\begin{equation*}
\mathrm{q}\in \mathrm{span}\left\{ dx^{1}|_{\mathrm{x}_{T}},\dots ,dx^{l}|_{%
\mathrm{x}_{T}}\right\}
\end{equation*}%
with $1\leq l\leq d$ we shall write
\begin{equation*}
\left( \mathrm{q},0\right) \in \mathrm{span}\left\{ dx^{1}|_{\mathrm{x}%
_{T}},\dots ,dx^{d}|_{\mathrm{x}_{T}}\right\} =\mathcal{T}_{\mathrm{x}%
_{T}}^{\ast }\mathbb{R}^{d}
\end{equation*}%
for \textrm{q} "viewed" as element in $\mathcal{T}_{\mathrm{x}_{T}}^{\ast }%
\mathbb{R}^{d}$. We can describe $H_{0}$ as the set of those $\mathrm{k}%
=\left( k^{1},\dots ,k^{m}\right) \in H$ such that, for any \textrm{q}$\in
\mathrm{span}\left\{ dx^{1}|_{\mathrm{x}_{T}},\dots ,dx^{l}|_{\mathrm{x}%
_{T}}\right\} $,%
\begin{equation*}
\int_{0}^{T}\left\langle \left( \mathrm{q},0\right) ,\Phi _{T\leftarrow t}^{%
\mathrm{h}}\sigma _{i}\left( \phi _{t\leftarrow T}^{\mathrm{h}}\left(
\mathrm{x}_{T}\right) \right) \right\rangle \dot{k}_{t}^{i}dt=0;
\end{equation*}%
where, of course, $\left\{ x^{1},\dots ,x^{d}\right\} $ denotes the standard
coordinate chart of $\mathbb{R}^{d}$ and we tacitly use Einstein's summation
convention.

\begin{lemma}
The linear map $\tilde{\rho}_{\mathrm{h}}:$ $\mathrm{span}\left\{ dx^{1}|_{%
\mathrm{x}_{T}},\dots ,dx^{l}|_{\mathrm{x}_{T}}\right\} \rightarrow H$ given
by%
\begin{equation*}
\tilde{\rho}_{\mathrm{h}}\left( \mathrm{q}\right) :=\left(
\begin{array}{c}
\int_{0}^{\cdot }\left\langle \left( \mathrm{q},0\right) ,\Phi _{T\leftarrow
t}^{\mathrm{h}}\sigma _{1}\left( \phi _{t\leftarrow T}^{\mathrm{h}}\left(
\mathrm{x}_{T}\right) \right) \right\rangle dt \\
\cdots \\
\int_{0}^{\cdot }\left\langle \left( \mathrm{q},0\right) ,\Phi _{T\leftarrow
t}^{\mathrm{h}}\sigma _{m}\left( \phi _{t\leftarrow T}^{\mathrm{h}}\left(
\mathrm{x}_{T}\right) \right) \right\rangle dt%
\end{array}%
\right)
\end{equation*}%
for $i=1,\dots ,m$ and \thinspace $t\in \left[ 0,T\right] $ is one-one with
range $H_{0}^{\bot }$.
\end{lemma}

\begin{proof}
Since $H_{0}$ is the set of those \textrm{k}$\in H$ such that, for any
\textrm{q}$\in \mathrm{span}\left\{ dx^{1}|_{\mathrm{x}_{T}},\dots ,dx^{l}|_{%
\mathrm{x}_{T}}\right\} $,%
\begin{equation*}
\int_{0}^{T}\left\langle \left( \mathrm{q},0\right) ,\Phi _{T\leftarrow t}^{%
\mathrm{h}}\sigma _{i}\left( \phi _{t\leftarrow T}^{\mathrm{h}}\left(
\mathrm{x}_{T}\right) \right) \right\rangle \dot{k}_{t}^{i}dt=0
\end{equation*}%
we see that $H_{0}$ is the orthogonal complement in $H$ of%
\begin{equation*}
\left\{ \tilde{\rho}_{\mathrm{h}}\left( \mathrm{q}\right) :\mathrm{q}\in
\mathrm{span}\left\{ dx^{1}|_{\mathrm{x}_{T}},\dots ,dx^{l}|_{\mathrm{x}%
_{T}}\right\} \right\} ;
\end{equation*}%
i.e. $H_{0}^{\bot }$ is the range of $\tilde{\rho}_{\mathrm{h}}$.
Invertibility of the\ deterministic Malliavin matrix (along $h$) then
implies $\ker \tilde{\rho}_{\mathrm{h}}=\left\{ 0\right\} $ which shows that
$\tilde{\rho}_{\mathrm{h}}$ is one-one (and also that $H_{0}^{\bot }$ has
dimension $l$).
\end{proof}

\begin{lemma}
\label{LemmaFormOfh}For each minimizer $\mathrm{h}\in \mathcal{K}_{\mathrm{a}%
}^{\min }$, there exists a unique $\mathrm{q}=\mathrm{q}\left( \mathrm{h}%
\right) \in \mathrm{span}\left\{ dx^{1}|_{\mathrm{x}_{T}},\dots ,dx^{l}|_{%
\mathrm{x}_{T}}\right\} $ s.t.%
\begin{equation*}
\mathrm{h}=D\phi _{T}\left( \mathrm{h}\right) ^{\ast }\left[ \left( \mathrm{q%
},0\right) \right] .
\end{equation*}%
(Recall $D\phi _{T}\left( \mathrm{h}\right) :H\rightarrow \mathcal{T}_{%
\mathrm{x}_{T}}\mathbb{R}^{d};$ its adjoint then maps $\mathcal{T}_{\mathrm{x%
}_{T}}^{\ast }\mathbb{R}^{d}\rightarrow H$ where we identify $H^{\ast }$
with $H$.)
\end{lemma}

\begin{proof}
By assumption, $\mathrm{h}=\left( h^{1},\dots ,h^{m}\right) $ is a
minimizer, and so its differential $I^{\prime }\left( \mathrm{h}\right) $ is
$0$ on $T_{\mathrm{h}}\mathcal{K}_{a}\equiv H_{0}$. It follows that for
every \textrm{k}$\in H_{0}$,%
\begin{equation*}
\left\langle dI\left( \mathrm{h}\right) ,\mathrm{k}\right\rangle
=\int_{0}^{T}\sum_{i=1}^{m}\dot{h}_{t}^{i}\dot{k}_{t}^{i}dt=0
\end{equation*}%
so that $\mathrm{h}$ is in the orthogonal complement of $H_{0}$. It follows
that there exists a (unique, thanks to invertibility of the\ deterministic
Malliavin matrix along $\mathrm{h}$)%
\begin{equation*}
\mathrm{q}=\mathrm{q}\left( \mathrm{h}\right) \in \mathrm{span}\left\{
dx^{1}|_{\mathrm{x}_{T}},\dots ,dx^{l}|_{\mathrm{x}_{T}}\right\}
\end{equation*}%
such that $\mathrm{h}=\tilde{\rho}_{\mathrm{h}}\left( \mathrm{q}\right) $.
It follows that
\begin{equation*}
\dot{h}_{t}^{i}=\left\langle \left( \mathrm{q},0\right) ,\Phi _{T\leftarrow
t}^{\mathrm{h}}\sigma _{i}\left( \phi _{t\leftarrow T}^{\mathrm{h}}\left(
\mathrm{x}_{T}\right) \right) \right\rangle .
\end{equation*}%
It remains to see that, for any \textrm{k}$\in H$,%
\begin{equation*}
\left\langle \mathrm{k,h}\right\rangle _{H}=\left\langle \mathrm{k},D\phi
_{T}\left( \mathrm{h}\right) ^{\ast }\left[ \left( \mathrm{q},0\right) %
\right] \right\rangle _{H}=\left\langle \left( \mathrm{q},0\right) ,D\phi
_{T}\left( \mathrm{h}\right) \left[ \mathrm{k}\right] \right\rangle ,
\end{equation*}%
but this follows immediately from the computation%
\begin{eqnarray*}
\left\langle \mathrm{k,h}\right\rangle _{H} &=&\,\left\langle \mathrm{k},%
\hat{\rho}_{\mathrm{h}}\left( \mathrm{q}\right) \right\rangle _{H} \\
&=&\int_{0}^{T}\dot{k}_{t}^{i}\left\langle \left( \mathrm{q},0\right) ,\Phi
_{T\leftarrow t}^{\mathrm{h}}\sigma _{i}\left( \phi _{t\leftarrow T}^{%
\mathrm{h}}\left( \mathrm{x}_{T}\right) \right) \right\rangle dt \\
&=&\left\langle \left( \mathrm{q},0\right) ,\int_{0}^{T}\Phi _{T\leftarrow
t}^{\mathrm{h}}\sigma _{i}\left( \phi _{t\leftarrow T}^{\mathrm{h}}\left(
\mathrm{x}_{T}\right) \right) \dot{k}_{t}^{i}dt\right\rangle .
\end{eqnarray*}
\end{proof}

\begin{lemma}
\bigskip\ $I^{\prime \prime }\left( \mathrm{h}\right) $ is a bilinear form
on $H_{0}$ given by%
\begin{eqnarray*}
I^{\prime \prime }\left( \mathrm{h}\right) \left[ \mathrm{k},\mathrm{l}%
\right] &=&\left\langle \mathrm{k},\mathrm{l}\right\rangle _{H}-\left\langle
\left( \mathrm{q}\left( \mathrm{h}\right) ,0\right) ,D^{2}\phi _{T}\left(
h\right) \left[ \mathrm{k},\mathrm{l}\right] \right\rangle \\
&=&\left\langle \mathrm{k},\mathrm{l}\right\rangle _{H}-\left( \mathrm{q}%
\left( \mathrm{h}\right) ,D^{2}\psi _{T}\left( h\right) \left[ \mathrm{k},%
\mathrm{l}\right] \right)
\end{eqnarray*}%
where $\left( \mathrm{q}\left( \mathrm{h}\right) ,0\right) \in \mathcal{T}_{%
\mathrm{x}_{T}}^{\ast }\mathbb{R}^{d}$ was constructed in lemma \ref%
{LemmaFormOfh}. In particular, an element $\mathrm{k}\in H_{0}$ is in the
null-space $\mathcal{N}\left( \mathrm{h}\right) $ of $I^{\prime \prime
}\left( \mathrm{h}\right) $,
\begin{eqnarray*}
\mathrm{k} &\in &\mathcal{N}\left( \mathrm{h}\right) :=\left\{ \mathrm{k}\in
H_{0}:I^{\prime \prime }\left( \mathrm{h}\right) \left[ \mathrm{k},\mathrm{k}%
\right] =0\right\} \\
&=&\left\{ \mathrm{k}\in H_{0}:I^{\prime \prime }\left( \mathrm{h}\right) %
\left[ \mathrm{k},\cdot \right] \equiv 0\text{ on }H_{0}\text{ }\right\} .
\end{eqnarray*}%
if and only if (identifying $H^{\ast }$ with $H$)
\begin{equation*}
\left\langle \mathrm{k},\cdot \right\rangle _{H}-\left( \mathrm{p}%
_{T},D^{2}\psi _{T}\left( h\right) \left[ \mathrm{k},\cdot \right] \right)
\in H_{0}^{\bot }.
\end{equation*}
\end{lemma}

\begin{proof}
Take a smooth curve $c:\left( -\varepsilon ,\varepsilon \right) \rightarrow $
$\mathcal{K}_{\mathrm{a}}$ s.t. $c\left( 0\right) =\mathrm{h},\dot{c}\left(
0\right) =\mathrm{k}$. Then%
\begin{equation*}
I^{\prime \prime }\left( h\right) \left[ \mathrm{k},\mathrm{k}\right]
=\left\Vert \mathrm{k}\right\Vert _{H}^{2}+\left\langle \mathrm{h},\overset{%
..}{c}\left( 0\right) \right\rangle .
\end{equation*}%
From the previous lemma
\begin{eqnarray*}
I^{\prime \prime }\left( \mathrm{h}\right) \left[ \mathrm{k},\mathrm{k}%
\right] &=&\left\Vert \mathrm{k}\right\Vert _{H}^{2}+\left\langle \left(
\mathrm{q},0\right) ,D\phi _{T}\left( \mathrm{h}\right) \left[ \overset{..}{c%
}\left( 0\right) \right] \,\right\rangle \\
&=&\left\Vert \mathrm{k}\right\Vert _{H}^{2}+\left\langle \mathrm{q},D\psi
_{T}\left( \mathrm{h}\right) \left[ \overset{..}{c}\left( 0\right) \right]
\,\right\rangle .
\end{eqnarray*}%
On the other hand, since $\psi _{T}\left( c\left( t\right) \right) =\Pi
_{l}\phi _{T}\left( c\left( t\right) \right) \equiv \mathrm{a}$ for $t\in
\left( -\varepsilon ,\varepsilon \right) $ we have%
\begin{eqnarray*}
0 &=&\frac{d^{2}}{dt^{2}}\psi _{T}\left( c\left( t\right) \right) |_{t=0} \\
&=&\frac{d}{dt}D\psi _{T}\left( c\left( t\right) \right) \left[ \dot{c}%
\left( t\right) \right] |_{t=0} \\
&=&D^{2}\psi _{T}\left( \mathrm{h}\right) \left[ \mathrm{k},\mathrm{k}\right]
+D\psi _{T}\left( \mathrm{h}\right) \left[ \overset{..}{c}\left( 0\right) %
\right]
\end{eqnarray*}%
and hence%
\begin{eqnarray*}
I^{\prime \prime }\left( \mathrm{h}\right) \left[ \mathrm{k},\mathrm{k}%
\right] &=&\left\Vert \mathrm{k}\right\Vert _{H}^{2}-\left\langle \mathrm{q}%
,D^{2}\psi _{T}\left( \mathrm{h}\right) \left[ \mathrm{k},\mathrm{k}\right]
\,\right\rangle \\
&=&\left\Vert \mathrm{k}\right\Vert _{H}^{2}-\left\langle \left( \mathrm{q}%
,0\right) ,D^{2}\phi _{T}\left( \mathrm{h}\right) \left[ \mathrm{k},\mathrm{k%
}\right] \,\right\rangle .
\end{eqnarray*}%
The characterization of elements in $\mathcal{N}\left( \mathrm{h}\right) $
is then clear. Let us just remark that $\mathcal{N}\left( \mathrm{h}\right) $
is indeed equal to the space $\left\{ \mathrm{k}\in H_{0}:I^{\prime \prime
}\left( \mathrm{h}\right) \left[ \mathrm{k},\cdot \right] \equiv 0\text{ on }%
H_{0}\text{ }\right\} $ as is easily seen from the fact that $I^{\prime
\prime }\left( \mathrm{h}\right) $ is positive semi-definite, since $\mathrm{%
h}$ is (by assumption)\ a minimizer.
\end{proof}

If $U$ is a vector field on $\mathbb{R}^{d}$ we define the push-forward,
under the diffeomorphism $\left( \phi _{s\leftarrow T}^{\mathrm{h}}\right)
^{-1}$, by
\begin{equation*}
\left( \phi _{s\leftarrow T}^{\mathrm{h}}\right) _{\ast }^{-1}U\left(
\mathrm{z}\right) :=\left( \Phi _{s\leftarrow T}^{\mathrm{h}}\right)
^{-1}U\left( \phi _{s\leftarrow T}^{\mathrm{h}}\left( \mathrm{z}\right)
\right) \in \mathcal{T}_{\mathrm{z}}\mathbb{R}^{d}
\end{equation*}%
We shall then need the following known formula, cf. \cite[1.21]{Bismut}
combined with trivial time reparameterization $t\leadsto T-t$;
\begin{equation}
D\left( \phi _{t\leftarrow T}^{\mathrm{h}}\right) _{\ast }^{-1}U\left(
\mathrm{z}\right) \left[ \mathrm{k}\right] =\int_{t}^{T}\left[ \left( \phi
_{s\leftarrow T}^{\mathrm{h}}\right) _{\ast }^{-1}\sigma _{j},\left( \phi
_{t\leftarrow T}^{\mathrm{h}}\right) _{\ast }^{-1}U\right] \left( \mathrm{z}%
\right) \dot{k}_{s}^{j}ds.  \label{Bimsut121}
\end{equation}

\begin{lemma}
For \textrm{k}$,\mathrm{l}\in H$ we have, with $\mathrm{x}_{T}=\phi
_{T}\left( \mathrm{h}\right) $,%
\begin{eqnarray*}
D^{2}\phi _{T}\left( \mathrm{h}\right) \left[ \mathrm{k},\mathrm{l}\right]
&=&\int_{0}^{T}\int_{t}^{T}\left[ \left( \phi _{s\leftarrow T}^{\mathrm{h}%
}\right) _{\ast }^{-1}\sigma _{j},\left( \phi _{t\leftarrow T}^{\mathrm{h}%
}\right) _{\ast }^{-1}\sigma _{i}\right] \left( \mathrm{x}_{T}\right) \dot{k}%
_{s}^{j}\dot{l}_{t}^{i}dsdt \\
&&+\int_{0}^{T}\Phi _{T\leftarrow t}^{\mathrm{h}}\partial \sigma _{i}\left(
\phi _{t\leftarrow T}^{\mathrm{h}}\left( \mathrm{x}_{T}\right) \right) \Phi
_{t\leftarrow T}^{\mathrm{h}}D\phi _{T}\left( \mathrm{h}\right) \left[
\mathrm{k}\right] \dot{l}_{t}^{i}dt.
\end{eqnarray*}
\end{lemma}

\begin{proof}
Clearly%
\begin{equation*}
D\phi _{T}\left( \mathrm{h}\right) \left[ \mathrm{l}\right]
=\int_{0}^{T}\Phi _{T\leftarrow t}^{\mathrm{h}}\sigma _{i}\left( \phi
_{t\leftarrow T}^{\mathrm{h}}\left( \mathrm{x}_{T}\right) \right) \dot{l}%
_{t}^{i}dt
\end{equation*}%
where $\phi _{T}\left( \mathrm{h}\right) =\mathrm{x}_{T}$. Perturbing\textrm{%
\ }$\mathrm{h}$ implies
\begin{equation*}
\phi _{T}\left( \mathrm{h}+\varepsilon \mathrm{k}\right) =\mathrm{x}%
_{T}+\varepsilon D\phi _{T}\left( \mathrm{h}\right) \left[ \mathrm{k}\right]
+o\left( \varepsilon \right)
\end{equation*}%
and then%
\begin{equation*}
D\phi _{T}\left( \mathrm{h}+\varepsilon \mathrm{k}\right) \left[ \mathrm{l}%
\right] =\int_{0}^{T}\Phi _{T\leftarrow t}^{\mathrm{h}+\varepsilon \mathrm{k}%
}\sigma _{i}\left( \phi _{t\leftarrow T}^{\mathrm{h}+\varepsilon \mathrm{k}%
}\left( \mathrm{x}_{T}+\varepsilon D\phi _{T}\left( \mathrm{h}\right) \left[
\mathrm{k}\right] +o\left( \varepsilon \right) \right) \right) \dot{l}%
_{t}^{i}dt.
\end{equation*}%
Taking derivatives then leads us to\footnote{%
It should be noted that the term $D\phi _{T}\left( \mathrm{h}\right) \left[ k%
\right] $ is zero for $k\in H_{1}=\ker D\phi _{T}\left( \mathrm{h}\right) $;
in particular the second summand will vanish when $D^{2}\phi _{T}\left(
\mathrm{h}\right) \left[ \cdot ,\cdot \right] $ is restricted to $H_{1}$
i.e. when considering the point-point case $l=d$.}%
\begin{eqnarray*}
D^{2}\phi _{T}\left( \mathrm{h}\right) \left[ \mathrm{k},\mathrm{l}\right]
&=&\int_{0}^{T}D\left\{ \Phi _{T\leftarrow t}^{\mathrm{h}}\sigma _{i}\left(
\phi _{t\leftarrow T}^{\mathrm{h}}\left( \mathrm{x}_{T}\right) \right)
\right\} \left[ \mathrm{k}\right] \dot{l}_{t}^{i}dt \\
&&+\int_{0}^{T}\Phi _{T\leftarrow t}^{\mathrm{h}}\partial \sigma _{i}\left(
\phi _{t\leftarrow T}^{\mathrm{h}}\left( \mathrm{x}_{T}\right) \right) \Phi
_{t\leftarrow T}^{\mathrm{h}}D\phi _{T}\left( \mathrm{h}\right) \left[
\mathrm{k}\right] \dot{l}_{t}^{i}dt.
\end{eqnarray*}%
The proof is then finished using (\ref{Bimsut121}).
\end{proof}

Given \textrm{k}$\in H_{0}$, set%
\begin{equation}
\left(
\begin{array}{c}
0 \\
\eta%
\end{array}%
\right) :=D\phi _{T}\left( \mathrm{h}\right) \left[ \mathrm{k}\right]
\label{etak}
\end{equation}%
where the notation is meant to suggest that%
\begin{equation*}
\eta \in \mathcal{T}_{\mathrm{x}_{T}}N_{\mathrm{a}}\text{ where }N_{\mathrm{a%
}}=\left( \mathrm{a,\cdot }\right) \subset \mathbb{R}^{l}\times \mathbb{R}%
^{d-l}\cong \mathbb{R}^{d}.
\end{equation*}

\begin{proposition}
\label{PropVolterraK}Elements \textrm{k}$\in \mathcal{N}\left( \mathrm{h}%
\right) \subset H_{0}$ are characterized by (inhomogeneous, linear
"backward")\ Volterra equation\footnote{%
... which takes the usual form upon reparameterizing time $\tau \leftarrow
T-t$ ...}%
\begin{eqnarray*}
\dot{k}_{t}^{i} &=&\left\langle \left( \mathrm{q}\left( \mathrm{h}\right)
,0\right) ,\int_{t}^{T}\left[ \left( \phi _{s\leftarrow T}^{\mathrm{h}%
}\right) _{\ast }^{-1}\sigma _{j},\left( \phi _{t\leftarrow T}^{\mathrm{h}%
}\right) _{\ast }^{-1}\sigma _{i}\right] \left( \mathrm{x}_{T}\right) \dot{k}%
_{s}^{j}ds\right\rangle \\
&&+\left\langle \left( \mathrm{q}\left( \mathrm{h}\right) ,0\right) ,\Phi
_{T\leftarrow t}^{\mathrm{h}}\partial \sigma _{i}\left( \phi _{t\leftarrow
T}^{\mathrm{h}}\left( \mathrm{x}_{T}\right) \right) \Phi _{t\leftarrow T}^{%
\mathrm{h}}\left(
\begin{array}{c}
0 \\
\eta%
\end{array}%
\right) \right\rangle \\
&&+\left\langle \left( \theta ,0\right) ,\Phi _{T\leftarrow t}^{\mathrm{h}%
}\sigma _{i}\left( \phi _{t\leftarrow T}^{\mathrm{h}}\left( \mathrm{x}%
_{T}\right) \right) \right\rangle .
\end{eqnarray*}%
where%
\begin{equation*}
\eta =\eta \left( \mathrm{k}\right) \in \mathrm{span}\left\{ \partial
_{l+1}|_{\mathrm{x}_{T}},\dots ,\partial _{d}|_{\mathrm{x}_{T}}\right\} =%
\mathcal{T}_{\mathrm{x}_{T}}N_{\mathrm{a}}
\end{equation*}%
is given by (\ref{etak}) and
\begin{equation*}
\theta =\theta \left( \mathrm{k}\right) \in \mathrm{span}\left\{ dx^{1}|_{%
\mathrm{x}_{T}},\dots ,dx^{l}|_{\mathrm{x}_{T}}\right\} =\mathcal{T}_{%
\mathrm{x}_{T}}^{\ast }N_{\mathrm{a}}^{\perp }.
\end{equation*}
\end{proposition}

\begin{remark}
When \textrm{k}$\in \mathcal{N}\left( \mathrm{h}\right) $ is also in $%
H_{1}=\ker D\phi _{T}\left( \mathrm{h}\right) $ (which is always true in the
point-point setting!) we have $\eta =0;\,$the equation for \textrm{k}
simplifies accordingly and matches precisely the Bismut's equation \cite[1.65%
]{Bismut}.
\end{remark}

\begin{remark}
It is an important step in our argument to single out $\eta $. In fact, we
must not use%
\begin{equation*}
\left(
\begin{array}{c}
0 \\
\eta%
\end{array}%
\right) =\int_{0}^{T}\Phi _{T\leftarrow s}^{\mathrm{h}}\sigma _{j}\left(
\phi _{s\leftarrow T}^{\mathrm{h}}\left( \mathrm{x}_{T}\right) \right) \dot{k%
}_{s}^{j}ds
\end{equation*}%
as integral term for \textrm{\.{k}} in the above integral equation for
\textrm{\.{k}}. Indeed, doing so would lead to a Fredholm integral equation
(of the second kind) for $\dot{k}$ whereas it will be crucial for the
subsequent argument to have a Volterra structure. (Solutions to such
Volterra equations are unique; the same is not true for Fredholm integral
equations.)
\end{remark}

\begin{proof}
For fixed $\mathrm{k}\in H_{0}$, we write%
\begin{equation*}
\left(
\begin{array}{c}
0 \\
\eta%
\end{array}%
\right) :=D\phi _{T}\left( \mathrm{h}\right) \left[ \mathrm{k}\right] .
\end{equation*}%
With slight abuse of notation (Riesz!) the previous result then implies that%
\begin{eqnarray}
\left\{ D^{2}\phi _{T}\left( \mathrm{h}\right) \left[ \mathrm{k},\cdot %
\right] \right\} _{t}^{i} &=&\int_{t}^{T}\left[ \left( \phi _{s\leftarrow
T}^{\mathrm{h}}\right) _{\ast }^{-1}\sigma _{j},\left( \phi _{t\leftarrow
T}^{\mathrm{h}}\right) _{\ast }^{-1}\sigma _{i}\right] \left( \mathrm{x}%
_{T}\right) \dot{k}_{s}^{j}ds  \label{D2phiT} \\
&&+\Phi _{T\leftarrow t}^{\mathrm{h}}\partial \sigma _{i}\left( \phi
_{t\leftarrow T}^{\mathrm{h}}\left( \mathrm{x}_{T}\right) \right) \Phi
_{t\leftarrow T}^{\mathrm{h}}\left(
\begin{array}{c}
0 \\
\eta%
\end{array}%
\right) .  \notag
\end{eqnarray}%
On the other hand, for $\mathrm{k}\in \mathcal{N}\left( \mathrm{h}\right) $,
we know that
\begin{equation*}
\left\langle \mathrm{k},\cdot \right\rangle _{H}-\left\langle \left( \mathrm{%
q}\left( \mathrm{h}\right) ,0\right) ,D^{2}\phi _{T}\left( h\right) \left[
\mathrm{k},\cdot \right] \right\rangle \in H_{0}^{\perp }=\text{range}\left(
\tilde{\rho}_{\mathrm{h}}\right) .
\end{equation*}%
Hence, recalling%
\begin{equation*}
\tilde{\rho}_{\mathrm{h}}\left( \theta \right) =\left\langle \left( \theta
,0\right) ,\Phi _{T\leftarrow t}^{\mathrm{h}}\sigma _{i}\left( \phi
_{t\leftarrow T}^{\mathrm{h}}\left( \mathrm{x}_{T}\right) \right)
\right\rangle ,\,\,\,
\end{equation*}%
it follows from (\ref{D2phiT}) that%
\begin{eqnarray*}
\dot{k}_{t}^{i} &=&\left\langle \left( \mathrm{q}\left( \mathrm{h}\right)
,0\right) ,\int_{t}^{T}\left[ \left( \phi _{s\leftarrow T}^{\mathrm{h}%
}\right) _{\ast }^{-1}\sigma _{j},\left( \phi _{t\leftarrow T}^{\mathrm{h}%
}\right) _{\ast }^{-1}\sigma _{i}\right] \left( \mathrm{x}_{T}\right) \dot{k}%
_{s}^{j}ds\right\rangle \\
&&+\left\langle \left( \mathrm{q}\left( \mathrm{h}\right) ,0\right) ,\Phi
_{T\leftarrow t}^{\mathrm{h}}\partial \sigma _{i}\left( \phi _{t\leftarrow
T}^{\mathrm{h}}\left( \mathrm{x}_{T}\right) \right) \Phi _{t\leftarrow T}^{%
\mathrm{h}}\left(
\begin{array}{c}
0 \\
\eta%
\end{array}%
\right) \right\rangle \\
&&+\left\langle \left( \theta ,0\right) ,\Phi _{T\leftarrow t}^{\mathrm{h}%
}\sigma _{i}\left( \phi _{t\leftarrow T}^{\mathrm{h}}\left( \mathrm{x}%
_{T}\right) \right) \right\rangle
\end{eqnarray*}
\end{proof}

\begin{remark}
If we introduce the orthogonal complement $H_{2}$ so that%
\begin{equation*}
H_{0}=H_{1}\oplus H_{2}\text{ (orthogonal)}
\end{equation*}%
the map%
\begin{equation*}
\mathrm{k}\mapsto D\phi _{T}\left( \mathrm{h}\right) \left[ \mathrm{k}\right]
=\left(
\begin{array}{c}
0 \\
\eta%
\end{array}%
\right) \mapsto \eta
\end{equation*}%
is a bijection from $H_{2}\rightarrow $ $\mathcal{T}_{\mathrm{x}_{T}}N_{%
\mathrm{a}}$.
\end{remark}

\subsection{Jacobi variation}

Again, the starting point is the formula
\begin{eqnarray*}
\dot{h}_{t}^{i} &=&\left\langle \mathrm{p}_{T},\Phi _{T\leftarrow t}^{%
\mathrm{h}}\sigma _{i}\left( \phi _{t\leftarrow T}^{\mathrm{h}}\left(
\mathrm{x}_{T}\right) \right) \right\rangle \\
&=&\left\langle \left( \mathrm{q}(\mathrm{h}),0\right) ,\Phi _{T\leftarrow
t}^{\mathrm{h}}\sigma _{i}\left( \phi _{t\leftarrow T}^{\mathrm{h}}\left(
\mathrm{x}_{T}\right) \right) \right\rangle
\end{eqnarray*}%
where we recall%
\begin{equation*}
\mathrm{p}_{T}=\left( \mathrm{q}(\mathrm{h}),0\right) ,\,\,\mathrm{x}_{T}\in
\left( \mathrm{a},\cdot \right) \equiv N_{\mathrm{a}}\text{.}
\end{equation*}%
We keep $\mathrm{p}_{T}$ and $\mathrm{x}_{T}$ fixed and note that the
Hamiltonian (backward) dynamics are such that%
\begin{equation*}
\pi \mathrm{H}_{t\leftarrow T}\left( \mathrm{x}_{T},\mathrm{p}_{T}\right)
=\phi _{t\leftarrow T}^{\mathrm{h}}\left( \mathrm{x}_{T}\right) .
\end{equation*}%
Replace $\mathrm{p}_{T}$ by $\mathrm{p}_{T}+\varepsilon (\theta ,0)$ above, $%
\mathrm{x}_{T}$ by $\mathrm{x}_{T}+\varepsilon \left(
\begin{array}{c}
0 \\
\eta%
\end{array}%
\right) $ and write $\mathrm{h}\left( \varepsilon \right) =\left(
h^{1}\left( \varepsilon \right) ,\dots ,h^{m}\left( \varepsilon \right)
\right) $ for the according control\footnote{%
... which can be constructed explicitly from the Hamiltonian\ (backward) flow%
\begin{equation*}
\left( \mathrm{x}_{t}\left( \varepsilon \right) ,\mathrm{p}_{t}\left(
\varepsilon \right) \right) :=\mathrm{H}_{t\leftarrow T}\left( \mathrm{x}%
_{T}+\varepsilon \left(
\begin{array}{c}
0 \\
\eta%
\end{array}%
\right) ,\mathrm{p}_{T}+\varepsilon (\theta ,0)\right)
\end{equation*}%
and the usual formula $\mathrm{\dot{h}}\left( \varepsilon \right)
_{t}^{i}=\left\langle \sigma _{i}\left( \mathrm{x}_{t}\left( \varepsilon
\right) \right) ,\mathrm{p}_{t}\left( \varepsilon \right) \right\rangle .$}
which satisfies the relation%
\begin{equation*}
\dot{h}\left( \varepsilon \right) _{t}^{i}=\left\langle \mathrm{p}%
_{T}+\varepsilon \left( \theta ,0\right) ,\Phi _{T\leftarrow t}^{\mathrm{h}%
\left( \varepsilon \right) }\sigma _{i}\left( \phi _{t\leftarrow T}^{\mathrm{%
h}\left( \varepsilon \right) }\left( \mathrm{x}_{T}+\varepsilon \left(
\begin{array}{c}
0 \\
\eta%
\end{array}%
\right) \right) \right) \right\rangle \text{ .}
\end{equation*}%
Define the \textit{Jacobi type variation}%
\begin{equation*}
\mathrm{g}:=\partial _{\left( \theta ,\eta \right) }\mathrm{h}:=\frac{%
\partial \mathrm{h}\left( \varepsilon \right) }{\partial \varepsilon }%
|_{\varepsilon =0}
\end{equation*}%
so that, with $\mathrm{g}=\left( g^{1},\dots ,g^{m}\right) $,
\begin{eqnarray*}
\dot{g}_{t}^{i} &=&\left\langle \mathrm{p}_{T},D\left\{ \Phi _{T\leftarrow
t}^{\mathrm{h}}\sigma _{i}\left( \phi _{t\leftarrow T}^{\mathrm{h}}\left(
\mathrm{x}_{T}\right) \right) \right\} \left[ \mathrm{g}\right] \right\rangle
\\
&&+\left\langle \mathrm{p}_{T},\Phi _{T\leftarrow t}^{\mathrm{h}}\partial
\sigma _{i}\left( \phi _{t\leftarrow T}^{\mathrm{h}}\left( \mathrm{x}%
_{T}\right) \right) \Phi _{t\leftarrow T}^{\mathrm{h}}\left(
\begin{array}{c}
0 \\
\eta%
\end{array}%
\right) \right\rangle \\
&&+\left\langle \left( \theta ,0\right) ,\Phi _{T\leftarrow t}^{\mathrm{h}%
}\sigma _{i}\left( \phi _{t\leftarrow T}^{\mathrm{h}}\left( \mathrm{x}%
_{T}\right) \right) \right\rangle .
\end{eqnarray*}%
With $\mathrm{p}_{T}=\left( \mathrm{q}(\mathrm{h}),0\right) $ and formula (%
\ref{Bimsut121}) we see that $\mathrm{\dot{g}}$ satisfies the identical
(inhomogeneous, linear backward\footnote{%
Trivial reparameterization $t\leadsto T-t$ will bring it in standard
"forward" form.}\ Volterra equation) as the one given for \textrm{\.{k}} in
proposition \ref{PropVolterraK}. By basic uniqueness theory for such
Volterra equations we see that $\mathrm{\dot{g}}=$\textrm{\.{k}} as elements
in $L^{2}\left( \left[ 0,T\right] ,\mathbb{R}^{m}\right) $, and hence $%
\mathrm{g}=$\textrm{k} as elements in $H$.

\begin{proposition}
\label{PropKernelElementsAreJacobiAndConverrse}\bigskip Let $\mathrm{k}\in
\mathcal{N}\left( \mathrm{h}\right) \subset H_{0}$ with associated
parameters
\begin{eqnarray*}
\theta &\in &\mathrm{span}\left\{ dx^{1}|_{\mathrm{x}_{T}},\dots ,dx^{l}|_{%
\mathrm{x}_{T}}\right\} =\mathcal{T}_{\mathrm{x}_{T}}^{\ast }N_{\mathrm{a}%
}^{\perp } \\
\eta &\in &\mathrm{span}\left\{ \partial _{l+1}|_{\mathrm{x}_{T}},\dots
,\partial _{d}|_{\mathrm{x}_{T}}\right\} =\mathcal{T}_{\mathrm{x}_{T}}N_{%
\mathrm{a}}
\end{eqnarray*}%
provided by proposition \ref{PropVolterraK}. (In particular, $\eta $ is
given by $D\phi _{T}\left( \mathrm{h}\right) \left[ \mathrm{k}\right] $, cf.
(\ref{etak}).) Then \textrm{k} can be written in terms of a Jacobi type
variation
\begin{equation*}
\mathrm{k}=\partial _{\left( \theta ,\eta \right) }\mathrm{h}.
\end{equation*}%
Conversely, any Jacobi type variation, with $\theta \in $ $\mathcal{T}_{%
\mathrm{x}_{T}}^{\ast }N_{\mathrm{a}}^{\perp },\eta \in \mathcal{T}_{\mathrm{%
x}_{T}}N_{\mathrm{a}}$ yields an element in $\mathcal{N}\left( \mathrm{h}%
\right) $.
\end{proposition}

\begin{proof}
The first part follows from the above discussion and it only remains to
prove the converse part. Since we have seen that every Jacobi type variation
$\mathrm{g}:=\partial _{\left( \theta ,\eta \right) }\mathrm{h}$ satisfies
the appropriate Volterra equation, cf. proposition \ref{PropVolterraK}, we
only need to check
\begin{equation*}
\left(
\begin{array}{c}
0 \\
\eta%
\end{array}%
\right) =D\phi _{T}\left( \mathrm{h}\right) \left[ \mathrm{g}\right]
\end{equation*}%
and we leave this as an easy exercise to the reader.
\end{proof}

Recall that we say that $\mathrm{x}_{0}$ is \textit{non-focal} for $\left(
\mathrm{a},\cdot \right) \equiv N_{\mathrm{a}}$ along $\mathrm{h}$ if for
all $\theta \in \mathcal{T}_{\mathrm{x}_{T}}^{\ast }N_{\mathrm{a}}^{\perp }$%
, $\eta \in \mathcal{T}_{\mathrm{x}_{T}}N_{\mathrm{a}}$
\begin{equation*}
\partial _{\varepsilon }|_{\varepsilon =0}\pi \mathrm{H}_{0\leftarrow
T}\left( \mathrm{x}_{T}+\varepsilon \left(
\begin{array}{c}
0 \\
\eta%
\end{array}%
\right) ,\mathrm{p}_{T}+\varepsilon \left( \theta ,0\right) \right)
=0\implies \left( \theta ,\eta \right) =0.
\end{equation*}%
In the point-point setting (i.e. $l=d$ so that $\theta \in \mathcal{T}_{%
\mathrm{x}_{T}}^{\ast }\mathbb{R}^{d},\,\eta =0$) the criterion reduces to
\begin{equation*}
\partial _{\varepsilon }|_{\varepsilon =0}\pi \mathrm{H}_{0\leftarrow
T}\left( \mathrm{x}_{T},\mathrm{p}_{T}+\varepsilon \theta \right) =0\implies
\theta =0\text{;}
\end{equation*}%
disregarding time reparameterization $t\leadsto T-t$ and the fact that our
setup allows for a non-zero drift vector field, this is precisely Bismut's
non-conjugacy condition \cite[p.50]{Bismut}.

\begin{corollary}
The point $\mathrm{x}_{0}$ is \textit{non-focal} for $\left( \mathrm{a}%
,\cdot \right) \equiv N_{\mathrm{a}}$ along $\mathrm{h}$ if and only if $%
I^{\prime \prime }\left( \mathrm{h}\right) $, i.e. the second derivative of $%
\left. \left\Vert \cdot \right\Vert _{H}^{2}\right\vert _{\mathcal{K}_{%
\mathrm{a}}}$ at the minimizer $\mathrm{h}$, viewed as quadratic form on $%
H_{0}=\ker D\left( \Pi _{l}\phi _{T}\right) \left( \mathrm{h}\right) $, is
non-degenerate, i.e.%
\begin{equation*}
\mathcal{N}\left( \mathrm{h}\right) \equiv \left\{ 0\right\} .
\end{equation*}
\end{corollary}

\begin{proof}
"$\Rightarrow $": Take $\mathrm{k}\in \mathcal{N}\left( \mathrm{h}\right) ;$
from proposition \ref{PropKernelElementsAreJacobiAndConverrse}%
\begin{equation*}
\mathrm{k}=\partial _{\left( \theta ,\eta \right) }\mathrm{h}\equiv \partial
_{\varepsilon }|_{\varepsilon =0}\mathrm{h}\left( \varepsilon \right)
\end{equation*}%
for suitable $\theta \in \mathcal{T}_{\mathrm{x}_{T}}^{\ast }N_{\mathrm{a}%
}^{\perp },\eta \in \mathcal{T}_{\mathrm{x}_{T}}N_{\mathrm{a}}$; in fact,%
\begin{equation*}
\left(
\begin{array}{c}
0 \\
\eta%
\end{array}%
\right) =D\phi _{T\leftarrow 0}^{\mathrm{h}}\left( \mathrm{x}_{0}\right) %
\left[ \mathrm{k}\right] .
\end{equation*}%
The criterion says that if%
\begin{equation*}
\partial _{\varepsilon }|_{\varepsilon =0}\pi \mathrm{H}_{0\leftarrow
T}\left( \mathrm{x}_{T}+\varepsilon \left(
\begin{array}{c}
0 \\
\eta%
\end{array}%
\right) ,\mathrm{p}_{T}+\varepsilon \left( \theta ,0\right) \right)
=\partial _{\varepsilon }|_{\varepsilon =0}\left( \phi _{0\leftarrow T}^{%
\mathrm{h}\left( \varepsilon \right) }\right) \left( \mathrm{x}%
_{T}+\varepsilon \left(
\begin{array}{c}
0 \\
\eta%
\end{array}%
\right) \right)
\end{equation*}%
equals zero then $\left( \theta ,\eta \right) $ must be zero. But this is
indeed the case here since
\begin{eqnarray*}
&&\partial _{\varepsilon }|_{\varepsilon =0}\left( \phi _{0\leftarrow T}^{%
\mathrm{h}\left( \varepsilon \right) }\right) \left( \mathrm{x}%
_{T}+\varepsilon \left(
\begin{array}{c}
0 \\
\eta%
\end{array}%
\right) \right) \\
&=&D\left\{ \phi _{0\leftarrow T}^{\mathrm{h}}\left( \mathrm{x}_{T}\right)
\right\} \left[ \partial _{\varepsilon }|_{\varepsilon =0}\mathrm{h}\left(
\varepsilon \right) \right] +\Phi _{0\leftarrow T}^{\mathrm{h}}\left(
\begin{array}{c}
0 \\
\eta%
\end{array}%
\right) \\
&=&D\left\{ \phi _{0\leftarrow T}^{\mathrm{h}}\left( \mathrm{x}_{T}\right)
\right\} \left[ \mathrm{k}\right] +\Phi _{0\leftarrow T}^{\mathrm{h}}D\phi
_{T\leftarrow 0}^{\mathrm{h}}\left( \mathrm{x}_{0}\right) \left[ \mathrm{k}%
\right] \\
&=&D\left\{ \phi _{0\leftarrow T}^{\mathrm{h}}\circ \phi _{T\leftarrow 0}^{%
\mathrm{h}}\left( \mathrm{x}_{0}\right) \right\} \left[ \mathrm{k}\right] \\
&=&0.
\end{eqnarray*}%
We thus conclude that the directional derivative $\partial _{\left( \theta
,\eta \right) }\mathrm{h}$, which of course depends linearly on $\left(
\theta ,\eta \right) $, vanishes. It then follows that \textrm{k}$=\partial
_{\left( \theta ,\eta \right) }\mathrm{h}=0$ which is what we wanted to show.%
\newline
"$\Leftarrow $": Assume there exists $\left( \theta ,\eta \right) \neq 0$ so
that%
\begin{equation*}
\partial _{\varepsilon }|_{\varepsilon =0}\pi \mathrm{H}_{0\leftarrow
T}\left( \mathrm{x}_{T}+\varepsilon \left(
\begin{array}{c}
0 \\
\eta%
\end{array}%
\right) ,\mathrm{p}_{T}+\varepsilon \left( \theta ,0\right) \right) =0.
\end{equation*}%
Then \textrm{k}$:=\partial _{\left( \theta ,\eta \right) }\mathrm{h}$ yields
an element in the null-space $\mathcal{N}\left( \mathrm{h}\right) $. We need
to see that \textrm{k} is non-zero. Assume otherwise, i.e. \textrm{k}$=0$.
Then $D\phi _{T\leftarrow 0}^{\mathrm{h}}\left( \mathrm{x}_{0}\right) \left[
\mathrm{k}\right] =0$ and hence also $\eta =0$. From the Volterra equation
for \textrm{k} we see that%
\begin{equation*}
0=\left\langle \left( \theta ,0\right) ,\Phi _{T\leftarrow t}^{\mathrm{h}%
}\sigma _{i}\left( \phi _{t\leftarrow T}^{h}\left( \mathrm{x}_{T}\right)
\right) \right\rangle =\tilde{\rho}_{\mathrm{h}}\left( \left( \theta
,0\right) \right) .
\end{equation*}%
But $\ker \tilde{\rho}_{\mathrm{h}}$ was seen to be trivial and so $\theta
=0 $; in contradiction to assumption $\left( \theta ,\eta \right) \neq 0$.
\end{proof}

\section{Examples}

We now consider a number of examples which illustrate the use of our main
result, theorem \ref{thm:MainThm}. As already noted in remark \ref{rem:loc},
the boundedness assumptions on the vector fields (i.e. SDE coefficients) are
rarely met. It is, however, easy enough in all the following examples to check the
localization estimate (\ref{Assloc}) so that application of theorem \ref{thm:MainThm} is
indeed fully justified.

\subsection{Scalar Ornstein-Uhlenbeck process}

As a warmup, consider a $1$-dimensional Ornstein-Uhlenbeck process $%
Y^{\varepsilon }$ with small noise parameter $\varepsilon $. (Since there is
no projection here, there is no need to distinguish between $l$-dimensional $%
\mathrm{Y}^{\varepsilon }$ and $d$-dimensional $\mathrm{X}^{\varepsilon }$.)
Fix $\gamma >0$ and assume dynamics of the form%
\begin{equation*}
dY_{t}^{\varepsilon }=\left( \alpha \varepsilon +\beta Y_{t}^{\varepsilon
}\right) dt+\gamma \varepsilon dW_{t},\,\,\,Y_{0}^{\varepsilon }=\varepsilon
\hat{y}_{0}\in \mathbb{R}
\end{equation*}%
with explicit solution at time $T>0$ given by the variation of constants
formula
\begin{equation*}
Y_{T}^{\varepsilon }=\varepsilon \hat{y}_{0}e^{\beta T}+\varepsilon \alpha
\int_{0}^{T}e^{\beta \left( T-t\right) }dt+\varepsilon \gamma
\int_{0}^{T}e^{\beta \left( T-t\right) }dW_{t}.
\end{equation*}%
In particular, using It\^{o}'s isometry, $Y_{T}^{\varepsilon }\sim N\left(
\varepsilon \mu ,\varepsilon ^{2}\sigma ^{2}\right) $ with%
\begin{equation}
\mu :=\hat{y}_{0}e^{\beta T}+\alpha \int_{0}^{T}e^{\beta \left( T-t\right)
}dt,\,\,\sigma ^{2}:=\gamma ^{2}\int_{0}^{T}e^{2\beta \left( T-t\right) }dt
\label{1DOUmu}
\end{equation}%
and so $Y_{T}^{\varepsilon }$ admits a density of the form
\begin{equation}
f^{\varepsilon }\left( y,T\right) =\frac{1}{\varepsilon \sigma \sqrt{2\pi }}%
\exp \left( -\frac{\left( y-\varepsilon \mu \right) ^{2}}{2\varepsilon
^{2}\sigma ^{2}}\right) \equiv \varepsilon ^{-1}e^{-c_{1}/\varepsilon
^{2}}e^{c_{2}/\varepsilon }\left( c_{0}+O\left( \varepsilon \right) \right) ,
\label{OUdensity1D}
\end{equation}%
where, in particular,
\begin{equation}
c_{1}=\frac{1}{2}\frac{y^{2}}{\sigma ^{2}}\text{ and }c_{2}=\mu y/\sigma ^{2}%
\text{.}  \label{OUexample_c1c2}
\end{equation}%
Let us derive the same from our theorem \ref{thm:MainThm}. Since $%
y_{0}:=\lim_{\varepsilon \rightarrow 0}\,Y_{0}^{\varepsilon }=0$, the
associated control problem is of the form%
\begin{equation*}
d\phi _{t}^{h}=\beta \phi _{t}^{h}dt+\gamma dh_{t},\,\,\,\phi
_{0}^{h}=y_{0}=0.
\end{equation*}%
The Hamiltonian is given by $\mathcal{H}\left( y,p\right) =\beta yp+\frac{1}{%
2}\gamma ^{2}p^{2}$ and the Hamiltonian ODEs to be solved read
\begin{equation*}
\dot{y}_{t}=\beta y_{t}+\gamma ^{2}p_{t},\,\,\,\dot{p}_{t}=-\beta p_{t}.
\end{equation*}%
(with boundary data $y_{0}=0,y_{T}=y$). By variation of constants, $%
y_{T}=y_{t}e^{\beta \left( T-t\right) }+\int_{t}^{T}e^{\beta \left(
T-s\right) }\gamma ^{2}p_{s}ds$ and it easily follows that the Hamiltonian
flow, as function of $\left( y_{T},p_{T}\right) $, is given by%
\begin{eqnarray}
p_{t} &=&p_{T}e^{\beta \left( T-t\right) }  \notag \\
y_{t} &=&y_{T}e^{-\beta \left( T-t\right) }-e^{-\beta \left( T-t\right)
}\int_{t}^{T}e^{\beta \left( T-s\right) }\gamma ^{2}p_{s}ds  \notag \\
&=&y_{T}e^{-\beta \left( T-t\right) }-\gamma ^{2}p_{T}e^{-\beta \left(
T-t\right) }\int_{t}^{T}e^{2\beta \left( T-s\right) }ds
\label{OU1dexample_HamFlow}
\end{eqnarray}%
Taking into account the boundary data $y_{t}|_{t=0}=0$ and $y_{T}=y,$ we find%
\begin{equation*}
y=\gamma ^{2}p_{T}\int_{0}^{T}e^{2\beta \left( T-s\right) }ds\implies p_{T}=%
\frac{y}{\sigma ^{2}}
\end{equation*}%
where $\sigma $ was defined in (\ref{1DOUmu}). According to (\ref{Formula_h0}%
), the (only candidate for a) minimizing control $h_{0}$ is then given via $%
\dot{h}_{0}\left( t\right) =\gamma p_{t}=\gamma p_{T}e^{\beta \left(
T-t\right) }$ so that%
\begin{equation*}
\Lambda \left( y\right) =\frac{1}{2}\int_{0}^{T}\left\vert \dot{h}_{0}\left(
t\right) \right\vert ^{2}dt=\frac{1}{2}p_{T}^{2}\gamma
^{2}\int_{0}^{T}e^{2\beta \left( T-t\right) }dt=\frac{1}{2}p_{T}^{2}\sigma
^{2}=\frac{1}{2}\frac{y^{2}}{\sigma ^{2}}
\end{equation*}%
in agreement with $c_{1}$ as given in (\ref{OUexample_c1c2}). To compute $%
c_{2}$ we specialize (\ref{eq:SDEYHat}) to our situation and the resulting
ODE\ reads%
\begin{equation*}
d\hat{Y}_{t}=\beta \hat{Y}_{t}dt+\alpha dt,\,\,\,\,\hat{Y}_{0}=\hat{y}_{0}.
\end{equation*}%
One readily computes $\hat{Y}_{T}=\hat{y}_{0}e^{\beta T}+\alpha
\int_{0}^{T}e^{\beta \left( T-t\right) }dt$, which equals precisely $\mu $
as defined (\ref{1DOUmu}). Noting that $\Lambda ^{\prime }\left( y\right)
=y/\sigma ^{2}$ we find indeed $\hat{Y}_{T}\Lambda ^{\prime }\left( y\right)
=\mu y/\sigma ^{2}=c_{2}$ in agreement with (\ref{OUexample_c1c2}).

Finally, a word concerning the non-degeneracy condition \textrm{(ND)}, upon
which a justified application of theorem \ref{thm:MainThm} relies. Clearly,
as we have seen, there is only one minimizer. Invertibility of the
deterministic Malliavin covariance matrix is trivially guaranteed due to
ellipticity (here: $\gamma >0$). Finally, the non-focality condition (which
here reduces to the a non-conjugacy condition) requires $y_{0}=\pi \mathrm{H}%
_{0\leftarrow T}\left( y_{T},p_{T}\right) $ to be non-degenerate as function
of $p_{T}$. But this follows from (\ref{OU1dexample_HamFlow})$|_{t=0}$;
indeed%
\begin{equation*}
\frac{\partial y_{0}}{\partial p_{T}}=-\gamma ^{2}e^{-\beta
T}\int_{0}^{T}e^{2\beta \left( T-s\right) }ds\neq 0.
\end{equation*}

\subsection{Langevin dynamics, tail behaviour}

We consider a classical hypoelliptic situation, with Langevin dynamics given
by
\begin{align*}
dY& =Zdt,\quad Y_{0}=\hat{y}_{0}, \\
dZ& =dW_{t},\quad Z_{0}=\hat{z}_{0}.
\end{align*}%
Of course, $Y_{T}$ is Gaussian with mean $\mu =\hat{y}_{0}+\hat{z}_{0}T$ and
variance%
\begin{equation*}
\sigma ^{2}:=\mathbb{V}\left[ Y_{T}\right] =\mathbb{E}\left[
\int_{0}^{T}\int_{0}^{T}W_{s}W_{t}dsdt\right] =2\left[ \int_{0<s<t<T}sdsdt%
\right] =\int_{0}^{T}t^{2}dt=T^{3}/3.
\end{equation*}%
We are \textit{not} looking here at the short time behaviour of $Y_{t}$ as $%
t\downarrow 0$: Indeed, condition (\ref{H1cond}) is not satisfied here and
indeed the $\log $ density of $Y_{t}$ is proportional to $1/\sigma ^{2}=$ $%
O\left( t^{-3}\right) $ as $t\downarrow 0$ which is not at all the behaviour
described in corollary (\ref{CorShortTime}). Instead, we fix $T>0$ and note
that the density of $Y_{T}$ is of the form%
\begin{equation}
\frac{1}{\sqrt{2\pi \sigma ^{2}}}e^{-\frac{\left( y-\mu \right) ^{2}}{%
2\sigma ^{2}}}\sim \text{(const)}e^{-\frac{3}{2T^{3}}\left( y^{2}-2\mu
y\right) }\equiv \text{(const)}e^{-c_{1}y^{2}+c_{2}y}\text{ as }y\uparrow
\infty .  \label{2DGaussHypoExplicit}
\end{equation}%
We now show how to derive this from theorem \ref{thm:MainThm}.

\textbf{Scaling:} Set $Y^{\varepsilon }:=\varepsilon Y$ and similarly for $Z$%
. Then%
\begin{align*}
dY_{t}^{\varepsilon }& =Z_{t}^{\varepsilon }dt,\quad Y_{0}^{\varepsilon
}=\varepsilon \hat{y}_{0}. \\
dZ_{t}^{\varepsilon }& =\varepsilon dW_{t},\quad Z_{0}^{\varepsilon
}=\varepsilon \hat{z}_{0}.
\end{align*}%
(We also set $y_{0}=\lim_{\varepsilon \rightarrow 0}Y_{0}^{\varepsilon }=0$
and similarly $z_{0}=\lim_{\varepsilon \rightarrow 0}Z_{0}^{\varepsilon }=0$%
.) The density expansion of $Y_{T}$, as the space variable $y$ tends to $%
+\infty $, is readily obtained from the density expansion of $%
Y_{T}^{\varepsilon }$, at unit in space, as $\varepsilon =1/y$ tends to
zero. It remains to \textbf{check the assumptions} for theorem \ref%
{thm:MainThm}. With $\sigma _{0}=z\partial _{y}$ and $\sigma _{1}=\partial
_{z}$ we have $\left[ \sigma _{0},\sigma _{1}\right] =\partial _{y}$ which
not only implies the weak H\"{o}rmander's condition \ref{H} but a stronger
"Bismut $\mathrm{H2}$ type" condition which implies \cite[Thm 1.10]{Bismut}
invertibility of $C_{T}^{h,x_{0}}$ for all $h\neq 0$. We are interested in
paths going from the origin in $\mathbb{R}^{2}$ to $N=\left( a,\cdot \right)
$ with $a=1$ and it is easy to see that this is possible upon replacing $W$
by a suitable Cameron-Martin path; in other words,
\begin{equation*}
\mathcal{K}_{a}\neq \varnothing .
\end{equation*}%
(Cf. \cite{Ju} for an abstract criterion that applies in this example).
Since $h\equiv 0$ will never stir us from $\left( 0,0\right) $ to $N$ we
only need to check that $\left( 0,0\right) \times N$ satisfies condition
\textrm{(ND)}. To this end, we note that the Hamiltonian in the present
setting is%
\begin{equation*}
\mathcal{H}\left( \left(
\begin{array}{c}
y \\
z%
\end{array}%
\right) ;\left( p,q\right) \right) =pz+\frac{1}{2}q^{2};
\end{equation*}%
the Hamiltonian ODEs $\dot{y}=z,\,\dot{z}=q;\,\dot{p}=0,\,\dot{q}=-p$ with
(time $T$) terminal data are immediately solved and yield the Hamiltonian
(backward) flow
\begin{eqnarray*}
&&\mathrm{H}_{t\leftarrow T}\left( \left(
\begin{array}{c}
y_{T} \\
z_{T}%
\end{array}%
\right) ;\left( p_{T},\,q_{T}\right) \right) \\
&=&\left( \left(
\begin{array}{c}
y_{T}-z_{T}\left( T-t\right) +q_{T}\left( T-t\right) ^{2}/2+p_{T}\left(
T-t\right) ^{3}/6 \\
z_{T}-q_{T}\left( T-t\right) -p_{T}\left( T-t\right) ^{2}/2%
\end{array}%
\right) ;\left( p_{T},\,q_{T}+p_{T}\left( T-t\right) \right) \right) .
\end{eqnarray*}%
For later reference let us also note%
\begin{equation}
\pi \mathrm{H}_{0\leftarrow T}\left( \left(
\begin{array}{c}
y \\
z%
\end{array}%
\right) ;\left( p,\,q\right) \right) =\left(
\begin{array}{c}
y-zT+qT^{2}/2+pT^{3}/6 \\
z-qT-pT^{2}/2%
\end{array}%
\right) .  \label{LangevinEx_forNF}
\end{equation}%
We solve the Hamiltonian ODEs as boundary value problem (with $%
y_{0}=z_{0}=0;\,y_{T}=a=1$ and $q_{T}=0$). With the explicit form of $\pi
\mathrm{H}_{0\leftarrow T}$, the matching (time $T$) terminal data for
Hamiltonian (backward) flow is immediately computed;
\begin{equation*}
\left.
\begin{array}{c}
0=1-z_{T}T+p_{T}T^{3}/6 \\
0=z_{T}-p_{T}T^{2}/2%
\end{array}%
\right\} \implies z_{T}=3/\left( 2T\right) ,\,p_{T}=3/T^{3}\,
\end{equation*}%
and so the (time $T$) terminal data is found to be $\left( \ast \right)
:=\left( \left(
\begin{array}{c}
1 \\
3/\left( 2T\right)%
\end{array}%
\right) ;\left( 3/T^{3},0\right) \right) $. In particular,%
\begin{equation*}
\mathrm{H}_{t\leftarrow T}|_{\left( \ast \right) }=\left( \left(
\begin{array}{c}
1-\frac{3}{2T}\left( T-t\right) +\frac{1}{2T^{3}}\left( T-t\right) ^{3} \\
z_{T}-\frac{3}{2T^{3}}\left( T-t\right) ^{2}%
\end{array}%
\right) ;\left( \frac{3}{T^{3}},\,\frac{3}{T^{3}}\left( T-t\right) \right)
\right) .
\end{equation*}%
With a look at (\ref{LangevinEx_forNF}), non-focality now follows from%
\begin{equation*}
\det \left( \partial _{z}\pi \mathrm{H}_{0\leftarrow T}|_{\left( \ast
\right) }\,\,\,|\,\,\,\partial _{p}\pi \mathrm{H}_{0\leftarrow T}|_{\left(
\ast \right) }\right) =\det \left( \left(
\begin{array}{c}
-T \\
1%
\end{array}%
\right. \,\left\vert
\begin{array}{c}
T^{3}/6 \\
-T^{2}/2%
\end{array}%
\right. \right) =T^{3}\left( \frac{1}{2}-\frac{1}{6}\right) \neq 0.
\end{equation*}%
Writing $\mathrm{H}_{t\leftarrow T}|_{\left( \ast \right) }=\mathrm{H}%
_{t\leftarrow 0}\left( y_{0},z_{0};p_{0},q_{0}\right) =:\left(
y_{t},z_{t};p_{t},q_{t}\right) ,$ the minimizing control $\mathrm{h}_{0}=%
\mathrm{h}_{0}\left( t\right) $ is then found following the recipe given in
remark \ref{HowToComputeH0}. Since $\sigma _{1}$ is the $z$-coordinate
vector field,
\begin{equation*}
\mathrm{\dot{h}}_{0}\left( t\right) =\left\langle \sigma _{1}\left(
\begin{array}{c}
y_{t} \\
z_{t}%
\end{array}%
\right) ,\left( p_{t},\,q_{t}\right) \right\rangle =q_{t}=\frac{3}{T^{3}}%
\left( T-t\right)
\end{equation*}%
and so
\begin{equation*}
c_{1}:=\Lambda \left( 1\right) =\frac{1}{2}\int_{0}^{T}\left\vert \mathrm{%
\dot{h}}_{0}\left( t\right) \right\vert ^{2}dt=\frac{3}{2T^{3}}
\end{equation*}%
in agreement with (\ref{2DGaussHypoExplicit}). Scaling actually implies $%
\Lambda ^{\prime }\left( 1\right) =2\Lambda \left( 1\right) $. For the
second order constant $c_{2}$, we need to compute $\hat{Y}_{T}$ where%
\begin{equation*}
d\hat{Y}_{t}=\hat{Z}_{t}dt,\,\,\,\,\hat{Y}_{0}=\hat{y}_{0},\,\,\,d\hat{Z}%
_{t}=0,\,\,\,\,\hat{Z}_{0}=\hat{z}_{0}.
\end{equation*}%
This leads immediately to $\hat{Y}_{T}=\hat{y}_{0}+\hat{z}_{0}T=:\mu $ and
then $c_{2}=\Lambda ^{\prime }\left( 1\right) \hat{Y}_{T}=2\mu c_{1},\,$%
again in agreement with the Gaussian computation.

\subsection{An elliptic example with flat metric and degeneracy\label{EEwD}}

Consider the small noise problem for the stochastic differential equation%
\begin{eqnarray*}
dY^{\varepsilon } &=&\varepsilon dW^{1}+\theta Z^{\varepsilon }\varepsilon
dW^{2},\,\,Y_{0}^{\varepsilon }=0; \\
dZ^{\varepsilon } &=&\varepsilon dW^{2},\,\,Z_{0}^{\varepsilon }=0;
\end{eqnarray*}%
where $\theta \in \left[ 0,1\right] $, say. Note that it could be
immediately rephrased as short-time problem $\left( T=1,t=\varepsilon
^{2}\right) $. We are in an elliptic (Riemannian) setting. In fact, the induced metric on $%
\mathbb{R}^{2}$ is flat i.e. has  zero-curvature and hence empty
cut-locus. Clearly, $Y_{T}^{\varepsilon }$ admits a density, say $%
f^{\varepsilon }\left( y\right) $ at time$~T=1$. Considering the point $y=1$%
, for instance, it is not hard to see that
\begin{equation*}
\varepsilon ^{2}\log f^{\varepsilon }\left( 1\right) \sim -\frac{1}{2}\text{
as }\varepsilon \downarrow 0.
\end{equation*}%
At least when $\theta =0$ it is obvious from $Y_{T}^{\varepsilon }$ $\sim
N\left( 0,\varepsilon ^{2}T\right) $ that one has the expansion%
\begin{equation*}
f^{\varepsilon }\left( 1\right) =\varepsilon ^{-1}e^{-\frac{1}{2\varepsilon
^{2}}}\left( c_{0}+O\left( \varepsilon \right) \right)
\end{equation*}%
for some (easy to compute) $c_{0}>0.$ Interestingly, the general situation
is much more involved. Exploiting the fact that $Y_{T}^{\varepsilon }$ can
be written as the independent sum of a Gaussian and a (non-centered)
Chi-square random-variable, $f^{\varepsilon }\left( y\right) $ is given by a
convolution integral and a direct (tedious) analysis shows that%
\begin{equation}
f^{\varepsilon }\left( 1\right) =\left\{
\begin{array}{c}
\varepsilon ^{-1}e^{-\frac{1}{2\varepsilon ^{2}}}\left( c_{0}+O\left(
\varepsilon \right) \right) \text{ \ \ \ when }\theta \in \lbrack 0,1) \\
\varepsilon ^{-3/2}e^{-\frac{1}{2\varepsilon ^{2}}}\left( c_{0}+O\left(
\varepsilon \right) \right) \text{ \ \ \ when }\theta =1%
\end{array}%
\right. .  \label{ChiSquExample}
\end{equation}%
While the energy is equal to $1/2$, no matter the value $\theta \in \left[
0,1\right] $, we see the appearance of an atypical algebraic factor $%
\varepsilon ^{-3/2}$ in the case $\theta =1$.

\bigskip

With a view towards applying our theorem \ref{thm:MainThm}: we have vector
fields $\sigma _{1},\sigma _{2}$ of the form $\partial _{y},$ $\theta
z\partial _{y}+\partial _{z}$. One checks without difficulty that $\mathrm{h}%
_{0}\left( t\right) =\left( t,0\right) $ is the (unique) element in $%
\mathcal{K}_{a}^{\min }$, for any $\theta =\left[ 0,1\right] $. In
particular, the "most-likely" arrival point is $\left( 1,0\right) \in \left(
1,\cdot \right) $. (Minimizers and energy start to look different when $%
\theta >1$ which is why we have focused on $\theta \in \left[ 0,1\right] .$)
In the case $\theta =1$, the explicit "backward" and projected Hamiltonian
flow is
\begin{eqnarray*}
&&\pi \mathrm{H}_{0\leftarrow T}\left( \left(
\begin{array}{c}
y_{T} \\
z_{T}%
\end{array}%
\right) ,\left( p_{T},q_{T}\right) \right) \\
&=&\left(
\begin{array}{c}
y_{T}+\frac{1}{2}\left( p_{T}z_{T}T+q_{T}T-z_{T}\right) ^{2}-\left( p_{T}T+%
\frac{1}{2}z_{T}^{2}\right) \\
z_{T}-q_{T}T-p_{T}z_{T}T%
\end{array}%
\right) .
\end{eqnarray*}%
From this expression, it is then easy to check that $\left( 0,0\right) $ is
focal for $\left( 1,\cdot \right) $. (Proposition \ref{BA1_18} then implies
that the Hessian of the energy at $\mathrm{h}_{0}$ is degenerate. In fact, a
simple computation shows that in this example the null-space of $I^{\prime
\prime }\left( \mathrm{h}_{0}\right) $ is given by $\mathcal{N}\left(
\mathrm{h}_{0}\right) =\mathrm{span}\left\{ \mathrm{k}\right\} $ where $%
\mathrm{k}\in \mathcal{T}_{\mathrm{h}_{0}}\mathcal{K}_{1}\backslash \left\{
0\right\} ,\,$ $\mathrm{k}:\left[ 0,T\right] \rightarrow \mathbb{R}^{2}$
given by $\mathrm{k}\left( t\right) =\left( 0,t\right) $. It follows that
one must not apply theorem \ref{thm:MainThm} here, and indeed, the
prediction of the theorem (algebraic factor $\varepsilon ^{-1}$) would be
false in the case $\theta =1$, as we know from (\ref{ChiSquExample}). On the
other hand, one checks without trouble that for $\theta <1$ the situation is
non-focal, all our assumptions are then met, and so theorem \ref{thm:MainThm}
yields the correct expansion, in agreement with (\ref{ChiSquExample}).

Let us insist that in this example, when $\theta =1$, the degeneracy is precisely due
to focality, whereas the corresponding point-point problem (after all, there is a unique optimal
path from the origin to $(1,0) \in (1,\cdot)$) is non-degenerate. 

\subsection{Brownian motion on the Heisenberg group}\label{LevyAreaSection}

Following a similar discussion by Takanobu--Watanabe \cite{TW}, we consider

\begin{equation*}
\sigma _{1}\left(
\begin{array}{c}
x \\
y \\
z%
\end{array}%
\right) =\left(
\begin{array}{c}
1 \\
0 \\
-y/2%
\end{array}%
\right) ,\,\sigma _{2}\left(
\begin{array}{c}
x \\
y \\
z%
\end{array}%
\right) =\left(
\begin{array}{c}
0 \\
1 \\
x/2%
\end{array}%
\right) .
\end{equation*}%
The solution $\phi \left( \mathrm{h}\right) \equiv \left( x_{\cdot
},y_{\cdot },z_{\cdot }\right) $ to the corresponding controlled ordinary
differential equation in the sense of (\ref{dphih}), with drift $\sigma
_{0}\equiv 0$, has a simple geometric interpretation. Write $\mathrm{h}%
=\left( h^{1},h^{2}\right) $ and assume for simplicity that we start at the
origin, $x_{0}=y_{0}=z_{0}=0$. Then $\left( x_{\cdot },y_{\cdot }\right)
\equiv \left( h^{1},h^{2}\right) $ and $z_{t}$ is the (signed)\ area between
the curve $\left( x_{s},y_{s}:0\leq s\leq t\right) $ and the chord from $%
\left( x_{t},y_{t}\right) $ to $\left( x_{0},y_{0}\right) =\left( 0,0\right)
$ where multiplicity and orientation are taken into account. (When one
starts away from the origin, the interpretation just given holds for $%
t\mapsto \left( x_{0,t},y_{0,t},z_{0,t}\right) $ where%
\begin{equation*}
\left(
\begin{array}{c}
-x_{0} \\
-y_{0} \\
-z_{0}%
\end{array}%
\right) \ast \left(
\begin{array}{c}
x_{t} \\
y_{t} \\
z_{t}%
\end{array}%
\right) =\left(
\begin{array}{c}
x_{0,t} \\
y_{0,t} \\
z_{0,t}%
\end{array}%
\right) ;
\end{equation*}%
where $\left( x,y,z\right) \ast \left( x^{\prime },y^{\prime },z^{\prime
}\right) =\left( x+x^{\prime },y+y^{\prime },z+z^{\prime }+\frac{1}{2}\left(
xy^{\prime }-yx^{\prime }\right) \right) $ gives $\mathbb{R}^{3}$ the
so-called $3$-dimensional Heisenberg group structure. We shall consider unit time horizon, $T=1$, so that $%
\mathrm{h}:\left[ 0,1\right] \rightarrow \mathbb{R}^{2}$. It follows that $%
\left\Vert \mathrm{h}\right\Vert _{H}$ is precisely the (Euclidean) length
of the planar path $\left( x_{t},y_{t}:0\leq t\leq 1\right) $.

The corresponding diffusion process, \textit{Brownian motion on the }$3$\textit{-dimensional
Heisenberg group}, is given by
\begin{equation*}
d\mathrm{X}_{t}=\sigma _{1}\left( \mathrm{X}_{t}\right) dW_{t}^{1}+\sigma
_{2}\left( \mathrm{X}_{t}\right) dW_{t}^{2},\,\,\,\mathrm{X}_{0}=\mathrm{x}%
_{0}\in \mathbb{R}^{3}.
\end{equation*}%
It can also be viewed as the \textit{Brownian rough path} over planar Brownian motion $\left( W^{1},W^{2}\right) $,
see e.g. \cite{FV} and the references therein; the third component is precisely {\it L\'evy's stochastic area}. Set $%
\mathrm{X}^{\varepsilon }:=\delta _{\varepsilon }\mathrm{X}$ where $\delta
_{\varepsilon }\left( x,y,z\right) =\left( \varepsilon x,\varepsilon
y,\varepsilon ^{2}z\right) $ is the dilation operator on the Heisenberg
group. We now consider the small noise problem%
\begin{equation} \label{EBMeps}
d\mathrm{X}_{t}^{\varepsilon }=\sigma _{1}\left( \mathrm{X}_{t}^{\varepsilon
}\right) \varepsilon dW_{t}^{1}+\sigma _{2}\left( X_{t}^{\varepsilon
}\right) \varepsilon dW_{t}^{2}.
\end{equation}%
Since the Hamiltonian flow is analytically tractable \cite{Ga} there is hope
for quite explicit computations.

\subsubsection{Takanobu--Watanabe expansions and focality}
 
We now use our methods\footnote{We note that the localization estimate (\ref{Assloc})
is readily justifed, e.g. by using the Fernique type result \cite{FO}, applicable to
Brownian motion on the Heisenberg group.}
 to recover all "non-degenerate" marginal density 
expansions \cite{TW} based on (\ref{EBMeps}), at time $T=1$ and started at 
$$\mathrm{X}_{0}^{\varepsilon}=0.$$ 
In the notation of that paper, section 7,  we cover their cases (I)$_{1},$(I)$_{2},$(III)$_{1},$(III)$_{2},$(III)%
$_{3}$. The main difference, comparing the approach \cite{TW} with ours, is
that our criterion \textrm{(ND)} bypasses the involved analysis, carried out
by hand in \cite{TW}, of the infinite-dimensional Hessian of the energy at
the minimizer. On the other hand, our approach (presently) does not deal
with degenerate minima, and we do not cover their cases (I)$_{3}$, (II),
(III)$_{4}$; all of which are, of course, ruled out by violating condition
\textrm{(ND)}. The most interesting situation perhaps is the point-line case (III) in which the
degenerate subcase (III)$_4$ is precisely due to focality whereas the corresponding point-point problem
is non-degenerate; this is similar in spirit to the example given in section \ref{EEwD}. 

In order to compute marginal density expansions of $%
\mathrm{X}_{T}^{\varepsilon }$ with $T=1$ and $%
\mathrm{X}_{0}^{\varepsilon }=0$ $\forall \varepsilon >0$ we first note that the
 Hamiltonian takes the form
\begin{equation*}
\mathcal{H}\left( x,y,z;p,q,r\right) =\frac{1}{2}\left( p-\frac{ry}{2}%
\right) ^{2}+\frac{1}{2}\left( q+\frac{rx}{2}\right) ^{2}
\end{equation*}%
and the Hamilton ODEs are%
\begin{equation*}
\left(
\begin{array}{c}
\dot{x} \\
\dot{y} \\
\dot{z}%
\end{array}%
\right) =\left(
\begin{array}{c}
p-\frac{ry}{2} \\
q+\frac{rx}{2} \\
r\left( y^{2}+x^{2}\right) /4+qx/2-py/2%
\end{array}%
\right) ,\,\,\,\left(
\begin{array}{c}
\dot{p} \\
\dot{q} \\
\dot{r}%
\end{array}%
\right) =\left(
\begin{array}{c}
-\frac{r}{2}\left( q+\frac{rx}{2}\right) \\
\frac{r}{2}\left( p-\frac{ry}{2}\right) \\
0%
\end{array}%
\right) .
\end{equation*}%
We compute the Hamiltonian flow. Noting that $r_{t}\equiv r$ constant in
time, we see $\dot{p}$ (resp. $\dot{q}$) is $-r/2$ (resp. $r/2$) times $\dot{%
y}~$\ (resp $\dot{x})$ so that%
\begin{eqnarray}
p_{t} &=&-\frac{r}{2}\left( y_{t}-y_{0}\right) +p_{0},\,  \label{LA_pq} \\
\,q_{r} &=&\frac{r}{2}\left( x_{t}-x_{0}\right) +q_{0}.  \notag
\end{eqnarray}%
It follows that $\left( p,q\right) $ can be expressed "affine linearly" in
terms of $\left( x,y\right) $ and also that $\left( x,y\right) $ is the
solution of a 2-dimensional inhomogenous, linear ODE and one immediately
finds, if $r\neq 0$,%
\begin{eqnarray}
x_{t} &=&\left( \frac{q_{0}}{r}+\frac{x_{0}}{2}\right) \cos \left( rt\right)
+\left( \frac{p_{0}}{r}-\frac{y_{0}}{2}\right) \sin \left( rt\right) -\left(
\frac{q_{0}}{r}-\frac{x_{0}}{2}\right)  \label{LA_xy} \\
y_{t} &=&\left( \frac{q_{0}}{r}+\frac{x_{0}}{2}\right) \sin \left( rt\right)
-\left( \frac{p_{0}}{r}-\frac{y_{0}}{2}\right) \cos \left( rt\right) +\left(
\frac{p_{0}}{r}+\frac{y_{0}}{2}\right) .  \notag
\end{eqnarray}%
Note that $x_{t}+iy_{t}=\left\{ \left( \frac{q_{0}}{r}+\frac{x_{0}}{2}%
\right) -i\left( \frac{p_{0}}{r}-\frac{y_{0}}{2}\right) \right\} e^{irt}$
which makes it plain that $\left( x_{t},y_{t}:0\leq t\leq 1\right) $ form
arcs of circles. Theses circles have radius
\begin{equation}
\rho :=\sqrt{\left( \frac{q_{0}}{r}+\frac{x_{0}}{2}\right) ^{2}+\left( \frac{%
p_{0}}{r}-\frac{y_{0}}{2}\right) ^{2}}  \label{LA_rho}
\end{equation}%
i.e. $\sim 1/r$ as $r\downarrow 0$ (at least when $p_{0}^{2}+q_{0}^{2}>0$).
The limiting case $r=0$ then should correspond to straight lines; and indeed
the Hamiltonian ODEs simplify to $\dot{x}=p,\dot{p}=0$ (similarly for $y,q$)
and so, if $r=0$,%
\begin{equation}
x_{t}=x_{0}+tp_{0},\,\,\,y_{t}=y_{0}+tq_{0}\text{.}  \label{LA_straightlines}
\end{equation}%
(For later reference, let us state explicitely that the Hamiltonian flow
projected to $\left( x_{\cdot },y_{\cdot }\right) $-components is a straight
line if and only if $r=0$.) It remains to determine $z=\left( z_{t}\right) $%
. When $r=0$, this is trivial computation left to the reader. (Actually, if $%
x_{0}=y_{0}=z_{0}=0$ which will always be the case later on, $z_{t}\equiv 0$%
.) Assume now that $r\neq 0$. From the Hamiltonian ODEs we see that $\dot{z}$
is independent of $z$ and so a simple integration over $\left[ 0,t\right] $
yields%
\begin{eqnarray}
z_{t} &=&z_{0}+\frac{1}{8r^{2}}[rt\left( \left( 2q_{0}+rx_{0}\right)
^{2}+\left( -2p_{0}+ry_{0}\right) ^{2}\right)  \label{LA_z} \\
&&-4r\left( p_{0}x_{0}+q_{0}y_{0}\right) \left( -1+\cos \left( rt\right)
\right) +\left( -4p_{0}^{2}-4q_{0}^{2}+r^{2}\left(
x_{0}^{2}+y_{0}^{2}\right) \right) \sin \left( rt\right) ].  \notag
\end{eqnarray}%
For later references we note that (\ref{LA_z})$%
|_{x_{0}=y_{0}=z_{0}=0,t=1,z_{1}=z}$ specializes to%
\begin{equation}
z=\frac{p_{0}^{2}+q_{0}^{2}}{2}\frac{r-\sin r}{r^{2}}.
\label{LA_z_xinit0sub_and_unit_t}
\end{equation}

With (\ref{LA_xy}),(\ref{LA_z}),(\ref{LA_pq}) and $r_{t}\equiv r=r_{0}$ \ we
are now in possession of the explicit solution to the Hamilonian flow $%
\mathrm{H}_{0\rightarrow t}$ as function of the initial data $\left(
x_{0},y_{0},z_{0};p_{0},q_{0},r_{0}\right) $. When $x_{0}=y_{0}=z_{0}=0$,
the projected Hamiltonian flow equals the control $\mathrm{h}\mathbf{=}%
\left( h^{1},h^{2}\right) \equiv \left( x,y\right) $ and its area,%
\begin{equation*}
\pi \mathrm{H}_{0\rightarrow \cdot }\left( 0;p_{0},q_{0},r_{0}\right) \equiv
\left( h^{1},h^{2},\frac{1}{2}\left( \int_{0}^{\cdot
}h^{1}dh^{2}-h^{2}dh^{1}\right) \right) .
\end{equation*}%
In that case, cf. (\ref{EnergyFormula}) with $\sigma _{0}=0$, we have the
simple formula%
\begin{eqnarray} \label{simpleEnergyFormulaEBM}
\frac{1}{2}\left\Vert \mathrm{h}\right\Vert _{H}^{2} &=&\frac{1}{2}\left(
p_{0}^{2}+q_{0}^{2}\right)  \label{LA_energyformula} \\
&=&\frac{1}{2}r^{2}\left[ \left( \frac{p_{0}}{r}\right) ^{2}+\left( \frac{%
q_{0}}{r}\right) ^{2}\right] =\frac{1}{2}r^{2}\rho ^{2}.  \notag
\end{eqnarray}%
In other words, $\left\Vert \mathrm{h}\right\Vert _{H}=\left\vert
r\right\vert \rho $ which is in perfect agreement with $\left\Vert \mathrm{h}%
\right\Vert _{H}$ being the Euclidean length of $\mathrm{h:}\left[ 0,1\right]
\rightarrow \mathbb{R}^{2}$; after all, $\mathrm{h}$ is the arc of a circle
with radius $\rho $ and angle $r$.

We now compute marginal density expansions; to facilitate
comparison with \cite{TW} we use the same case distinctions.\footnote{%
Of course, cases I-III below are not all possible coordinate projections of $%
\mathrm{X}_{1}^{\varepsilon }$ but the remaining cases are either Gaussian, $%
X_{1}^{\varepsilon },Y_{1}^{\varepsilon },\left( X_{1}^{\varepsilon
},Y_{1}^{\varepsilon }\right) $, in which case explicit densities are
available, or reduce to one of the above by symmetry, i.e. $\left(
Y_{1}^{\varepsilon },Z_{1}^{\varepsilon }\right) $ by switching the r\^{o}%
les of $X$ and $Y$.}

\textbf{Case (I):} expand the density of $\mathrm{X}_{1}^{\varepsilon
}=\left( X_{1}^{\varepsilon },Y_{1}^{\varepsilon },Z_{1}^{\varepsilon
}\right) $; here $d=l=3$ and we are dealing with a point-point problem:
given $x,y,z\in \mathbb{R}^{3}$ we look for a curve $\left(
x_{t},y_{t}:0\leq t\leq 1\right) $ of minimal length which, after joining $%
\left( x_{1,}y_{1}\right) =\left( x,y\right) $ and $\left(
x_{0},y_{0}\right) =\left( 0,0\right) $ by a straight line, encloses
(signed) area $z$.

\textbf{Case (II):} expand the density of\textit{\ L\'{e}vy's area} $%
Z_{1}^{\varepsilon }$; here $d=3;l=1$ and we are dealing with a point-plane
problem: given $\,z\in \mathbb{R}$, we need to find a curve $\left(
x_{t},y_{t}:0\leq t\leq 1\right) $ of minimal length which, after joining
its endpoint $\left( x_{1,}y_{1}\right) $ with $\left( x_{0},y_{0}\right)
=\left( 0,0\right) $ by a straight line, encloses (signed) area $z$.

\textbf{Case (III): }the density of $\left( X_{1}^{\varepsilon
},Z_{1}^{\varepsilon }\right) $; here $d=3,l=2$ and we are dealing with a
point-to-line problem: given $x,z\in \mathbb{R}^{2}$, we need to find a
curve $\left( x_{t},y_{t}:0\leq t\leq 1\right) $ of minimal length which
starts at $\left( 0,0\right) $ and arrives at the target manifold $\left(
x,\cdot \right) $ at unit time, such that after joining $\left(
x_{1,}y_{1}\right) $ and $\left( x_{0},y_{0}\right) =\left( 0,0\right) $ by
a straight line, it encloses (signed) area $z$.

\bigskip

In all cases we have to solve the Hamiltonian ODEs subjected to the right
boundary conditions and then check our non-degeneracy condition. Let us note
straight away that we need to rule out $x=y=z=0$ in case (I); $z=0$ in case
(II) and $x=y=0$ in Case (III). Indeed, in all these situations the obvious
minimizing control is $\mathrm{h}\equiv 0$ - but then $\det C^{\mathrm{0}}=0$%
, and hence $\mathrm{(ND)}$ is violated\footnote{%
To see the importance of the condition $C^{\mathrm{h}}\neq 0$ consider the
case (I). The generic prediction of theorem \ref{thm:MainThm} is an
expansion with algebraic factor $\varepsilon ^{-3}$. However, when $x=y=z=0$%
, one has known algebraic factor $\varepsilon ^{-4}$. Such "on-diagonal"
behaviour of hypoelliptic heat-kernels is discussed in \cite{BAL91, BAL91b}.}%
. (For the same reason, these situations are disregarded in \cite{TW}). In
all other situations, $\mathrm{h}$ is \textit{not} identically equal to zero
and hence, by Bismut's $\mathrm{H2}$ condition or direct verification, $\det C^{\mathrm{h}}\neq 0$.
In particular, checking our non-degeneracy conditions boils down to check $%
\#\left\{ \text{minimizers}\right\} <\infty $ and then non-focality (if $d=l$
better called non-conjugacy).

\textbf{Case (I.1)} $x^{2}+y^{2}>0,\,z=0$; the unique shortest path between $%
\left( 0,0\right) $ and $\left( x,y\right) $ is a straight-line which
obviously has zero area throughout and hence is compatible with $z=0$. In
particular then $\mathrm{h}\left( t\right) =\left( tx,ty\right) $ is the
(unique) minimizer and from (\ref{LA_straightlines}) we see $p_{0}=x,q_{0}=y$%
. By considering either the length of $\mathrm{h}$ or by recalling (\ref%
{LA_energyformula}), $\left\Vert \mathrm{h}\right\Vert _{H}^{2}=x^{2}+y^{2}$%
. As for the Hamiltonian flow, we must have $r=0$ (for otherwise, $\mathrm{h}
$ would be an arc), hence $p_{\cdot },q_{\cdot }$ are constant (since $\dot{p%
},\dot{q}\propto r$) with constant values $p_{0}$ and $q_{0}$ respectively.
It can be checked below that $\mathrm{(ND)}$ holds (only non-conjugacy
remains to be checked), as a consequence of theorem \ref{thm:MainThm} we
then have%
\begin{equation*}
f^{\varepsilon }\left( x,y,z;T\right) |_{T=1}\sim e^{-\frac{%
x^{2}+y^{2}}{2\varepsilon ^{2}}}\varepsilon ^{-3} c_0
\end{equation*}%
in agreement with the corresponding expansion given in \cite[Sec. 7, case
(I.2)]{TW}.


\textbf{Case (I.2)} $x^{2}+y^{2}>0,z\neq 0$; since $r=0$
always leads to straight lines, and straight lines have zero area, we
necessarily have $r\neq 0$. Imposing terminal conditions $x_{1}=x,y_{1}=y$ in (\ref{LA_xy}) and then
$z_1=z$, cf.  (\ref{LA_z_xinit0sub_and_unit_t}), allows one to see that, 
given $x, y, z$ as specified by Case (I.2) there is a unique $r\in \left( -2\pi ,2\pi \right)
\backslash \{0\}$ for which
\begin{equation*}
\frac{r-\sin r}{8\sin ^{2}\left( \frac{r}{2}\right) }=\frac{z}{x^{2}+y^{2}}.
\end{equation*}
holds\footnote{%
The same formula appears in \cite[Sec. 7, case (I.2)]{TW}; note that $%
r=2\sigma $ in their notation so that $0<\left\vert \sigma \right\vert <\pi $%
. See also \cite{Ga, M}.}.  It remains to see non-degeneracy of%
\begin{equation*}
\frac{\partial \pi \mathrm{H}_{0\leftarrow 1}\left( x,y,z;p,q,r\right) }{%
\partial \left( p,q,r\right) }|_{\ast }
\end{equation*}%
where $\ast $ stands for evaluation at $\mathrm{H}_{0\rightarrow 1}\left(
0,0,0;p_{0},q_{0},r_{0}\right) =\left( x,y,z,p,q,r\right) $ with
\begin{equation}
p=x\frac{r}{2}\cot \frac{r}{2}, \,\, q=y\frac{r}{2}\cot \frac{r}{2}.
\label{LA_CaseI2_nondegPQ}
\end{equation}%
Now $\left( x_{0},y_{0},z_{0}\right) =\pi \mathrm{%
H}_{0\leftarrow 1}\left( x,y,z;p,q,r\right) $ (which of course equals zero upon evaluation $|_\ast$) can be differentiated in
closed form with respect to $p,q,r$. Taking into account (\ref%
{LA_CaseI2_nondegPQ}), the resulting matrix is indeed a function of $r,x,y$
(observe that dependecy of $z$ also drops out because $z$ appears
additively). A tedious computation (for which we used MATHEMATICA) then
shows
\begin{equation*}
\det \frac{\partial \pi \mathrm{H}_{0\leftarrow 1}\left( x,y,z;p,q,r\right)
}{\partial \left( p,q,r\right) }|_{\ast }=\left( x^{2}+y^{2}\right) \frac{%
\frac{r}{2}\cos \frac{r}{2}-\sin \left( \frac{r}{2}\right) }{r^{2}\sin
\left( \frac{r}{2}\right) }.
\end{equation*}%
By assumption $x^{2}+y^{2}>0$ and we since the remaining fraction above as function of $r$, is strictly
negative on $\left( -2\pi ,2\pi \right)$, we obtain the desired non-degeneracy. In other words, $(0,0,0) $
 and $(x,y,z) $ are non-conjugate (along
the unique minimizer). Hence, after computing the energy with the aid of formula (\ref{simpleEnergyFormulaEBM}),
our theorem \ref{thm:MainThm} gives
\begin{equation} \label{TWI2density}
f^{\varepsilon }\left( x,y,z;T\right) |_{T=1}\sim e^{-\frac{%
x^{2}+y^{2}}{2\varepsilon ^{2}}\frac{r^{2}}{4\sin ^{2}\frac{r}{2}}%
}\varepsilon ^{-3} c_0
\end{equation}%
in agreement with the corresponding expansion given in \cite[Sec. 7, case
(I.2)]{TW}.

\textbf{Case (I.3)} $x=y=0,\,z\neq 0$. We are looking for the shortest path $%
\left( x_{\cdot },y_{\cdot }\right) $ which starts and ends (at unit time)
at the origin, subject to enclosed area $z$. Any minimizing path (which
actually must be a full circles, obviously of radius $\sqrt{\left\vert
z\right\vert /\pi }$; and then perimeter $2\sqrt{\left\vert z\right\vert \pi
}$ and energy $2\left\vert z\right\vert \pi $) may be rotated by some angle $%
\theta \in \lbrack 0,2\pi )$ to yield another, and distinct, minimizing
path. In other words, the assumption of finitely many minimizers which
formed part of condition $\mathrm{(ND)}$ is violated. And indeed, \cite{TW}
find (note the algebraic factor $\varepsilon ^{-4}$ in contrast to theorem %
\ref{thm:MainThm} with generic prediction $\varepsilon ^{-l},\,l=3$)
\begin{equation*}
f^{\varepsilon }\left( x,y,z;T\right) |_{T=1}\sim e^{-\frac{%
2\pi \left\vert z\right\vert }{\varepsilon ^{2}}}\varepsilon ^{-4} c_0.
\end{equation*}

\textbf{Case (II) }Given $\,z\neq 0$, we need to find a curve $\left(
x_{t},y_{t}:0\leq t\leq 1\right) $ of minimal length which, after joining
its endpoint $\left( x_{1,}y_{1}\right) $ with $\left( x_{0},y_{0}\right)
=\left( 0,0\right) $ by a straight line, encloses (signed) area $z$. As is
well-known ("Dido's problem", e.g. \cite{M} and the references therein), and
easily verified by solving the Hamiltonian ODEs with boundary data%
\begin{equation*}
x_{0}=y_{0}=z_{0}=0;\text{ }p_{1}=q_{1}=0\text{ and }z_{1}=z,
\end{equation*}%
the solution to this classical isoperimetric problem is a half-circle (hence
the "angle" $\left\vert r\right\vert $ equals $\pi $; the sign of $r$ is the
sign of $z$). Note that a half-circle of given area $\left\vert z\right\vert
$ has radius $\sqrt{2\left\vert z\right\vert /\pi }$, the length of the arc
is then $\sqrt{2\left\vert z\right\vert \pi }$, the energy equal to $%
\left\vert z\right\vert \pi $.)\

Again, any such half-circle may be rotated by some angle $\theta \in \lbrack
0,2\pi )$ to yield another, and distinct, minimizing path. For the same
reason as in case (I.3) above, condition $\mathrm{(ND)}$ is thus violated.
And indeed, \cite{TW} find a density expansion of L\'{e}vy's area $%
Z_{1}^{\varepsilon }$ of the form (note the algebraic factor $\varepsilon
^{-2}$ in contrast to theorem \ref{thm:MainThm} with generic prediction $%
\varepsilon ^{-l},\,l=1$)%
\begin{equation*}
f^{\varepsilon }\left( z;T\right) |_{T=1}\sim \mathrm{(const)}e^{-\frac{%
\left\vert z\right\vert \pi }{\varepsilon ^{2}}}\varepsilon ^{-2}
\end{equation*}

\textbf{Case (III)}. Given $x,z$, we need to find a curve $\left(
x_{t},y_{t}:0\leq t\leq 1\right) $ of minimal length which starts at $\left(
0,0\right) $ and arrives at the target manifold $\left( x,\cdot \right) $ at
unit time, such that after joining $\left( x_{1,}y_{1}\right) $ and $\left(
x_{0},y_{0}\right) =\left( 0,0\right) $ by a straight line, it encloses
(signed) area $z$. This translates to the following boundary data for the
Hamiltonian ODEs,
\begin{equation}
x_{0}=y_{0}=z_{0}=0;\text{ }x_{1}=x,\text{ }q_{1}=0\text{ and }z_{1}=z.
\label{LA_CaseIII_HBC}
\end{equation}%
We start with an informal discussion. To avoid essentially trivial
situations (in which minimizers are straight lines) we assume $z\neq 0$, and
then w.l.o.g. $z>0$. If we ignore momentarily $x_{1}=x$ (so that we are back
in case II), the energy minimizing path is a half-circle with area $z$;
hence of radius $\rho :\rho =\sqrt{2z/\pi }$; note that the diameter is $%
2\rho =\sqrt{8z/\pi }$. Consider case (III.1) in which this quantity is
strictly greater than $\left\vert x\right\vert $. By symmetry, there are two
minimizing paths - both half-circles - ending at $\left( x,\pm y^{\ast
},z\right) $ for some computable $y^{\ast }>0$. As $z$ decreases, eventually
one has equality $\left\vert x\right\vert =\sqrt{8z/\pi }$ and the two
minimizing paths now collapse into one (half-circle) which ends at $\left(
x,0,z\right) $; following \cite{TW} we call this case (III.4). Note that the
"half-circle" condition $\left\vert r\right\vert =\pi $ (and actually $r=\pi
$ here since $z>0$) holds for both (III.1), (III.4). Finally, we follow \cite%
{TW} in calling " $\left\vert x\right\vert <\sqrt{8z/\pi }$" Case (III.2).
In this case, no half-circle\ (with diameter $2\rho =\sqrt{8z/\pi }$) can
possibly be energy-minimizing for the point-line problem for it cannot
possibly satisfy the admissiblity condition $x_{1}=x$. The (unique)
minimizer in this case is a "less curved" arc (with angle $\left\vert
r\right\vert <\pi $ and actually $r\in \left( 0,\pi \right) $ since $z>0$)
which ends at $\left( x,0,z\right) $. We now claim that cases (III.1) and
(III.2)\ are not focal (so that theorem \ref{thm:MainThm} applies) while
case (III.4)\ is focal. Note that the point-point problem from the origin to
the most likely arrival point $\left( x,0,z\right) $ is non-conjugate; i.e.
we are dealing with a genuine focality phenomena here. (For the sake of
completeness we also consider case (III.3) below which deals with straight
lines.)

\textbf{Case (III.1)} Assume $%
\left\vert x\right\vert <$ $\sqrt{8\left\vert z\right\vert /\pi }$ ($%
\implies z\neq 0$); see figure 1. 
\begin{figure}[hb]
  \centering
  \includegraphics[scale=0.7]{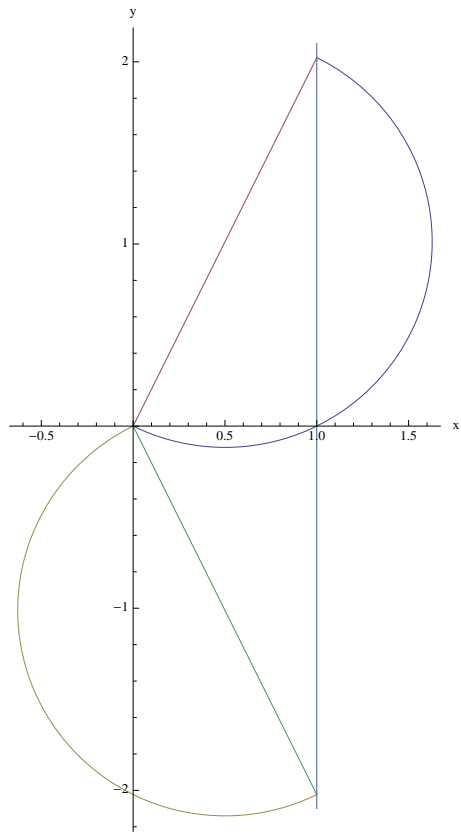}
  \caption[]
   {Case (III.1) is illustrated by drawing both "half-circle" minimizers for
arrival at $\left( x,\cdot ,z\right) $ with $x=1,z=2$.}
\end{figure}

We analyze the Hamiltonian ODEs with boundary data (\ref%
{LA_CaseIII_HBC}). With $x_{0}=y_{0}=0$ and terminal conditions $x_{1}=x$ in
(\ref{LA_xy}) we have%
\begin{eqnarray*}
rx &=&p_{0}\sin r-q_{0}\left( 1-\cos r\right)  \\
&=&p_{0}\sin r-2q_{0}\sin ^{2}\frac{r}{2}.
\end{eqnarray*}%
Recall $q_{1}-q_{0}=\frac{r}{2}\left( x_{1}-x_{0}\right) $, a simple
consequence from the Hamiltonian ODEs. Transversality condition $q_{1}=0$
(and $x_{1}=x,x_{0}=0$) then translates to
$ -2q_{0}=rx$. Plugging this into the previous equation leaves us with
$
rx\left( 1+\cos r\right) =2p_{0}\sin r
$.
On the other hand, $z\neq 0\implies r\neq 0$ (no straight lines!) and we
already pointed out in (\ref{LA_z_xinit0sub_and_unit_t}) that taking into
account $z_{1}=z$ in (\ref{LA_z}), in addtion to $x_{0}=y_{0}=z_{0}=0$,
gives
\begin{equation*}
z=\frac{p_{0}^{2}+q_{0}^{2}}{2}\frac{r-\sin r}{r^{2}}.
\end{equation*}%
Noting that $r-\sin r\neq 0$ when $r\neq 0$, it follows that
\begin{eqnarray*}
\frac{z\sin ^{2}r}{r-\sin r} &=&\frac{4p_{0}^{2}\sin ^{2}r+4q_{0}^{2}\sin
^{2}r}{8r^{2}} \\
&=&x^{2}\frac{\left( 1+\cos r\right) ^{2}+\sin ^{2}r}{8}=x^{2}\frac{1+\cos r%
}{4}.
\end{eqnarray*}%
If $\cos r=-1$ (equivalently: $r=\pi +2\pi \mathbb{Z}$) this equation is
trivially satisfied; indeed both sides are zero since $\sin $ vanishes on $%
\pi +2\pi \mathbb{Z}$. Otherwise, the above equation may be written as%
\begin{equation}
\frac{x^{2}}{z}=\frac{4\sin ^{2}r}{\left( r-\sin r\right) \left( 1+\cos
r\right) }.  \label{LA_CaseIII_ratioxx_z}
\end{equation}%
Focus on $z>0$ in the sequel, the other case being similar. Note that the
right-hand-side above has a removable singularity at $r=\pi $ and takes%
,  as $r\rightarrow \pi $, the value $8/\pi $. In fact,
it is easy to see
 that the graph of the right-hand-side
above as function of $r$ stays strictly below $8/\pi $ as $r\in \left( \pi
,\infty \right) $. The assumption made in the (present) case III.1 is
precisely $x^{2}/z<$ $8/\pi $. In particular then, each solution - there may
be more than one - to (\ref{LA_CaseIII_ratioxx_z}), given $x^{2}/z\in
\lbrack 0,8/\pi )$, will be strictly bigger than $\pi $. We then have found
the following possible values for $r:$%
\begin{equation*}
r\in \left\{ \pm \pi ,\pm 3\pi ,...\right\} \cup \left\{ \text{solutions to (%
\ref{LA_CaseIII_ratioxx_z}), "}>\pi \text{"}\right\} .
\end{equation*}%
We can see that %
those values of $r$ for which $%
\left\vert r\right\vert $ is smallest correspond to the energy minimizing
choice. Hence, in case (III.1), we have $r=\pm \pi $. Accordingly,%
\begin{equation*}
q_{0}=-\mathrm{sgn}\left( z\right) \pi x/2,\,\,\,p_{0}=\pm \sqrt{\frac{\pi }{%
4}\left( 8\left\vert z\right\vert -\pi x^{2}\right) }
\end{equation*}%
and also%
\begin{equation*}
p_{1}=0,\,y_{1}=\pm 2p_{0}/\pi ,r=\pm \pi
\end{equation*}%
which complements our apriori knowledge ($x_{1}=x,q_{1}=0,z_{1}=z$) of the
Hamiltonian ODE solution at unit time. Note that we have \textit{two}
minimizing paths here, with respective arrival points $\left( x,\pm
2p_{0}/\pi ,z\right) \in \left( x,\cdot ,z\right) $. As usual, the energy is%
\begin{equation*}
\frac{1}{2}\left( p_{0}^{2}+q_{0}^{2}\right) =\left\vert z\right\vert \pi.
\end{equation*}%
(The absence of $x$ in the energy is not surprising, since we are
effectively dealing with (two) half-circles, radius (and then length) are
fully determined by the prescribed area $z$.) Note that the energy function
if smooth in a neighbourhood of $z$ since $z\neq 0$. At last, a computation
gives
\begin{equation*}
\frac{\partial \pi \mathrm{H}_{0\leftarrow 1}\left(
x;y_{1},z,p_{1},q_{1},r\right) }{\partial \left( p_{1},y_{1},r\right) }%
=\left(
\begin{array}{ccc}
0 & 0 & -\frac{y_{1}}{2} \\
\frac{2}{r} & 0 & \frac{x}{2} \\
\frac{rx+r\left( x+ry_{1}\right) }{2r^{2}} & -\frac{ry_{1}}{4} & 0%
\end{array}%
\right) ,
\end{equation*}%
the determinant of which equals 
\begin{equation*}
y_{1}^{2}/4=p_{0}^{2}/\pi ^{2}=\frac{1}{4}\left( \frac{8}{\pi }\left\vert
z\right\vert -x^{2}\right) >0
\end{equation*}%
In particular, $\mathrm{x}_{0}=0$ is non-focal for $\left( x,\cdot ,z\right)
$ along either of the two minimizers. As a consequence, theorem \ref{thm:MainThm} gives
\begin{equation*}
f^{\varepsilon }\left( x,z;T\right) |_{T=1}\sim e^{- \left\vert z\right\vert \pi / \varepsilon^2
}\varepsilon ^{-2} c_0
\end{equation*}%
in agreement with the corresponding expansion given in \cite[Sec. 7, case
(III.1)]{TW}.

\textbf{Case (III.2)} Assume $%
\left\vert x\right\vert >\sqrt{8\left\vert z\right\vert /\pi }$ ($\implies
x\neq 0$) and also $z\neq 0$; see figure 2.
\begin{figure}[hb]
  \centering
  \includegraphics[scale=0.7]{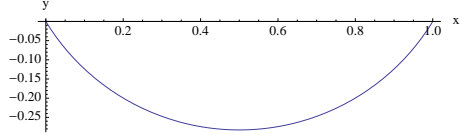} 
  \caption[]
   {Case (III.2) is illustrated by drawing the unique minimizing arc for
arrival at $\left( x,\cdot ,z\right) $ with $x=1,z=1/5$.}
\end{figure}
 (The case $z=0$ is simpler and discussed
separately below). We proceed exactly as above, but now there is a (unique)
solution $r_{0}:=r\in (-\pi ,\pi )\backslash \left\{ 0\right\} $ to (\ref%
{LA_CaseIII_ratioxx_z}), which corresponds to the energy correspond to the
energy minimizing choice. We then find
\begin{equation*}
p_{1}=x\frac{r}{2}\cot \frac{r}{2},\,y_{1}=0,r=r_{0}
\end{equation*}%
which complements our apriori knowledge ($x_{1}=x,q_{1}=0,z_{1}=z$) of the
Hamiltonian ODE solution at unit time. (The results corresponds precisely
to the point-point problem discussed in
case (I.2) with arrival point $\left( x,0,z\right) $.) A computation then
shows that%
\begin{equation*}
\frac{\partial \pi \mathrm{H}_{0\leftarrow 1}\left(
x;y_{1},z,p_{1},q_{1},r\right) }{\partial \left( p_{1},y_{1},r\right) }%
|=\left(
\begin{array}{ccc}
-\frac{\sin r}{r} & \frac{\sin r}{r} & \frac{x\cot \left( \frac{r}{2}\right)
\left( -r+\sin r\right) }{2r} \\
\frac{1-\cos r}{r} & \cos ^{2}\frac{r}{2} & \frac{x\left( r-\sin r\right) }{%
2r} \\
\frac{x}{r}-\frac{1}{2}x\cot \frac{r}{2} & \frac{1}{4}rx\cot \frac{r}{2} &
\frac{x^{2}\left( -2+r\cot \frac{r}{2}\right) \cot \frac{r}{2}}{4r}%
\end{array}%
\right) ,
\end{equation*}%
the determinant of which simplifies to
\begin{equation*}
\frac{x^{2}\left( -2+r\cot \frac{r}{2}\right) \cot \frac{r}{2}}{r}\text{. }
\end{equation*}%
Since $x\neq 0$, it suffices to remark that the remaining factor, as
function of $r$ only, does not take the value zero for $r\in \left( -\pi
,\pi \right) \backslash \left\{ 0\right\} $. In fact, as is easy to see, %
the determinant has a removable
singularity at zero remains strictly negative on the entire open interval $%
\left( -\pi ,\pi \right) $. We thus established non-focality (note however, that the determinante \textit{does
}vanish in the limit $\left\vert r\right\vert \uparrow \pi $; this is the
focal case (III.4) discussed below.) Theorem \ref{thm:MainThm} then gives
\begin{equation*}
f^{\varepsilon }\left( x,z;T\right) |_{T=1}\sim e^{-c_1 / \varepsilon^2} \varepsilon ^{-2} c_0
\end{equation*}%
where $c_1$ is determined exactly as in case (I.2), equation (\ref{TWI2density}), just with $y=y_1=0$.

\textbf{Case (III.3)}\ Assume $\left\vert x\right\vert >\sqrt{8\left\vert
z\right\vert /\pi }=0$ (i.e. $x\neq 0,\,z=0$) One finds without trouble $%
r=0,\,p_{0}=x,q_{0}=0$ and then $p_{1}=x,y_{1}=0$; from transversality of
course $q_{1}=0$. The (unique) minimizing path is then given by $\left(
x_{t},y_{t}\right) =\left( tx,0\right) $, the energy is equal to $\left\vert
x\right\vert $. Non-focality can be checked e.g. by recycling the expression of
case (III.2) in the limit $r=0$.\newline

\textbf{Case (III.4)}  %
%
%
%
%
All computations from either case (III.1) remain valid. We have $y_{1}=0$ (so
that there is a unique minimizer) and the determinant, which was seen to be $%
\propto y_{1}^{2}$ is now equal to zero. By definition, $\mathrm{x}_{0}=0$
is then focal for $\left( x,\cdot ,z\right) $ along the (now: unique)
minimizer. Equivalently, we may approach this from case (III.2), by taking
the limit $\left\vert r\right\vert \uparrow \pi $ (geometrically this
amounts to have more and more curved arcs from $\left( 0,0\right) $ to $%
\left( x,0\right) $ until we arrive at the half-circle solution from).
Either way, being in a focal situation, we cannot apply theorem \ref{thm:MainThm} and indeed in
\cite{TW} a density expansion is given with algebraic factor $\varepsilon
^{-5/2}$, in contrast to the generic prediction $%
\varepsilon ^{-l},\,l=2$ of our theorem.

\subsubsection{Starting point with O($\varepsilon)$-dependence; appearance
of $e^{c_2/{\varepsilon}}$-factor}

Let us briefly illustrate how our methods allow to go beyond the results
of Takanobu--Watanabe, which are - in the non-degenerate case - marginal density expansions
based on (\ref{EBMeps}), started at the origin and run til unit time ($T=1$), of the form
\[
\bar{f}^\varepsilon \left( \mathrm{a} \right)| \sim e^{-c_{1}/\varepsilon
^{2}}\varepsilon ^{-l} c_0 \text{ with }c_{1}=\Lambda \left( \mathrm{a}\right) \text{ and }c_{0}>0.
\]%
For instance, in the case (III) above, we have  $\mathrm{a}=\left( x,z\right) \in \mathbb{R}^{l}$
for $l=2$.  To this end, we again consider marginal density 
expansions based on (\ref{EBMeps}), but now started order $\varepsilon$ away from the origin.
\textit{For simplicity only}, we shall consider the subcase (III.1), the energy
in this case was computed to be $\Lambda \left( \mathrm{a}\right) =\Lambda \left( x,z\right) =\pi |z|$,
and take the starting point 
\[
X_{0}=0,\,Y_{0}=\varepsilon \hat{y}_{0},\,Z_{0}=\varepsilon \hat{z}_{0}.
\]%
According to our general theory, as laid out in
theorem \ref{thm:MainThm}, perturbation of the starting point to first order in $\varepsilon 
$ will lead to appearance of second order exponential terms in the density (at time $T=1$),%
\[
f^\varepsilon\left( \mathrm{a}\right) \sim  e^{-c_{1}/\varepsilon
^{2}}e^{c_{2}/\varepsilon }\varepsilon ^{-l} c_0.
\]%
We now compute $c_2$. With $^{\prime }$ for the derivative with respect to $\mathrm{a}=\left(
x,z\right) $ as usual, we have, assuming $z>0$ w.l.o.g.,%
\[
\Lambda ^{\prime }\left( \mathrm{a}\right) =\left( \partial _{x}\Lambda
\left( x,z\right) ,\partial _{z}\Lambda \left( x,z\right) \right) =\left(
0,\pi \right) .
\]%
On the other hand, we need to compute $\mathrm{\hat{Y}}_{T}=\left( \hat{X}%
_{T},\hat{Z}_{T}\right) $, along (apriori each of the two) minimizing
controls $\mathrm{h}=\left( h^{1},h^{2}\right) $, based on the following
auxilary ODE,  
\begin{eqnarray*}
d\hat{X} &=&0,d\hat{Y}=0,d\hat{Z}=-\frac{1}{2}\hat{Y}dh^{1}+\frac{1}{2}\hat{X%
}dh^{2}, \\
\hat{X}_{0} &=&0,\,\hat{Y}_{0}=\hat{y}_{0},\,\hat{Z}_{0}=\hat{z}_{0}.
\end{eqnarray*}%
Noting that $\mathrm{h}_{0}=\left( 0,0\right) ,\,\mathrm{h}_{T}\in \left(
x,\cdot \right) $, and thanks to $\hat{X}\equiv 0$, the computation is
identical for both controls, $\hat{Z}_{T}=\hat{z}_{0}-\hat{y}_{0}x/2$. And
it follows that $c_{2}=\Lambda ^{\prime }\left( \mathrm{a}\right) \cdot 
\mathrm{\hat{Y}}_{T}=\pi \left( \hat{z}_{0}-\hat{y}_{0}x/2\right) $. We thus
proved

\begin{proposition}
Let $\left( X^{\varepsilon },Y^{\varepsilon },Z^{\varepsilon }\right) $ be $%
\varepsilon $-dilated Brownian motion on the Heisenberg group, started at $%
\left( 0,\varepsilon \hat{y}_{0,}\varepsilon \hat{z}_{0}\right) $. Then $%
\left( X_{1}^{\varepsilon },Z_{1}^{\varepsilon }\right) $ admits a density
which in the case (III.1), say when $z>0$ and $\left\vert x\right\vert <$ $%
\sqrt{8z\pi }$, has an expansion as $\varepsilon \to 0$, for some $c_0 >0$,  of the form%
\begin{equation}
f_{\varepsilon }\left( x,z; T\right) |_{T=1} \sim e^{-\pi z/\varepsilon ^{2}}e^{\pi
\left( \hat{z}_{0}-\hat{y}_{0}x/2\right) /\varepsilon }\varepsilon ^{-2}c_{0}.
\label{ExpXZwithyz0hat}
\end{equation}
\end{proposition}

Note that, upon taking $\hat{y}_{0}=\hat{z}_{0}=0$, we recover the previous
expansion of case (III.1),%
\[
\bar{f}^{\varepsilon }\left( x,z\right) \sim e^{-\pi z/\varepsilon
^{2}}\varepsilon ^{-2}c_{0}.
\]%
Other cases than (III.1), and also $\hat{X}_{0}=\hat{x}_{0}\neq 0$, are
treated similarly but the computations are more involved. Let us, instead,
verify (\ref{ExpXZwithyz0hat}) by a reduction to the zero starting point
case, using the Heisenberg group structure. Namely,%
\[
\left( \bar{X}^{\varepsilon },\bar{Y}^{\varepsilon },\bar{Z}^{\varepsilon
}\right) :=\left( 0,-\varepsilon \hat{y}_{0},-\varepsilon \hat{z}_{0}\right)
\ast \left( X^{\varepsilon },Y^{\varepsilon },Z^{\varepsilon }\right) 
\]
satisfies the same stochastic differential equations, but now started at the
origin. In particular, $$\left( \bar{X}_{T}^{\varepsilon },\bar{Z}%
_{T}^{\varepsilon }\right) =\left( X_{T}^{\varepsilon },Z_{T}^{\varepsilon
}-\varepsilon \left( \hat{z}_{0}+\hat{y}_{0}X_{T}^{\varepsilon }/2\right)
\right) $$ and so \footnote{%
Strictly speaking, this argument requires to check local uniformity of the
expansions with respect to $\mathrm{a}=\left( x,z\right) .$}%
\begin{eqnarray*}
f^{\varepsilon }\left( x,z\right)  &=&\bar{f}^{\varepsilon }\left(
x,z-\varepsilon \left( \hat{z}_{0}+\hat{y}_{0}x/2\right) \right)  \\
&\sim &e^{-\pi \left( z-\varepsilon \left( \hat{z}_{0}+\hat{y}_{0}x/2\right)
\right) /\varepsilon ^{2}}\varepsilon ^{-2}c_{0} \\
&\sim &e^{-\pi z}e^{\pi \left( \hat{z}_{0}-\hat{y}_{0}x/2\right)
/\varepsilon }\varepsilon ^{-2}c_{0}.
\end{eqnarray*}%
in agreement with the above proposition (the constant $c_{0}$ was allowed to
change here). It is not hard to devise variations on the theme (in
particular, upon inclusion of drift vector fields of order $\varepsilon $)
where theorem \ref{thm:MainThm} still applies but the above reasoning based on the (rigid)
Heisenberg group structure fails.

\end{document}